\documentclass[11pt]{article}

\usepackage[english]{babel}
\usepackage[T1]{fontenc}

\usepackage{amsfonts}
\usepackage{amsmath}
\usepackage{amssymb}
\usepackage{amsthm}
\usepackage{mathtools}
\usepackage{stmaryrd}
\usepackage{dsfont}
\usepackage{graphicx}
\usepackage{hyperref}
\usepackage{titlesec}
\usepackage{mathabx}
\usepackage{enumitem,tocloft}
\usepackage[top=2.5cm, bottom=2.5cm, left=3cm, right=3cm]{geometry}
\usepackage{xcolor}
\usepackage{fourier-orns}
\usepackage[utf8]{inputenc}
\usepackage{bm}
\usepackage{bbm}
\usepackage{framed}
\usepackage{supertabular}
\hypersetup{linktocpage,breaklinks=true}
\usepackage{mathrsfs}
\usepackage{csquotes}
\usepackage{bbold}

\SetSymbolFont{stmry}{bold}{U}{stmry}{m}{n}

\usepackage[backend=bibtex, style=alphabetic, maxbibnames=6, sorting=nyt]{biblatex}

\theoremstyle{plain}
\newtheorem{theorem}{Theorem}[section]
\newtheorem{corollary}[theorem]{Corollary}

\newtheorem{theoremletter}{Theorem}

\newtheorem{propositionletter}[theoremletter]{Proposition}

\newtheorem{lemma}[theorem]{Lemma}
\newtheorem{proposition}[theorem]{Proposition}

\theoremstyle{definition}
\newtheorem{definition}[theorem]{Definition}
\newtheorem{example}[theorem]{Example}
\newtheorem{question}[theorem]{Question}
\newtheorem{remark}[theorem]{Remark}

\numberwithin{equation}{section}

\newcommand{\G}{\Gamma}

\newcommand{\g}{\gamma}

\newcommand{\La}{\Lambda}

\newcommand{\field}{\mathbb{f}}

\newcommand{\R}{\mathbb{R}}

\newcommand{\Z}{\mathbb{Z}}	

\newcommand{\N}{\mathbb{N}}	

\newcommand{\fsym}[1]{\mathsf{FSym}(#1)}

\newcommand{\sym}[1]{\mathsf{Sym}(#1)}

\newcommand{\shuf}[1]{\mathsf{Shuffler}(#1)}

\newcommand{\shufn}[2]{\mathsf{Shuffler}^{\circ {#1}}(#2)}

\newcommand{\juggler}[2]{\mathsf{Shuffler}_{#1}(#2)}

\newcommand{\jugglern}[3]{\mathsf{Shuffler}_{#2}^{\circ {#1}}(#3)}

\newcommand{\cloner}[1]{\mathsf{Cloner}_{\field}(#1)}

\newcommand{\clonern}[2]{\mathsf{Cloner}_{\field}^{\circ {#1}}(#2)}

\newcommand{\designer}[1]{\mathsf{Designer}_{F}(#1)}

\newcommand{\act}{\curvearrowright}

\newcommand{\ld}{\mathrm{L}}

\newcommand{\supp}{\mathrm{supp}\ } 

\newcommand{\diam}{\mathrm{diam}\ } 

\newcommand{\prof}[1]{j_{1,#1}}

\newcommand{\halo}{\mathscr{L}}

\makeatletter
\newcommand*{\defeq}{\mathrel{\rlap{%
                     \raisebox{0.3ex}{$\m@th\cdot$}}%
                     \raisebox{-0.3ex}{$\m@th\cdot$}}%
                     =}
\makeatother

\begin{document}

\begin{titlepage}
\setcounter{page}{1}
\title{On quantitative orbit equivalence for lamplighter-like groups}
\author{\small{Corentin Correia and Vincent Dumoncel}}

\date{\today}
\maketitle

\begin{abstract}
We focus on halo products, a class of groups introduced by Genevois and Tessera, and whose geometry mimics lamplighters. Famous examples are lampshufflers. Motivated by their work on the classifications up to quasi-isometry of these groups, we initiate a more quantitative study of their geometry. Indeed, it follows from the work of Delabie, Koivisto, Le Maître and Tessera that quantitative orbit equivalence between amenable groups is closely related to their large scale geometry, such a connection being justified by the use, in their main results, of a well-known quasi-isometry invariant: the isoperimetric profile.\par
Inspired by their work on quantitative orbit equivalence between lamplighters, we prove a stability result for orbit equivalence of permutational halo products, going beyond the framework of standard halo products, using a new notion of orbit equivalence of pairs. Combined with our asymptotics of isoperimetric profiles obtained in an earlier article, we prove that most of these constructions are quantitatively optimal. For instance, we show that $\shuf{\Z^{k+\ell}}$ and $\shuf{\Z^{k}}$ are $\ld^p$ orbit equivalent if and only if $p<\frac{k}{k+\ell}$, thus quantifying how much the geometries of these non-quasi-isometric groups differ. We finally build orbit equivalence couplings using the notion of F\o lner tiling sequences.
\end{abstract}

\smallskip

	
	\smallskip
	
	{
		\small	
		\noindent\textbf{{MSC-classification:}}	
		Primary 37A20; Secondary 20F65, 20F69.
	}


\tableofcontents

\section{Introduction}

In this paper, we focus on the following class of groups, introduced in~\cite{GT24a}:

\begin{definition}
Let $X$ be a set. A~\textit{halo of groups $\halo$ over $X$} is the data, for any subset $S\subset X$, of a group $L(S)$ such that:
\begin{itemize}
    \item for all $R,S\subset X$, if $R\subset S$ then $L(R)\leqslant L(S)$;
    \item $L(\emptyset)=\lbrace 1\rbrace$ and $L(X)=\langle L(S) : S\subset X \;\text{finite}\rangle$;
    \item for all $R,S\subset X$, $L(R\cap S)=L(R)\cap L(S)$.
\end{itemize}
\end{definition}

Given an action $H\curvearrowright X$ and a morphism $\alpha\colon H \to \text{Aut}(L(X))$ satisfying $\alpha(h)(L(S))=L(hS)$ for any $S\subset X$ and $h\in H$, the~\textit{permutational halo product} $\halo_{X,\alpha}H$ is the semi-direct product 
\begin{equation*}
    \halo_{X,\alpha}H \defeq L(X)\rtimes_{\alpha}H.
\end{equation*}

Throughout the paper, we will consider permutational halo products where $H$ acts on a quotient $X\defeq H/M$ by left-multiplication, where $M$ is a normal subgroup of $H$. Moreover,~\textit{standard} halo products will refer to the case where $X=H$ with the left-multiplication action $\alpha\colon H\curvearrowright H$, and $\halo_{H,\alpha} H$ is simply denoted by $\halo H$.

\smallskip

Halo products have recently been an interesting class to study up to quasi-isometry and bi-Lipschitz equivalence, as summarized in the two following paragraphs.

\paragraph{Rigidity for standard halo products.} 

\sloppy Examples of halo products include for instance wreath products, but also~\textit{lampshufflers}, defined as 
\begin{equation*}
    \shuf{H}=\fsym{H}\rtimes H
\end{equation*}
and for which $L(S)=\fsym{S}$ is, for any $S\subset H$, the group of finitely supported permutations $S\rightarrow S$. These groups already appeared several times in the literature, in relations with many topics of interest in group theory, see for instance~\cite{Yad09, HO16, BZ19, EZ21, SCZ21, GT24a, Sil24}.

\smallskip

As another instance of this construction, we can also mention~\textit{lampcloners}, of the form
\begin{equation*}
    \cloner{H}=\text{FGL}(H)\rtimes H
\end{equation*}
with a field $\field$, where $\text{FGL}(H)$ is the group of linear automorphisms of the $\field$-vector space $V_{H}$ freely generated by the elements of $H$ fixing all but finitely many elements of the formal basis $\lbrace e_{h}:h\in H\rbrace$ and where, for each $S\subset H$, $L(S)=\text{FGL}(S)$ is the subgroup of $\text{FGL}(H)$ of linear automorphisms that fix $H\setminus S$ and stabilize the subspace $\langle S\rangle$. We refer the reader to Section~\ref{sec:halo} for more examples. 

\smallskip

The motivation of~\cite{GT24a} to introduce such a general framework is that the semi-direct product structure provides a foliation of these spaces that must be, if $H$ satisfies additional mild assumptions, \enquote{quasi-preserved} by quasi-isometries, allowing the authors to show strong rigidity phenomena for quasi-isometries between such spaces, and thus extending the classification already obtained in~\cite{GT24b}. For instance, for lampshufflers, they managed to prove the following. 

\begin{theorem}[{\cite[Corollary~8.9]{GT24a}}]\label{th:QIrigidityShufGT24}
Let $H$ and $K$ be finitely presented one-ended groups. Then $\shuf{H}$ and $\shuf{K}$ are quasi-isometric if and only if $H$ and $K$ are biLipschitz equivalent.
\end{theorem}

The main tool used in~\cite{GT24a} to establish such results is a new quasi-isometry invariant referred to as the~\textit{thick bigon property}, also introduced in~\cite{GT24a}. Finitely presented one-ended groups are particular examples of groups having this property~\cite[Lemma~3.4]{GT24a}. 

\smallskip

There are also quasi-isometric classification results in some classes where the thick bigon property is unavailable.

\begin{theorem}[{\cite[Theorem~H]{cordum25}}]\label{thm:QIclassificationiteratedlampshufflers}
Let $d_1,d_2,m,n\ge 0$ be positive integers. Let $H$ and $K$ be virtually abelian finitely generated groups with growth degrees $d_1$ and $d_2$ respectively. Then $\shufn{m}{H}$ and $\shufn{n}{K}$ are quasi-isometric if and only if $m=n$ and $d_1=d_2$.
\end{theorem}

Here, $\shufn{n}{H}$ denotes the $n$-th fold iterated lampshuffler over $H$, defined inductively by $\shufn{0}{H} \defeq H$ and $\shufn{(n+1)}{H} \defeq \shuf{\shufn{n}{H}}$. 

\paragraph{Rigidity for permutational halo products.} For permutational halo products, such as permutational lamplighters $F\wr_{H/M}H$, permutational lampshufflers $\shuf{H, M} \defeq \fsym{H/M}\rtimes H$
or permutational lampcloners $\cloner{H,M}\defeq \text{FGL}(H/M)\rtimes H$, with $M\lhd H$, the literature about quasi-isometric rigidity is less developed. Until now, only the case of permutational lamplighters have been investigated~\cite{Dum24}. As an illustration:

\begin{theorem}[{\cite[Corollary~1.4]{Dum24}}]\label{thm:QIrigidityofPWP}
Let $d_{1},d_{2},k_{1},k_{2}\ge 2$ be integers such that $d_{1}\ge k_{1}\ge 2$, $d_{2}\ge k_{2}\ge 2$. Let $E$, $F$ be non-trivial finite groups. Then $E\wr_{\Z^{k_{1}}}\Z^{d_{1}}$ and $F\wr_{\Z^{k_{2}}}\Z^{d_{2}}$ are quasi-isometric if and only if $d_{1}=d_{2}$, $k_{1}=k_{2}$ and $|E|$ and $|F|$ are powers of a common number. 
\end{theorem}

We refer to~\cite[Theorem~1.3]{Dum24} for a more general statement and the corresponding assumptions.

\paragraph{The main question of the paper.}

The starting point of our work is therefore the following, rather \enquote{vague} question.

\begin{question}\label{question:startingquestion}
In the case where two halo products $\halo H$ and $\halo K$ are not quasi-isometric, how much do their geometries differ?
\end{question}

To tackle the question, quantitative orbit equivalence provides a nice framework and allows for a~\textit{measured} comparison between the geometries of the two groups. For instance, Theorem~\ref{thm:optimalityofL^pnessiteratedlampshufflersINTRO} and~\ref{thm:optimalityiteratedlampshufflersINTRO} are more quantitative versions of Theorem~\ref{thm:QIclassificationiteratedlampshufflers}, whereas Theorem~\ref{th:ClassificationQOEPermutationalZd Intro} is a more quantitative version of Theorem~\ref{thm:QIrigidityofPWP}.

\smallskip 

Our statements will focus on the particular examples of (permutational) wreath products, lampjugglers (halo products that extend lampshufflers) and lampcloners. However, our methods for each of these halo products do not seem to be independant of each other. We believe that there is a general underlying pattern that could lead the interested reader to a statement encompassing all particular examples of halo products. We sketch the main steps of the strategies in Sections~\ref{sec:generalmethodstability} and~\ref{sec:generalobservationtilinghalo}. 

\paragraph{Quantitative orbit equivalence.} Measure equivalence has been introduced by Gromov~\cite{Gro93} as a measured analogue of quasi-isometry, and in fact, it follows from the work of Shalom~\cite{Sha04} that quasi-isometric amenable groups are~\textit{$\ld^{\infty}$ measure equivalent}, namely measure equivalent with some strong quantitative assumption on functions called the~\textit{cocycles}.

\smallskip

In the paper, we will not give the definitions of quantitative measure equivalence since we focus on a particular instance, called~\textit{orbit equivalence}, and the historical results will also be stated only in terms of this particular notion.

\begin{definition}
Let $H$ and $K$ be countable groups. An~\textit{orbit equivalence} coupling for $H$ and $K$ is a standard probability space $(X,\mu)$ equipped with essentially free pmp $H$- and $K$-actions that share the same orbits up to a null set, namely for almost every $x\in X$, the $H\cdot x$ and $K\cdot x$ are equal.
\end{definition}

If such an orbit equivalence coupling exists, we say that the groups $H$ and $K$ are~\textit{orbit equivalent}. An orbit equivalence coupling $(X,\mu)$ between two groups $H,K$ provides measurable maps $c_{H,K}\colon H\times X\to K$ and $c_{K,H}\colon K\times X\to H$, called the~\textit{cocycles}, and defined by the equalities
\begin{equation*}
    h\cdot x=c_{H,K}(h,x)\cdot x \;\text{and}\;k\cdot x=c_{K,H}(k,x)\cdot x
\end{equation*}    
for every $h\in H$, $k\in K$ and for almost every $x\in X$.

\smallskip

Let us assume that $H$ and $K$ are finitely generated. Given a non-decreasing map $\varphi\colon\R_{+}\to\R_{+}$, a cocycle $c_{H,K}$ is $\varphi$\textit{-integrable} if for every $h\in H$, there exists a constant $C_{h}>0$ such that
\begin{equation*}
\int_{X}{\varphi\left(\frac{\left|c_{H,K}(h,x)\right|_{S_{K}}}{C_{h}}\right)\mathrm{d}\mu(x)}<\infty,
\end{equation*}
where $S_K$ is a finite generating set of $K$. For instance, given a real number  $p>0$ (for instance $p<1$ is allowed), the cocycle is $\ld^p$ if it satisfies the previous condition with $\varphi(x)=x^p$. Moreover, by an $\ld^{\infty}$ cocycle $c_{H,K}$, we will mean that for every $h\in H$, $c_{H,K}(h,\cdot)\colon X\to K$ essentially takes finitely many values, and this assumption does not ask for finitely generated groups.

\smallskip

Then, we say that an orbit equivalence coupling from $H$ to $K$ is~\textit{$(\varphi,\psi)$-integrable} if the corresponding cocycles $c_{H,K}$ and $c_{K,H}$ are respectively $\varphi$- and $\psi$-integrable. Moreover, we may write for instance $(\ld^p,\psi)$-integrable for any $p\in [0,\infty]$, $(\ld^{<\infty},\psi)$-integrable when the corresponding cocycle $c_{H,K}$ is $\ld^p$ for every $p>0$, or $(\varphi,\ld^0)$-integrable when no requirement is made on the second cocycle $c_{K,H}$, etc.

\smallskip

Note that measure equivalence also provides cocycles and we can similarly ask for quantitative constraints on it. Moreover two groups are $(\varphi,\psi)$-integrably orbit equivalent if and only if they are $(\varphi,\psi)$-integrably measure equivalent and the measure equivalence coupling admits fundamental domains of equal measure.

\smallskip

In contrast to Ornstein-Weiss theorem which states that any countable infinite amenable groups are orbit equivalent, a natural question is the preservation of geometric properties of groups under quantitative forms of orbit or measure equivalence. Historically, the most natural quantitative forms of orbit equivalence (and even measure equivalence) are the $\ld^{\infty}$ and $\ld^1$ cases. Many rigidity results have been uncovered in this context, showing that $\ld^1$ and $\ld^{\infty}$ orbit equivalence (as well as $\ld^1$ and $\ld^{\infty}$ measure equivalence) captures the geometry of groups. Here is the famous rigidity result due to Shalom (see also~\cite{BFS13} and~\cite{Aus16} for some results in the $\ld^1$ case).

\begin{theorem}[\cite{Sha04}]\label{thm:ShalomBilip intro}
Two amenable groups are bi-Lipschitz equivalent if and only if they are $\ld^{\infty}$ orbit equivalent.
\end{theorem}

Now for non bi-Lipschitz equivalent groups, our interest specifically focuses on the following result of Delabie, Koivisto, Le Maître and Tessera, providing explicit upper bounds on how integrable the cocycles of an orbit equivalence coupling can be.

\begin{theorem}[{\cite[Theorem~1.1]{DKLMT22}}]\label{thm:dklmt intro}
Let $\varphi\colon\R_{+}\rightarrow\R_{+}$ be a non-decreasing function such that $t\mapsto \frac{t}{\varphi(t)}$ is non-decreasing. Let $H$ and $K$ be finitely generated groups. Assume that there exists a $(\varphi,\ld^0)$-integrable orbit equivalence coupling from $H$ to $K$. Then their isoperimetric profiles satisfy the asymptotic inequality
\begin{equation*}
    \varphi\circ\prof{K}(n) \preccurlyeq \prof{H}(n).
\end{equation*}
\end{theorem}

We refer the reader to Section~\ref{sec:isoprof} for the definition of the isoperimetric profile. For now, the only idea to keep in mind is the fact that this map is only interesting when the group is amenable, and in fact it measures how much amenable the group is: the faster the map goes to infinity, the more the group is amenable. Therefore, in contrast to Ornstein-Weiss, Theorem~\ref{thm:dklmt intro} demonstrates that these quantitative strengthenings of orbit equivalence enable us to capture the geometries of amenable groups.

\smallskip

A consequence of this result is that, given positive integers $d_1>d_2$, there is no $\ld^p$ orbit equivalence between $\Z^{d_1}$ and $\Z^{d_2}$ if $p>\frac{d_2}{d_1}$. Moreover explicit constructions, using the notion of \textit{F\o lner tiling sequences} (see Section~\ref{sec:folnerTiling} for precise definitions), provide an $\ld^p$ orbit equivalence for every $p<\frac{d_2}{d_1}$. Finally the first-named author recently proved that the threshold $p=\frac{d_2}{d_1}$ cannot be achieved~\cite[Theorem~A]{Corr24}. We refer the reader to Section~\ref{sec:orbit equivalence} for more details.

\smallskip

Quantitative orbit and measure equivalences are therefore closely related to geometric group theory. In some sense, the maximal quantification that we can put on the cocycles tells us how much the groups are geometrically distinct.

\paragraph{Constructions of orbit equivalence couplings.}

Coming back to halo products, Delabie, Koivisto, Le Maître and Tessera have already established quantitative orbit equivalence couplings between wreath products~\cite{DKLMT22}. They used two techniques:
\begin{enumerate}[label=(\Alph*)]
    \item\label{item:techniqueDKLMT A} the first one consists in finding stability results, for instance if $H$ and $K$ are $(\varphi,\psi)$-integrably orbit equivalent, then so are $\G\wr H$ and $\G\wr K$ for every group $\G$;
    \item\label{item:techniqueDKLMT B} the second explicit construction of coupling use the notion of F\o lner tiling sequences, that yields natural group actions generalizing $\Z$-odometers.
\end{enumerate}
The goal is to extend techniques~\ref{item:techniqueDKLMT A} and~\ref{item:techniqueDKLMT B} to other examples of halo products, in each specific case: lampshufflers (more generally lampjugglers) and lampcloners. Moreover, regarding technique~\ref{item:techniqueDKLMT A}, we go in fact even further and investigate the case of permutational halo products. For instance, for permutational wreath products, the strategy is to derive an orbit equivalence coupling between $\G\wr_{H/M}H$ and $\G\wr_{K/N}K$, provided a coupling between $H$ and $K$ that takes into account the subgroups $M$ and $N$. This new notion, introduced in a common work with Romain Tessera~\cite{CDT26}, is called~\textit{orbit equivalence of pairs}.

\begin{definition}[{\cite{CDT26}}]
Let $H$ and $K$ be countable groups, with subgroups $M\leq H$, $N\leq K$. We say that the pairs $(H,M)$ and $(K,N)$ are~\textit{orbit equivalent} if there exist essentially free pmp $H$- and $K$-actions on a standard and atomless probability space $(X,\mu)$, satisfying $H\cdot x=K\cdot x$ and $M\cdot x=N\cdot x$ for almost every $x\in X$. The space $(X,\mu)$ is called an~\textit{orbit equivalence coupling} between the pairs $(H,M)$ and $(K,N)$.
\end{definition}

An orbit equivalence coupling $(X,\mu)$ between two pairs $(H,M)$ and $(K,N)$
is $(\varphi, \psi)$-integrable if it is $(\varphi,\psi)$-integrable as a coupling between the ambient groups $H$ and $K$. In the paper, we will always be interested in normal pairs, namely the subgroups $M$ and $N$ are normal in $H$ and $K$ respectively.

In a more quantitative manner, a $(\varphi,\psi)$-orbit equivalence between $\G\wr_{H/M}H$ and $\G\wr_{K/N}K$ will be given by a $(\varphi,\psi)$-integrable coupling of pairs between the base groups $H$ and $K$. Examples of couplings of pairs, with quantitative information, are provided in~\cite{CDT26}, and they turn out to be sharp. Here is an illustrative example for pairs of the form $(\Z^d,\Z^k)$ with $d>k>0$, where $\Z^k$ is understood as the subgroup $\{0\}^{d-k}\times\Z^k$ of $\Z^d$.

\begin{proposition}[{\cite{CDT26}}]
Let $d_{1},d_{2},k_{1},k_{2}$ be positive integers satisfying $0<k_{1}<d_{1}$ and $0<k_{2}<d_{2}$. Then, denoting $P=\min{\left (\frac{k_{2}}{k_{1}},\frac{d_{2}-k_{2}}{d_{1}-k_{1}}\right)}$ and $Q=\min{\left (\frac{k_{1}}{k_{2}},\frac{d_{1}-k_{1}}{d_{2}-k_{2}}\right)}$, there exists an $(\ld^{<P},\ld^{<Q})$-integrable orbit equivalence coupling from $(\Z^{d_{1}},\Z^{k_1})$ to $(\Z^{d_{2}},\Z^{k_2})$.
\end{proposition}

\paragraph{Couplings between halo products: stability results.}The first contribution of our article is the extension to other examples of halo products, even permutational ones, of the stability results established in~\cite{DKLMT22} for wreath products, as mentioned in the technique~\ref{item:techniqueDKLMT A} above. For the case of permutational non-standard halo products, the main tool we require is the notion of~\textit{orbit equivalence of pairs}.

\begin{theoremletter}[see Theorems~\ref{thm:stabilitypropertyforOEbetweenPWP},~\ref{thm:stabilityofcouplings+quantificationLampjugglers} and~\ref{thm:stabilityofcouplings+quantificationLampcloners}]\label{thm:stabilityofcouplings+quantification intro}
Let $H$ and $K$ be two finitely generated groups, with normal subgroups $M\lhd H$ and $N\lhd K$. Let $\varphi,\psi\colon\R_{+}\rightarrow\R_{+}$ be non-decreasing maps. Assume that there exists a $(\varphi,\psi)$-integrable orbit equivalence coupling from $(H,M)$ to $(K,N)$. Then:
\begin{itemize}
    \item given two finitely generated groups $\G$ and $\La$, if there exists a $(\varphi,\psi)$-integrable orbit equivalence coupling from $\G$ to $\La$, then there exists a $(\varphi,\psi)$-integrable orbit equivalence coupling from $\G\wr_{H/M}H$ to $\La\wr_{K/N}K$;
    \item there exists a $(\varphi,\psi)$-integrable orbit equivalence coupling from $\shuf{H,M}$ and $\shuf{K,N}$;
    \item given a finite field $\field$, there exists a $(\varphi,\psi)$-integrable orbit equivalence coupling from $\cloner{H,M}$ to $\cloner{K,N}$. 
\end{itemize}
Finally, for non necessarily finitely generated groups, the results hold true when removing the integrability conditions (namely we only focus on the existence of orbit equivalence couplings), or when replacing \enquote{$(\varphi,\psi)$-integrable} by \enquote{$\ld^{\infty}$}.
\end{theoremletter}

As a consequence of Theorem~\ref{thm:stabilityofcouplings+quantification intro}, given $d_1>d_2$, there exists an $\ld^p$ orbit equivalence between $\shuf{\Z^{d_1}}$ and $\shuf{\Z^{d_2}}$ for every $p<\frac{d_2}{d_1}$  (Theorem~\ref{thm:optimalityofL^pnessiteratedlampshufflersINTRO} will assert that this is sharp). By induction, the same holds between $\shufn{n}{\Z^{d_1}}$ and $\shufn{n}{\Z^{d_2}}$, for any $n\ge 1$.

\begin{remark}
Another consequence of Theorem~\ref{thm:stabilityofcouplings+quantification intro} is the stability of bi-Lipschitz equivalence, and the fact that it enables us to recover slightly weaker versions of existing results (see Corollary~\ref{cor:PreservationBiLipHalo} and Remark~\ref{rem:PreservationBiLipHalo}).
\end{remark}

\paragraph{Couplings between halo products using F\o lner tiling sequences.}

Technique~\ref{item:techniqueDKLMT B} that we mentioned earlier relies on~\textit{F\o lner tiling sequences}. In~\cite{DKLMT22}, it was introduced to construct couplings between iterated lamplighters but with different numbers of iterations. The first needed observation towards this construction is to notice that the existence of a F\o lner tiling sequence for a group $H$ naturally provides a F\o lner tiling sequence for the lamplighter $F\wr H$~\cite[Proposition~6.19]{DKLMT22}. Hence, we start by proving a similar property for lampshufflers (and more generally for lampjugglers) and lampcloners, applying the general phenomenon described in Proposition~\ref{prop:HowToBuildFolnerTiling}.

\begin{propositionletter}\label{thm:folnertilingINTRO}
Let $H$ be a finitely generated amenable group. If $H$ has a F\o lner tiling sequence, then so have $\shuf{H}$ and $\cloner{H}$.
\end{propositionletter}

See Theorems~\ref{thm:folnertilingshuffler} and~\ref{thm:folnertilingcloner} for the detailed statements and the additional properties satisfied by the tilings. Here also we shall notice that the technique is not specific to standard halos, and can be adapted to permutational halo products as well. 

\smallskip

Proposition~\ref{thm:folnertilingINTRO} enables us to build explicit orbit equivalence couplings between $\shufn{m}{\Z^{d_{1}}}$ and $\shufn{n}{\Z^{d_2}}$, where the numbers of iterations $m$ and $n$ are different (contrary to the consequence of Theorem~\ref{thm:stabilityofcouplings+quantification intro} explained at the end of the last paragraph, where $m=n$). Moreover, we are able to prove that these couplings are quantitatively optimal, see Theorem~\ref{thm:optimalityiteratedlampshufflersINTRO}.

\paragraph{Optimality of orbit equivalence couplings.} We now go back to the optimality of the couplings provided by Theorem~\ref{thm:stabilityofcouplings+quantification intro}. The results provided in this paragraph can be seen as measured analogs of rigidity results such as Theorems~\ref{th:QIrigidityShufGT24},~\ref{thm:QIclassificationiteratedlampshufflers} and~\ref{thm:QIrigidityofPWP} that we mentioned earlier. They give a more quantitative comparison between the geometries of non quasi-isometric groups.

\smallskip

To prove that the integrability conditions we obtain when applying Theorem~\ref{thm:stabilityofcouplings+quantification intro} are sharp, we make use of Theorem~\ref{thm:dklmt intro} which requires to know the isoperimetric profile of halo products. For wreath products, the computation of isoperimetric profile is completely known from~\cite{Ers03}. We have more recently found estimates on the profile of more general halo products, such estimates being exact for lampshufflers for instance~\cite{cordum25}. We refer the reader to Section~\ref{sec:PrelIsoProfHalo} for further details.

\smallskip

Notice that, since 
\begin{equation*}
    \varphi\circ \prof{K} \preccurlyeq \prof{H} \Longleftrightarrow \varphi\circ \prof{K}\circ \ln \preccurlyeq \prof{H}\circ \ln 
\end{equation*}
and since we have $\prof{\halo H} \simeq \prof{H}\circ \ln $ in many cases, one deduces that optimal orbit equivalence couplings between $H$ and $K$ provide optimal orbit equivalence couplings between $\halo H$ and $\halo K$. We begin with iterated lampshufflers.

\begin{theoremletter}[see Theorems~\ref{thm:optimalitypolynomialgrowthJugglers} and~\ref{th:ShufflerZdLamplighter}]\label{thm:optimalityofL^pnessiteratedlampshufflersINTRO}
Let $d_{1},d_{2}\ge 1$ be positive integers such that $d_{1}>d_{2}$. Let $n\ge 0$ be an integer. Let $F$ be a non-trivial finite group. Then:
\begin{itemize}
    \item there exists an $(\ld^p,\ld^{<\frac{d_1}{d_2}})$ orbit equivalence coupling from $\shufn{n}{\Z^{d_1}}$ to $\shufn{n}{\Z^{d_2}}$ if and only if $p<\frac{d_2}{d_1}$;
    \item there exists an $(\ln^{p},\exp)$-integrable orbit equivalence coupling from $\shufn{n}{F\wr\Z}$ to $\shufn{n}{\Z}$ if and only if $p<1$.
\end{itemize}
\end{theoremletter}

\begin{theoremletter}[see Corollary~\ref{cor:optimalityIteratedLampshuffler}]\label{thm:optimalityiteratedlampshufflersINTRO}
\sloppy Let $d_{1},d_{2},m,n\ge 1$ be positive integers such that $m>n$. There exists a $\left(\left(\frac{\ln^{\circ (m-n)}}{\ln^{\circ (m-n+1)}}\right)^{p},\mathrm{L}^{<\infty}\right)$-integrable orbit equivalence coupling from $\shufn{m}{\Z^{d_{1}}}$ to $\shufn{n}{\Z^{d_{2}}}$ if and only if $p<\frac{1}{d_{1}}$.
\end{theoremletter}

The crucial point of Theorem~\ref{thm:optimalityiteratedlampshufflersINTRO} is that the integrability threshold does not record $d_2$. Understanding quantitative orbit equivalence as a quantitative comparison of the geometries of non quasi-isometric groups, the intuitive interpretation of this phenomenon is that the group $\shufn{n}{\Z^{d_{2}}}$ is so geometrically negligeable compared to $\shufn{m}{\Z^{d_{1}}}$ that the geometry of the latter cannot see the difference when we change the value of the exponent $d_{2}$.

\smallskip

We emphasize here that optimality of many other couplings can be deduced from the computations of isoperimetric profiles. We refer to the end of Section~\ref{sec:folnertilingjugglers} for more details, more precisely Corollary~\ref{cor:boundOfIntegrability} and Remark~\ref{rem:FolnerWreathProduct}.

\smallskip

On the other hand, we do not have precise estimates on the isoperimetric profile of a lampcloner over a polynomial growth group, so we do not know if our couplings are quantitatively optimal. The following statement on lampcloners is thus weaker than our results on lampshufflers. However, when the based groups are $F\wr\Z$ and $\Z$, as in the second point of Theorem~\ref{thm:optimalityofL^pnessiteratedlampshufflersINTRO}, we are able to keep an \enquote{if and only if}.

\begin{theoremletter}[see Theorems~\ref{thm:optimalitypolynomialgrowthCloners} and~\ref{th:ShufflerZdLamplighter}]\label{thm:optimalitypolynomialgrowthCloners intro}
Let $\field$ be a finite field, let $d_{1},d_{2}$ be positive integers such that $d_{1}>d_{2}$ and let $F$ be a non trivial finite group. The following holds for every integer $n\ge 0$:
\begin{itemize}
    \item for every $p<\frac{d_{2}}{d_{1}}$, there exists an $(\ld^{p},\ld^{<\frac{d_{1}}{d_{2}}})$-integrable orbit equivalence coupling from $\clonern{n}{\Z^{d_{1}}}$ to $\clonern{n}{\Z^{d_{2}}}$;
    \item conversely, if there exists an $(\ld^{p},\ld^0)$-integrable orbit equivalence coupling from $\clonern{n}{\Z^{d_{1}}}$ to $\clonern{n}{\Z^{d_{2}}}$, then $p<\frac{2d_{2}}{d_{1}}$;
    \item there exists a $(\ln^p,\exp)$-integrable orbit equivalence from $\clonern{n}{F\wr\Z}$ to $\clonern{n}{\Z}$ if and only if $p<1$.
\end{itemize}
\end{theoremletter}

\begin{theoremletter}[see Corollary~\ref{cor:couplingsbetweeniteratedjugglers}]\label{thm:couplingsbetweeniteratedjugglers intro}
Let $\field$ be a finite field. Let $m,n\ge 0$ be natural integers such that $m>n$. Let $d_{1},d_{2}\ge 1$ be positive integers. Then:
\begin{itemize}
    \item for every $p<\frac{1}{2d_{1}}$, there exists a $((\ln^{\circ (m-n)})^p,\ld^0)$-integrable orbit equivalence coupling from $\clonern{m}{\Z^{d_{1}}}$ to $\clonern{n}{\Z^{d_{2}}}$;
    \item conversely, if there exists a $((\ln^{\circ (m-n)})^p,\ld^{<\infty})$-integrable orbit equivalence coupling from $\clonern{m}{\Z^{d_{1}}}$ to $\clonern{n}{\Z^{d_{2}}}$, then $p<\frac{1}{d_{1}}$.
\end{itemize}
\end{theoremletter}

Let us finally turn to permutational halo products over free abelian groups. Note that we can also get the isoperimetric profile of permutational halo products in particular cases, see for instance the computations of isoperimetric profiles in the proof of the following result. This statement in the particular case of permutational lamplighters can be viewed as a measured counterpart to Theorem~\ref{thm:QIrigidityofPWP}. 

\begin{theoremletter}[see Corollary~\ref{cor:ClassificationQOEPermutationalZd}]\label{th:ClassificationQOEPermutationalZd Intro}
Let $d_{1},d_{2},k_{1},k_{2}$ be positive integers such that $0\le k_{1}< d_{1}$, $0\le k_{2}<d_{2}$ and $\frac{d_{2}-k_{2}}{d_{1}-k_{1}}<1$, and such that $k_{1}=0$ if and only if $k_{2}=0$. Let $N_{1}$ and $N_{2}$ be normal subgroups of $\Z^{d_{1}}$ and $\Z^{d_{2}}$, with ranks $k_{1}$ and $k_{2}$ respectively. In the case where $k_{1}>0$ and $k_{2}>0$, assume that $\frac{d_{2}-k_{2}}{d_{1}-k_{1}}<\frac{k_{2}}{k_{1}}$. Then
    \begin{itemize}
        \item if $\G$ is either a finite group or a polynomial growth group, then there exists an $(\ld^{p},\ld^{0})$-integrable orbit equivalence from $\G\wr_{\Z^{d_{1}}/M}\Z^{d_{1}}$ to $\G\wr_{\Z^{d_{2}}/N}\Z^{d_{2}}$ if and only if $p<\frac{d_{2}-k_{2}}{d_{1}-k_{1}}$;
        \item there exists an $(\ld^{p},\ld^{0})$-integrable orbit equivalence from $\shuf{\Z^{d_{1}},M}$ to $\shuf{\Z^{d_{2}},N}$ if and only if $p<\frac{d_{2}-k_{2}}{d_{1}-k_{1}}$.
    \end{itemize}
\end{theoremletter}

\begin{remark}\label{rem:LimitationsTechniques}
Techniques~\ref{item:techniqueDKLMT A} and~\ref{item:techniqueDKLMT B} are complementary. Results coming from one technique cannot necessarily be recovered by the other technique.\par
On the one hand, with the stability results, we are not able to get couplings between iterated halo products with different numbers of iteration, whereas it can be achieved using F\o lner tiling sequences.\par
On the other hand, for couplings provided by F\o lner tiling sequences, the criterion that we use for a coupling to be $\varphi$-integrable (see~\cite[Proposition~6.9]{DKLMT22} that we recall in Section~\ref{sec:PreliminariesFolnerTiling}, see Theorem~\ref{thm:TilingSufficientConditionQuantitative}) shows its limitations for bigger and bigger groups. Indeed, notice that it relies on crude bounds using the diameters of the tiles. First, it prevents us from having better than integrable couplings (in particular we cannot get an $\ld^{\infty}$ orbit equivalence coupling). To get higher integrability, we must deepen our study of the tiles to get a finer criterion, which is untractable when the geometry of the groups are complicated. Secondly, still in the case of groups with complicated geometries, diameters of the tiles can grow very fast, which can prevent us from getting quantitatively optimal couplings. For instance, between $\shufn{m}{\Z^{d_{1}}}$ and $\shufn{n}{\Z^{d_{2}}}$ with $m>n$ as in Theorem~\ref{thm:optimalityiteratedlampshufflersINTRO}, we cannot directly get optimal couplings using F\o lner tiling sequences. We first must consider the case of $\shufn{(m-n)}{\Z^{d_{1}}}$ and $\Z^{d_{2}}$, where $\Z^{d_{2}}$ has tiles easier to handle, and then apply the stability result $n$ times.\par
Both techniques share common limitations. For instance, the existence of an optimal coupling between halo products of different nature, for example between $(\Z/2\Z)\wr\Z$ and $\shuf{\Z}$, is still an open question.\par
Finally, it is worth noticing that, by works of Genevois and Tessera~\cite{GT24a}, and the second-named author~\cite{Dum24}, the existence of a bi-Lipschitz equivalence between the base groups is not the sufficient and necessary condition for very \enquote{similar} (but different) halo products to be bi-Lipschitz equivalent, but rather the existence of a \textit{scaling quasi-isometry}, with scaling factor different than $1$. By \enquote{similar} halo products, we mean for instance, $\G\wr H$ and $\La\wr K$ when $\G$ and $\La$ are non trivial finite groups with different cardinalities, or lampjugglers $\juggler{r}{H}$ and $\juggler{s}{K}$ with $r\neq s$, or lampcloners $\cloner{H}$ and $\mathsf{Cloner}_{\mathbb{l}}(K)$ for finite fields $\field$ and $\mathbb{l}$ with different cardinalities. Since our ideas always lead us to a stability result for bi-Lipschitz equivalence, via $\ld^{\infty}$ orbit equivalence, this means that we cannot expect stability results like Theorem~\ref{thm:stabilityofcouplings+quantification intro} between \enquote{similar} halo products in full generality.
\end{remark}

\paragraph{Plan of the paper.} After a few preliminaries in Section~\ref{sec:preliminaries}, Section~\ref{sec:stability} is devoted to the construction of orbit equivalence couplings between halo products, by building natural actions of the halo products provided actions of the base groups. In Section~\ref{sec:folnerTiling}, we will use F\o lner tiling sequences to get other orbit equivalence couplings, for instance between $\Z^d$ and $\shuf{\Z^k}$. Finally, Section~\ref{sec:commentsandquestions} records various questions related to the article.

\paragraph{Acknowledgements.}

We are thankful to Juan Paucar for useful discussions about the topic, and to François Le Maître for fruitful suggestions.
\section{Preliminaries}\label{sec:preliminaries}

\subsection{Notations}\label{sec:notations}

Given non-decreasing functions $f,g\colon \R_{+}\to \R_{+}$, we write $f(x)=O(g(x))$ if there exists $C>0$ such that $f(x)\le Cg(x)$ for all $x$ large enough, and $f(x)=o(g(x))$ if $\frac{f(x)}{g(x)}$ goes to $0$ as $x$ goes to $+\infty$. Lastly, we write $f\sim g$, and we say that $f$ and $g$ are~\textit{equivalent}, if $\frac{f(x)}{g(x)}$ goes to $1$ as $x$ goes to $+\infty$.  

\smallskip

A slightly weaker notion is the one of~\textit{asymptotic equivalence}. Namely, if $f,g\colon \R_{+}\to \R_{+}$ are non-decreasing, we say that $f$ is~\textit{dominated} by $g$, and we write $f\preccurlyeq g$, if there is a constant $C>0$ such that $f(x)=O(g(Cx))$. We say that $f$ and $g$ are~\textit{asymptotically equivalent}, written $f\simeq g$, if $f\preccurlyeq g$ and $g\preccurlyeq f$. Note that two equivalent functions are asymptotically equivalent.

\smallskip

\sloppy Throughout the paper, the probability spaces are assumed to be atomless and standard. This implies that they are all isomorphic to $([0,1],\mathrm{Leb})$, where $\mathrm{Leb}$ stands for the Lebesgue measure on $[0,1]$. Given a measurable map $f\colon X\to Y$ between two measure spaces $(X,\mu)$ and $(Y,\nu)$, $f_{*}\mu$ stands for the push-forward of $\mu$ under $f$. The abbreviation pmp stands for~\textit{probability measure-preserving}.

\smallskip

Given a group $G$, we denote $1_{G}$ its neutral element, and if $G$ is generated by a finite set $S$, $\text{Cay}(G,S)$ refers to the Cayley graph of $G$ with respect to $S$, that is the graph whose vertices are elements of $G$ and whose edges are pairs of the form $(g,gs)$ with $g\in G$ and $s\in S\cup S^{-1}\setminus\lbrace 1_{H}\rbrace$, while $|\cdot|_{S}$ denotes the usual length function on $G$ associated to $S$. The notation $V_{G}$ refers to the~\textit{growth function} of $G$, defined by $V_{G}(n)=|\lbrace g\in G : |g|_{S}\le n\rbrace|$. Recall that its asymptotic behaviour is independent of the choice of $S$.

\subsection{Isoperimetric profile}\label{sec:isoprof}

For a finitely generated group $G$ and a finite generating set $S_{G}$, its $\ell^{1}$-\textit{isoperimetric profile}, or simply its \textit{isoperimetric profile}, is the function $\prof{G}\colon \N\to \R_{+}$ given by 
\begin{equation*}
    \prof{G}(n) \defeq \sup_{|A|\le n}\frac{|A|}{|\partial_{G} A|}
\end{equation*}
where $\partial_{G}A \defeq \lbrace g\in G\setminus A : \exists s\in S_{G}, \exists h\in A, g=hs\rbrace$ is the~\textit{boundary} of $A$ in $G$. There exists a more general notion of $\ell^{p}$-isoperimetric profile for $p\ge 1$, but we will only focus on the case $p=1$.

\smallskip

Recall also that the isoperimetric profile of a group $G$ is the generalized inverse of its~\textit{F\o lner function} $\textup{F\o l}_{G}\colon \N\to\R_{+}$, defined as
 \begin{equation*}
     \textup{F\o l}_{G}(n) \defeq \inf\left\lbrace |A| : \frac{|\partial_{G}A|}{|A|} \le \frac{1}{n} \right\rbrace.
 \end{equation*}

\begin{remark}\label{rem:InverseEachOther}
Notice that, given the asymptotic behaviour of the F\o lner function, we can deduce the asymptotic behaviour of the isoperimetric profile, even though they are not real inverses of each other, but only generalized inverses~\textit{a priori}. Indeed, a non-decreasing function $\R_{+}\to \R_{+}$ is always asymptotically equivalent to an increasing function $\R_{+}\to \R_{+}$ (cf.~\cite[Remark~1.2]{Corr24}) and it is not hard to check that $\simeq$ is preserved when passing to generalized inverses. Thus, in the sequel, we can and will assume that the F\o lner function and the isoperimetric profile are injective and then real inverses of each other. Hence, studying the asymptotic behaviour of $\prof{G}$ is the same as studying the asymptotic behaviour of $\textup{F\o l}_{G}$. 
\end{remark}

The asymptotic behaviour of the isoperimetric profile is stable under quasi-isometries, and is in particular independent of the choice of a generating set for $G$. More generally, the isoperimetric profile is monotonuous when passing to finitely generated subgroups:

\begin{theorem}[{\cite[Lemma~4]{Ers03}}]\label{thm:profilemonotonuousSubgroup}
Let $H$ be a finitely generated subgroup of a finitely generated group $G$. Then one has $\prof{G}(n) \preccurlyeq \prof{H}(n)$.
\end{theorem}

This has been widely generalized to general regular maps by Delabie, Koivisto, Le Maître and Tessera, using connections with quantitative measure equivalence.

\begin{theorem}[{\cite[Theorem~1.5]{DKLMT22}}]\label{thm:profilemonotonuous}
Let $G$ and $H$ be finitely generated amenable groups. If there exists a regular map from $G$ to $H$, then $\prof{H}(n) \preccurlyeq \prof{G}(n)$.
\end{theorem}

\smallskip

Note that a group is amenable if and only if its isoperimetric profile is unbounded. The idea to keep in mind is that the isoperimetric profile is a measurement of how much amenable a group is. The faster the isoperimetric profile tends to infinity, the more the group is amenable. In particular, the isoperimetric profile is a particularly well suited invariant to distinguish amenable groups with exponential growth up to quasi-isometries or regular maps, and it has now been computed for many classes of groups, among which:
\begin{itemize}
    \item $\prof{G}(n) \simeq n^{\frac{1}{d}}$ if $G$ has polynomial growth of degree $d\ge 1$;
    \item $\prof{G}(n)\simeq \ln(n)$ for solvable Baumslag-Solitar groups and lamplighters $F\wr\Z$, where $F$ is a non-trivial finite group;
    \item $\prof{G}(n) \simeq \ln(n)$ for any polycyclic group with exponential growth~\cite{Pit95, Pit00}, or more generally any exponential growth group within the class GES of Tessera~\cite[Corollary 5]{Tes13};
    \item $\prof{F\wr N}(n)\simeq \ln(n)^{\frac{1}{d}}$ with $F$ finite and non-trivial, and $N$ of growth degree $d\ge 1$~\cite{Ers03}.
\end{itemize}

More generally,~\cite[Theorem 1]{Ers03} provides a general formula for computing F\o lner functions of wreath products. Thus the last example can be extended to iterated wreath products. 

\smallskip

A fundamental result in geometric group theory is the one of Coulhon and Saloff-Coste, who proved in~\cite{CSC93} that for a finitely generated group $G$ and a finite symmetric generating set $S$ of $G$, one has 
\begin{equation*}
    \frac{|\partial_{G}F|}{|F|} \ge \frac{1}{4|S|}\cdot\frac{1}{\Phi_{G}(2|F|)}
\end{equation*}
for any finite set $F\subset G$, where $\Phi_{G}\colon \R_{>0} \to \N$, $\Phi_{G}(t) \defeq \min\lbrace n \ge 0 : V_{G}(n)>t\rbrace$ is the \textit{inverse growth function} of $G$. Since then, it has constantly been improved to thiner inequalities, see for instance~\cite{PS22}. Inverting this inequality and taking the sup, one directly gets the upper bound 
\begin{equation*}
    \prof{G}(n) \preccurlyeq \Phi_{G}(n)
\end{equation*}
on the isoperimetric profile of $G$. This upper bound is optimal when $G$ has polynomial growth, while if it has exponential growth, one only gets $\prof{G}(n) \preccurlyeq \ln(n)$. In fact,~\cite[Theorem~1.1]{BZ21} describes a large class of possible asymptotic behaviours for isoperimetric profiles of finitely generated groups with exponential growth, namely for any non-decreasing function $f$ such that $x\mapsto \frac{x}{f(x)}$ is non-decreasing, there exists a finitely generated group of exponential growth whose profile is $\simeq \frac{\ln(x)}{f(\ln(x))}$.

\smallskip

Isoperimetric profiles are also particularly studied for their relations with return probabilities of random walks on groups, see for instance~\cite{SCZ15, SCZ16, SCZ18, BZ21}.

\subsection{Orbit equivalence}\label{sec:orbit equivalence}

We now recall additional background on orbit equivalence and its quantitative versions. 

\begin{definition}
Let $H$ and $K$ be countable groups. An~\textit{orbit equivalence coupling} between $H$ and $K$ is a standard probability space $(X,\mu)$ equipped with free pmp $H$- and $K$-actions that share the same orbits up to a null set, namely for almost every $x\in X$, the $H$- and $K$-orbits of $x$ coincide.
\end{definition}

If such an orbit equivalence coupling exists, then we say that the groups $H$ and $K$ are~\textit{orbit equivalent}. An orbit equivalence coupling $(X,\mu)$ between two groups $H$ and $K$ provides measurable maps $c_{H,K}\colon H\times X\to K$ and $c_{K,H}\colon K\times X\to H$, called the~\textit{cocycles}, and defined by the equalities
\begin{equation*}
    h\cdot x=c_{H,K}(h,x)\cdot x \;\text{and}\;k\cdot x=c_{K,H}(k,x)\cdot x
\end{equation*}    
for every $h\in H$, $k\in K$ and for almost every $x\in X$. These maps are well-defined by freeness of the actions $H,K\act X$ and satisfy the \textit{cocycle identities}
\begin{equation*}
    c_{H,K}(h'h,x)=c_{H,K}(h',hx)c_{H,K}(h,x)\; \text{and}\;
    c_{K,H}(k'k,x)=c_{K,H}(k',kx)c_{K,H}(k,x)
\end{equation*}
for every $h,h'\in H$, every $k,k'\in K$ and almost every $x\in X$.

\smallskip

By freeness of the actions, two orbit equivalent groups have the same cardinality. It is not hard to prove that two finite groups with same cardinality are orbit equivalent. In the case of infinite groups, Ornstein and Weiss proved the following for amenable groups.

\begin{theorem}[\cite{OW80}]
Any two infinite amenable groups are orbit equivalent.
\end{theorem}

Thus, in order to get an interesting theory for amenable groups, we need to strengthen the definition of orbit equivalence. To this end, for finitely generated groups, we put constraints on our cocycles by requiring some~\textit{quantification}.

\begin{definition}[\cite{DKLMT22}]\label{def:quantitativeCocycle}
Let $H$ and $K$ be finitely generated groups, with finite generating sets $S_{H}$, $S_{K}$. Let $(X,\mu)$ be an orbit equivalence coupling between $H$ and $K$, and let $\varphi,\psi\colon\R_{+}\to\R_{+}$ be non-decreasing maps. We say that $(X,\mu)$ is a $(\varphi,\psi)$\textit{-integrable orbit equivalence coupling} from $H$ to $K$ if for every $h\in H$ and every $k\in K$, there exist constants $C_{h}>0$ and $C'_{k}>0$ such that
\begin{equation*}
\int_{X}{\varphi\left(\frac{\left|c_{H,K}(h,x)\right|_{S_{K}}}{C_{h}}\right)\mathrm{d}\mu(x)}<\infty\;\text{and}\;\int_{X}{\psi\left (\frac{\left|c_{K,H}(k,x)\right|_{S_{H}}}{C'_{k}}\right)\mathrm{d}\mu(x)}<\infty.
\end{equation*}    
In that case, we (abusively) say that the cocycle $c_{H,K}$ is \textit{$\varphi$-integrable} and that the cocycle $c_{K,H}$ is \textit{$\psi$-integrable}.
\end{definition}

Moreover, for $p>0$, we write $\ld^{p}$ instead of $\varphi$- or $\psi$-integrable if we consider the map $t\mapsto t^{p}$, we write $\ld^{<\infty}$ if the corresponding cocycle is $\ld^p$ for every $p>0$, and we write $\ld^{0}$ when no requirement is made on the cocycle. For instance, the orbit equivalence coupling is $(\varphi,\ld^{p}$)-integrable if $c_{H,K}$ is $\varphi$-integrable and $c_{K,H}$ is in $\ld^{p}(X,\mu)$; it is $(\ld^{p},\ld^0$)-integrable if $c_{G,H}$ is $\ld^{p}$. Finally, if an orbit equivalence coupling from $H$ to $K$ is $(\varphi,\varphi)$-integrable, then we say that it is a $\varphi$\textit{-integrable} orbit equivalence coupling from $H$ to $K$.

\begin{remark}
We will not deal with the more general notion of~\textit{measure equivalence}, and its quantitative forms (see~\cite{DKLMT22}). However, note that the existence of a $(\varphi,\psi)$-integrable orbit equivalence coupling from $H$ to $K$ is equivalent to the existence of a $(\varphi,\psi)$-integrable measure equivalence coupling from $H$ to $K$ with equal fundamental domains.
\end{remark}

\begin{remark}\label{rm:checkongenerators}
The constants $C_{h}$ appearing in the definition of $\varphi$-integrability of the cocycle $c_{H,K}$ are necessary because we need the following properties:
\begin{itemize}
\item this notion of $\varphi$-integrability does not depend on the choice of the finite generating set of $K$, since for any two finite generating sets $S_{K},S'_{K}$ of $K$, there exists a constant $C>0$ such that
\begin{equation*}
    \frac{1}{C}{\left|k\right|}_{S'_{K}}\le {\left|k\right|}_{S_{K}}\le C{\left|k\right|}_{S'_{K}}
\end{equation*}
for every $k\in K$;
\item if $\varphi \simeq \psi$, then $\varphi$-integrability and $\psi$-integrability are equivalent notions;
\item to prove that the cocycle $c_{H,K}\colon H\times X\to K$ is $\varphi$-integrable, it suffices to check the finiteness of 
\begin{equation*}
    \int_{X}{\varphi\left(\frac{\left|c_{H,K}(h,x)\right|_{S_{K}}}{C_{h}}\right)\mathrm{d}\mu(x)}
\end{equation*}
for every $h$ in a finite generating set of $H$. This follows from~\cite[Proposition~2.22]{DKLMT22}.
\end{itemize}
\end{remark}

As mentioned in the introduction, quantitative orbit and measure equivalences capture the geometry of groups, in contrast with Orstein-Weiss theorem. In the appendix of~\cite{Aus16}, Bowen shows the invariance of the growth function under $\ld^1$ orbit equivalence.

\begin{theorem}[{\cite[Theorem~B.2]{Aus16}}]
Let $H$ and $K$ be finitely generated groups. If $H$ and $K$ are $\ld^{1}$ orbit equivalent, then $V_{H}(n)\simeq V_{K}(n)$.
\end{theorem}

It is therefore natural to wonder whether these rigidity results still hold for more general quantifications as defined above. Given $p\ge 1$, the notion of $\ld^p$ orbit equivalence (more generally $\ld^p$ measure equivalence) has been introduced in~\cite{BFS13}, and more generally $(\varphi,\psi)$-integrable orbit equivalence was first defined in~\cite{DKLMT22} to study the weaker notion of $\ld^p$ orbit equivalence for $p<1$. In this wider setup, Delabie, Koivisto, Le Maître and Tessera refined Bowen's result as follows.

\begin{theorem}[{\cite[Theorem~3.1]{DKLMT22}}]
Let $H$ and $K$ be finitely generated groups. Let $\varphi\colon\R_{+}\rightarrow \R_{+}$ be an increasing and subadditive map. If there is a $(\varphi, \ld^{0})$-integrable orbit equivalence from $H$ to $K$, then 
\begin{equation*}
    V_{H}(n) \preccurlyeq V_{K}(\varphi^{-1}(n))
\end{equation*}
where $\varphi^{-1}$ is the inverse function of $\varphi$.
\end{theorem}

This inequality provides explicit upper bounds on how integrable the cocycles of an orbit equivalence coupling can be, thus implying that quantitative orbit equivalence is more complex than the trivial relation of orbit equivalence (among finitely generated amenable groups). 

\smallskip

Going further, Delabie, Koivisto, Le Maître and Tessera also proved in~\cite{DKLMT22} an inequality that involves rather the isoperimetric profile, which is more suitable to distinguish amenable groups of exponential growth for instance.

\begin{theorem}[{\cite[Theorem~1.1]{DKLMT22}}]\label{thm:dklmt}
Let $\varphi\colon\R_{+}\to\R_{+}$ be a non-decreasing function such that $t\mapsto \frac{t}{\varphi(t)}$ is non-decreasing. Let $H$ and $K$ be finitely generated groups. Assume that there exists a $(\varphi,\ld^0)$-integrable orbit equivalence coupling from $H$ to $K$. Then their isoperimetric profiles satisfy the asymptotic inequality
\begin{equation*}
    \varphi\circ\prof{K}(n) \preccurlyeq \prof{H}(n).
\end{equation*}
\end{theorem}

In most cases, this inequality means that $\varphi$ is asymptotically less that $\prof{H}\circ\prof{K}^{-1}$, and it is proved in~\cite{Corr24} that we cannot have an $(\prof{H}\circ\prof{K}^{-1},\ld^0)$-integrable orbit equivalence coupling from $H$ to $K$.

\begin{theorem}[{\cite[Theorem~B]{Corr24}}]\label{thm:threshold}
Let $H$ and $K$ be finitely generated groups. Assume that there exists a non-decreasing function $f_{H}$ and an increasing function $f_{K}$ satisfying $f_{H}(x) \simeq\prof{H}(x)$, $f_{K}(x)\simeq\prof{K}(x)$ and the following as $x\rightarrow+\infty$:
\begin{equation*}
		 f_{H}(x)=o\left(f_{K}(x)\right),
\end{equation*}
\begin{equation*}
\forall C>0,\; f_{H}(Cx)=O\left(f_{H}(x)\right),
\end{equation*}
\begin{equation*}
\forall C>0,\; f_{H} \circ  f_{K}^{-1}(Cx)=O\left(f_{H}\circ f_{K}^{-1}(x)\right).
\end{equation*}
Then there is no $(f_{H}\circ f_{K}^{-1},\ld^0)$-integrable orbit equivalence coupling from $H$ to $K$.
\end{theorem}

Actually, Theorems~\ref{thm:dklmt} and~\ref{thm:threshold} have been proved in~\cite{DKLMT22} and~\cite{Corr24} in the more general setting of measure equivalence.

\smallskip

Combined with explicit constructions in~\cite{DKLMT22} (see Section~\ref{sec:PreliminariesFolnerTiling}, more specifically Theorem~\ref{thm:ExplicitCouplingZd}, of the aforementioned paper), a nice consequence of these statements, among others, is the following.

\begin{corollary}
Let $d_{1}>d_{2}$ be positive integers. Then there exists an $(\ld^p,\ld^{<\frac{d_{1}}{d_{2}}})$ orbit equivalence coupling from $\Z^{d_{1}}$ to $\Z^{d_{2}}$ if and only if $p<\frac{d_{2}}{d_{1}}$.
\end{corollary}

This has been widely extended to the class of groups of polynomial growth by Delabie, Llosa-Isenrich and Tessera~\cite[Theorem~1.6]{DLIT25}.

\smallskip

In our computations below, we will require the following result of composition of actions, which enables us to have some kind of transitivity of quantitative orbit equivalence among finitely generated amenable groups. This is a reformulation of~\cite[Propositions~2.26,~2.29 and~2.30]{DKLMT22} in terms of orbit equivalence.

\begin{theorem}\label{thm:CompositionCouplings}
Let $G_{1}, G_{2}, G_{3}$ be three finitely generated groups and $\varphi,\varphi',\psi,\psi'\colon\R_{+}\to\R_{+}$ be non-decreasing maps. Let us assume that there exists a $(\varphi,\psi)$-integrable orbit equivalence coupling from $G_{1}$ to $G_{2}$, and a $(\varphi',\psi')$-integrable orbit equivalence coupling from $G_{2}$ to $G_{3}$. Let us also assume that the pairs $(\varphi,\varphi')$ and $(\psi',\psi)$ lie in the set of pairs of functions $(f,g)$ satisfying one of the following two conditions:
\begin{enumerate}[label=(\roman*)]
    \item\label{item:CompositionHyp1} either $f$ and $g$ are subadditive and $g$ is concave;
    \item\label{item:CompositionHyp2} or $f(x)\simeq g(x)\simeq x^p$ for some $p\ge 1$.
\end{enumerate}
Then there exists a $(\varphi'',\psi'')$-integrable orbit equivalence coupling from $G_{1}$ to $G_{3}$, where
\begin{equation*}
\varphi''(x)=\left\{\begin{array}{ll}
        \varphi\circ\varphi'(x) & \text{if }(\varphi,\varphi')\text{ satisfies Assumption~\textit{\ref{item:CompositionHyp1}}}\\
        x^p & \text{if }(\varphi,\varphi')\text{ satisfies Assumption~\textit{\ref{item:CompositionHyp2}} for some }p\ge 1
\end{array}\right.
\end{equation*}
    and
\begin{equation*}
\psi''(x)=\left\{\begin{array}{ll}
        \psi'\circ\psi(x) & \text{if }(\psi',\psi)\text{ satisfies Assumption~\ref{item:CompositionHyp1}}\\
        x^p & \text{if }(\psi',\psi)\text{ satisfies Assumption~\ref{item:CompositionHyp2} for some }p\ge 1
\end{array}\right..  
\end{equation*}
\end{theorem}

Finally, let us introduce orbit equivalence for pairs of groups, introduced and studied in a forthcoming paper~\cite{CDT26}.

\begin{definition}[{\cite{CDT26}}]
Let $H$ and $K$ be countable groups and let $M\lhd H$, $N\lhd K$ be normal subgroups. We say that the pairs $(H,M)$ and $(K,N)$ are~\textit{orbit equivalent} if there exists free pmp $H$- and $K$-actions on a standard and atomless probability space $(X,\mu)$, satisfying $H\cdot x=K\cdot x$ and $M\cdot x=N\cdot x$ for almost every $x\in X$. The space $(X,\mu)$ is called an~\textit{orbit equivalence coupling} between the pairs $(H,M)$ and $(K,N)$.
\end{definition}

Throughout the paper, we only consider~\textit{normal} pairs, meaning that $M$ and $N$ are normal subgroups of $H$ and $K$ respectively. A coupling $(X,\mu)$ between two normal pairs $(H,N)$ and $(K,M)$ naturally provides two maps 
\begin{equation*}
    \kappa_{H/M, K/N}\colon H/M \times X \to K/N, \; \kappa_{K/N, H/M}\colon K/N\times X\to H/M.
\end{equation*}
called the~\textit{pseudo-cocycles} of the coupling, defined as
\begin{equation*}
    \kappa_{H/M, K/N}([h]_{H/M},x)\defeq [c_{H,K}(h,x)]_{K/N}\text{ and }\kappa_{K/N,H/M}([k]_{K/N},x)\defeq [c_{K,H}(k,x)]_{H/M}
\end{equation*}
where $[h]_{H/M}$ denotes the class of $h\in H$ modulo $M$.

\smallskip

As cocycles provide random bijections $c_{H,K}(\cdot,x)\colon H\to K$ and $c_{K,H}(\cdot,x)\colon K\to H$ between the ambient groups, orbit equivalence of pairs additionally provide random bijections between the subgroups of the pairs: $c_{M,N}(\cdot,x)\colon M\to N$ and $c_{N,M}(\cdot,x)\colon N\to M$; as well as random bijections between the quotient groups, using the pseudo-cocycles: $\kappa_{H/M,K/N}(\cdot,x)\colon H/M\to K/N$ and $\kappa_{K/N,H/M}(\cdot,x)\colon K/N\to H/M$. Note that all these couples of bijections in fact yield bijections that are inverses of each other. Furthermore, the pseudo-cocycles satisfy the so-called~\textit{pseudo-cocycle identity}, analogous to the cocycle identity, namely:
$$\kappa_{H/M,K/N}([h'h]_{H/M},x)=\kappa_{H/M,K/N}([h']_{H/M},h x)\kappa_{H/M,K/N}([h]_{H/M},x)$$
and
$$\kappa_{K/N,H/M}([k'k]_{K/N},x)=\kappa_{K/N,H/M}([k']_{K/N},k x)\kappa_{K/N,H/M}([k]_{K/N},x)$$
for almost every $x\in X$ and for every $h,h'\in H$, $k,k'\in K$. 

\smallskip

Among other things, in~\cite{CDT26} are provided explicit examples of orbit equivalence couplings between pairs. In the light of our stability results below, we will use the following.

\begin{proposition}[{\cite{CDT26}}]\label{prop:OEofpairsfreeabeliangroups}
Let $d_{1},d_{2},k_{1},k_{2}$ be positive integers satisfying $0<k_{1}<d_{1}$ and $0<k_{2}<d_{2}$. Let $N_{1}$ and $N_{2}$ be normal subgroups of $\Z^{d_1}$ and $\Z^{d_2}$, with rank $k_{1}$ and $k_{2}$ respectively. Then, denoting $P=\min{\left (\frac{k_{2}}{k_{1}},\frac{d_{2}-k_{2}}{d_{1}-k_{1}}\right)}$ and $Q=\min{\left (\frac{k_{1}}{k_{2}},\frac{d_{1}-k_{1}}{d_{2}-k_{2}}\right)}$, there exists an $(\ld^{<P},\ld^{<Q})$-integrable orbit equivalence coupling from $(\Z^{d_{1}},N_{1})$ to $(\Z^{d_{2}},N_{2})$.
\end{proposition}

The integrability conditions in this statement are sharp.

\subsection{Halo products}\label{sec:halo}

\subsubsection{Definition and main examples}\label{sec:defHaloExample}

Let us define halo products and give the main examples we are interested in.

\begin{definition}
Let $X$ be a set. A~\textit{halo of groups $\halo$ over $X$} is the data, for any subset $S\subset X$, of a group $L(S)$ such that:
\begin{itemize}
    \item for all $R,S\subset X$, if $R\subset S$ then $L(R)\leqslant L(S)$;
    \item $L(\emptyset)=\lbrace 1\rbrace$ and $L(X)=\langle L(S) : S\subset X \;\text{finite}\rangle$;
    \item for all $R,S\subset X$, $L(R\cap S)=L(R)\cap L(S)$.
\end{itemize}
\end{definition}

Given an action $H\curvearrowright X$ and a morphism $\alpha\colon H \to \text{Aut}(L(X))$ satisfying $\alpha(h)(L(S))=L(hS)$ for any $S\subset X$ and $h\in H$, the~\textit{permutational halo product} $\halo_{X,\alpha}X$ is the semi-direct product 
\begin{equation*}
    \halo_{X,\alpha}H \defeq L(X)\rtimes_{\alpha}H.
\end{equation*}

The definition is motivated by permutational wreath products, which are basic examples of permutational halo products. Indeed, given groups $F,H$ and an action $H\subset X$, set $L(S)\defeq \bigoplus_{S}F$ for any $S\subset X$. Then $\halo_{X,\alpha}H$ coincides with $F\wr_{X}H$, where $\alpha$ is the action of $H$ on $H\curvearrowright\bigoplus_{X}F$ obtained by permuting the coordinates through $H\curvearrowright X$. In particular, for $X=H$ and the left-multiplication action of $H$ on itself, we recover a description of the wreath product $F\wr H$ as a halo product.

\smallskip

In the case of~\textit{standard} halo products, namely $X\defeq H$ and $H$ acts on itself by left-multiplication, we introduce some properties, introduced in~\cite{cordum25} but also used in~\cite{GT24a}, and satisfied by several examples.

\begin{definition}
    Let $H$ be a group and $\halo H$ be a standard halo product over $H$.
    \begin{itemize}
        \item If $H$ is finitely generated, with a finite generating set $S_H$, we say that $\halo H$ is~\textit{naturally generated} if it is generated by
        \begin{equation*}
            \left\{(1_{L(H)},s) : s\in S_H\right\}\cup\bigcup_{s\in S_H}{\left\{(\sigma,1_H) : L(\{1_H,s\})\right\}}.
        \end{equation*}
        \item We say that $\halo H$ has~\textit{consistent blocks} if for any finite subset $S\subset H$, the group $L(S)$ is finite and its cardinality only depends on the cardinality of $S$. In this case, we can define the~\textit{lamp growth function} $\La_{\halo H}\colon\N\to\N$ as
        \begin{equation*}
            \La_{\halo H}(n)\defeq |L(S)|
        \end{equation*}
        for every positive integer $n$ and any subset $S\subset H$ of cardinality $n$.
    \end{itemize}
\end{definition}

For instance, if $F$ is finite and $H$ is finitely generated, then the lamplighter $F\wr H$ is naturally generated, have consistent blocks and the lamp growth function is given by
\begin{equation*}
    \La_{F\wr H}(n)=|F|^n.
\end{equation*}

\paragraph{Lampshufflers.} Let $H$ be a group, and let $\fsym{H}$ the group of finitely supported permutations of $H$, that is the group of bijections $H\to H$ that are the identity outside a finite subset of $H$. The group $H$ acts naturally on $\fsym{H}$, via 
\begin{equation*}
    (h\cdot\sigma)(x) \defeq h\sigma(h^{-1}x), \; x\in H
\end{equation*}
for any $h\in H$ and $\sigma\in\fsym{H}$. Indeed, if $\sigma\colon H\to H$ is a finitely supported bijection and $h\in H$, then so is $h\cdot \sigma$ and $\supp{(h\cdot\sigma)}=h\cdot \supp{(\sigma)}$.

\smallskip

The~\textit{lampshuffler group over $H$}, denoted $\shuf{H}$, is then defined as the semidirect product 
\begin{equation*}
    \shuf{H}\defeq \fsym{H} \rtimes H.
\end{equation*}

It coincides with the halo product $\halo H$ where $L(S)\defeq \fsym{S}$ for any $S\subset H$. Additionally, if $H$ is finitely generated and $S_{H}$ denotes a finite generating set, then $\shuf{H}$ is generated by the finite set 
\begin{equation*}
S_{\shuf{H}}\defeq\left\lbrace(\tau_{1_H,s},1_H) : s\in S_{H}\cup S_{H}^{-1}\right \rbrace\cup\left\lbrace(\text{id},s) : s\in S_{H}\right\rbrace
\end{equation*}
where, given any $x,y\in H$, $\tau_{x,y}\in\fsym{H}$ is the transposition that swaps $x$ and $y$, that is $\tau_{x,y}(x)=y$, $\tau_{x,y}(y)=x$ and $\tau_{x,y}(h)=h$ for any $h\neq x,y$. 

An element $(\sigma,h)\in\shuf{H}$ can be seen as a labelling of the vertices of the Cayley graph $\text{Cay}(H,S_{H})$ (a vertex $p\in H$ carries the label $\sigma(p)$) together with an arrow pointing at some vertex $h\in H$, and there are two types of moves in $\text{Cay}(\shuf{H},S_{\shuf{H}})$ to go from $(\sigma, h)$ to a neighbouring vertex:
\begin{itemize}
    \item either the arrow goes from $h$ to a neighbouring vertex in $H$;
    \item or the arrow stands on the vertex $h\in H$, and swaps its label with the label of one of its neighbours in $H$.
\end{itemize}

Lampshufflers are naturally generated, have consistent blocks, and their lamp growth function is given by
\begin{equation*}
    \La_{\shuf{H}}(n)=n!.
\end{equation*}

\smallskip

Given a normal subgroup $M$, we analogously define the permutational lampshuffler $\shuf{H,M}\defeq \fsym{H/M}\rtimes H$, using the action $H\curvearrowright H/M$ by left translation. A (finite) generating set of this group is derived from a (finite) generating set of $H$ in a similar way.

\paragraph{Lampjugglers.} Lampshufflers are in fact particular instances of a broader family of groups, called~\textit{lampjugglers}. Given a group $H$ and an integer $r\ge 1$, the~\textit{lampjuggler over $H$} is the semi-direct product
\begin{equation*}
    \juggler{r}{H} \defeq \fsym{H\times\lbrace 1,\dots,r\rbrace} \rtimes H
\end{equation*}
where $H$ acts on $\fsym{H\times\lbrace 1,\dots,r\rbrace}$ through its initial action on $H\times\lbrace 1,\dots,r\rbrace$ given by $h\cdot (x,i) \defeq (hx, i)$. It can be described as the halo $\halo H$ where $L(S)\defeq \fsym{S\times\lbrace 1,\dots, r\rbrace}$, $S\subset H$. As for lampshufflers, lampjugglers are finitely generated as soon as so is $H$, and one can check that if $S_{H}$ is a finite generating set for $H$, then the finite set
\begin{equation*}
    \lbrace \tau_{(1_{H},i),(s,j)} : s\in S_{H}\cup S_{H}^{-1}, 1\le i,j\le r\rbrace  \cup \lbrace (\text{id},s) : s\in S_{H}\rbrace 
\end{equation*}
generates $\juggler{r}{H}$. Here, an element $(\sigma, h)\in\juggler{r}{H}$ can be seen as a labelling of the vertices of $\text{Cay}(H,S_{H})\times\lbrace 1,\dots,r\rbrace$ together with an arrow pointing at some vertex $h\in H$. Right-multiplying $(\sigma, h)$ by a generator for the above set amounts either to move the arrow from $h$ to a neighbouring vertex $hs$ in $H$, or to keep the arrow on $h\in H$ and switching the labels of two vertices in $h \times\lbrace 1, \dots, r\rbrace$ and $hs \times \lbrace1, \dots, r\rbrace$ for some neighbour $hs$ of $h$.

\smallskip

As for lampshufflers, lampshufflers are naturally generated, have consistent blocks, and their lamp growth function is given by
\begin{equation*}
    \La_{\juggler{r}{H}}(n)=(rn)!.
\end{equation*}

\smallskip

Given a normal subgroup $M$, we analogously define the permutational lampjuggler $\juggler{r}{H,M}\defeq \fsym{(H/M)\times\{1,\dots,r\}}\rtimes H$, using the action $H\curvearrowright H/M$ by left translation. A (finite) generating set of this group is derived from a (finite) generating set of $H$ in a similar way.

\paragraph{Lampcloners.} Let $H$ be a group and let $\field$ be a field. Denote $V_{H}$ the $\field$-vector space admitting $H$ as a basis, and denote $\lbrace e_{u} : u\in H\rbrace$ a formal basis. Let $\text{FGL}(H)$ be the group of linear automorphisms $V_{H}\to V_{H}$ that fix all but finitely many elements of $H$. This group can also be seen as the group of finitely supported invertibles matrices with coefficients in $k$ whose entries are indexed by $H\times H$. Once again, the action of $H$ on itself naturally yields an action of $H$ on $\text{FGL}(H)$. The~\textit{lampcloner over $H$} is the semi-direct product 
\begin{equation*}
    \cloner{H} \defeq \text{FGL}(H)\rtimes H.
\end{equation*}
It is a halo product, for the collection $L(S)\defeq \text{FGL}(S)$, for every $S\subset H$, where $\text{FGL}(S)$ is thought of as the subgroup of $\text{FGL}(H)$ of linear automorphisms $V_{H}\to V_{H}$ that fix $H\setminus S$ and that stabilise the subspace $\langle S\rangle\subset V_{H}$. 

In addition, if $\field$ is finite and if $H=\langle S_{H}\rangle$ is finitely generated, then the finite set
\begin{equation*}
    S_{\cloner{H}}\defeq\big\lbrace (\delta_{1_{H}}(\lambda), 1_{H}) : \lambda \in \field\setminus\lbrace 0\rbrace\big\rbrace\cup\big\lbrace (\tau_{1_{H},s}(\lambda),1_{H}) : s\in S_{H}\cup S_{H}^{-1}\big\rbrace \cup\lbrace (\text{id}, s) : s\in S_{H}\rbrace
\end{equation*}
generates $\cloner{H}$, where, given $p,q\in H$ and $\lambda\in \field\setminus\lbrace 0\rbrace$, $\delta_{p}(\lambda)$ is the~\textit{diagonal matrice} 
\begin{align*}
    \delta_{p}(\lambda)\colon
    \begin{array}{cll}
    V_{H}&\to & V_{H}\\
    \displaystyle \sum_{h\in H}\mu_{h}e_{h} &\mapsto &\displaystyle \sum_{h\neq p}\mu_{h}e_{h}+\lambda\mu_{p}e_{p}
    \end{array}
\end{align*}
and $\tau_{pq}(\lambda)$ is the \textit{transvection} 
\begin{align*}
    \tau_{pq}(\lambda)\colon
    \begin{array}{cll}
    V_{H}&\to &V_{H} \\
    \displaystyle \sum_{h\in H}\mu_{h}e_{h} &\mapsto &\displaystyle \sum_{h\neq p}\mu_{h}e_{h}+(\mu_{p}+\lambda\mu_{q})e_{p}
    \end{array}.
\end{align*}
Thus, thinking of an element $(\varphi,p)\in\cloner{H}$ as a labelling of $\text{Cay}(H,S_{H})$ (the vertex $h\in H$ has the label $\varphi(e_{h})\in V_{H}$), together with an arrow pointing at $p\in H$, right multiplying $(\varphi,p)$ by a generator from the above amounts either to move the arrow to an adjacent vertex $q$ of $p$ in $H$; or to keep the arrow where it stands and multiply $\varphi(e_{p})$ by a non-trivial element of $k$; or to keep the arrow where it stands and to~\textit{clone} the label $\varphi(e_{p})$ and add it to the label of a neighbour of $p$ after multiplication by an element of $\field\setminus\lbrace 0\rbrace$.

\smallskip

Lampcloners are naturally generated, have consistent blocks, and their lamp growth function is given by
\begin{equation*}
    \La_{\cloner{H}}(n)=\prod_{i=0}^{n-1}\left(|\field|^n-|\field|^i\right).
\end{equation*}
We will use the fact that $\ln{\La_{\cloner{H}}(n)}\sim Cn^2$ for some positive constant $C>0$.

\smallskip

Given a normal subgroup $M$, we analogously define the permutational lampcloner $\cloner{H,M}\defeq \text{FGL}(H/M)\rtimes H$, using the action $H\curvearrowright H/M$ by left translation. A (finite) generating set of this group is derived from a (finite) generating set of $H$ in a similar way.

\paragraph{Lampdesigners.} We finish this series of examples of halo products with lampdesigners, that we will not consider in this paper. Let $F$ and $H$ be two groups. The~\textit{lampdesigner over $H$} is the semi-direct product 
\begin{equation*}
    \designer{H} \defeq (F\wr_{H}\fsym{H})\rtimes H
\end{equation*}
where $H$ acts on $\bigoplus_{H}F$ by permuting the coordinates through its action on itself by left-multiplication and acts on $\fsym{H}$ as described above. It is the halo product $\halo H$ for the collection $L(S)\defeq F\wr_{S}\fsym{S}$, $S\subset H$. 

Lampdesigners are close from lampjuggler groups, and in fact if $F$ is finite, $\designer{H}$ is a subgroup of $\juggler{|F|}{H}$, via the map 
\begin{align*}
    \begin{array}{cll}
    \designer{H} &\to &\juggler{|F|}{H} \\
    ((f,\sigma),h) &\mapsto &(\sigma', h)
    \end{array}
\end{align*}
where, given a pair $(f,\sigma)\in F\wr_{H}\fsym{H}$, $\sigma'$ is the permutation of $H\times F$ given by $\sigma'(h,i)=(\sigma(h), f(h)i)$. Note also that $\designer{H}$ contains $\shuf{H}$ as a subgroup. 

\subsubsection{Isoperimetric profile of halo products}\label{sec:PrelIsoProfHalo}

This section describes what we know about isoperimetric profiles of halo products.

\smallskip

We say that a map $f\colon\R_{+}\to\R_{+}$ satisfies~\textit{Assumption~$(\star)$} if
\begin{equation*}
    \forall C>0,\; f(Cx)=O\left(f(x)\right).
\end{equation*}
We will often deal with finitely generated amenable groups $H$ whose isoperimetric profiles satisfie Assumption~$(\star)$. This condition also appeared in~\cite{Ers03, Corr24} and there is no known example of a group whose isoperimetric profile does not satisfy Assumption~$(\star)$. For instance, it applies for polynomial growth groups, solvable Baumslag-Solitar groups or polycyclic groups of exponential growth.

\smallskip

The following is a description of the F\o lner functions for wreath products, due to Erschler.

\begin{theorem}[{\cite[Theorem~1]{Ers03}}]
Let $\G$ and $H$ be finitely generated groups. Assume that $\G$ is not trivial and that $\prof{H}$ satisfies Assumption~$(\star)$. Then the following holds:
\begin{equation*}
    \textup{F\o l}_{\G\wr H}(n)\simeq \textup{F\o l}_{\G}(n)^{\textup{F\o l}_{H}(n)}.
\end{equation*}
\end{theorem}

Passing to the inverse, we deduce the isoperimetric profiles of wreath products. As examples:
\begin{itemize}
    \item $\prof{\G\wr H}(x)\simeq\prof{H}(\ln(x))$ if $\G$ is a non-trivial finite group;
    \item $\prof{\G\wr H}(x)\simeq\prof{H}\left (\frac{\ln(x)}{\ln(\ln(x))}\right )$ if $\G$ and $H$ have polynomial growth;
\end{itemize}

For lampshufflers over polynomial growth groups, it is proved in~\cite[Corollary~1.4]{EZ21} that
\begin{equation*}
    \prof{\shuf{H}}(x)\simeq\left(\frac{\ln(x)}{\ln(\ln(x))}\right)^{\frac{1}{k}}
\end{equation*}
where $k$ is the growth degree of $H$. For groups $H$ not having polynomial growth, and with mild assumptions, we recently proved the following.

\begin{theorem}[{\cite[Corollary~B]{cordum25}}]\label{thm:profileshufflercordum}
Let $H$ be a finitely generated amenable group with Assumption~$(\star)$. Assume that the isoperimetric profile $\prof{H}$ of $H$ satisfies
\begin{equation*}
   \prof{H}\left(\frac{\ln(x)}{\ln(\ln(x))}\right) \simeq \prof{H}(\ln(x)).
\end{equation*}
Then one has 
\begin{equation*}
    \prof{\shuf{H}}(x) \simeq \prof{H}(\ln(x)).
\end{equation*}
\end{theorem}

We also derived the isoperimetric profile of iterated lampshufflers.

\begin{theorem}[{\cite[Theorem~C]{cordum25}}]
Let $H$ be a finitely generated amenable group.
\begin{itemize}
    \item If $H$ has polynomial growth of degree $d\ge 1$, then one has
    \begin{equation*}
    \prof{\shufn{n}{H}}(x) \simeq \left(\frac{\ln^{\circ n}(x)}{\ln^{\circ (n+1)}(x)}\right)^{\frac{1}{d}}
\end{equation*}
for any $n\ge 1$.
    \item If its isoperimetric profile satisfies Assumption~$(\star)$ and $\prof{H}\left(\frac{x}{\ln(x)}\right) \simeq \prof{H}(x)$, then one has 
    \begin{equation*}
    \prof{\shufn{n}{H}}(x) \simeq \prof{H}(\ln^{\circ n}(x))
    \end{equation*}
for any $n\ge 1$.
\end{itemize}
\end{theorem}

Actually, we found estimates of the isoperimetric profile of more general halo products in a unified manner, including lampshufflers (and more generally lampjugglers), lampcloners or lampdesigners (see~\cite[Corollary~E]{cordum25} for the general statement).

\smallskip

For groups $H$ for which we know the isoperimetric profile of $\shuf{H}$, namely $H$ satisfies the assumptions of Theorem~\ref{thm:profileshufflercordum}, we also know the isoperimetric profile of the lampjugglers $\juggler{s}{H}$, $s\geq 1$. Indeed~\cite[Corollary~F]{cordum25} asserts that 
\begin{equation*}
    \prof{\juggler{s}{H}}(n)\simeq\prof{\shuf{H}}(n).
\end{equation*}

\smallskip

For lampcloners, the estimates are the following.

\begin{theorem}[{\cite[Theorem~G]{cordum25}}]
Let $n\geq 0$ be an integer. Let $H$ be a finitely generated amenable group. Let $\field$ be a field.
\begin{itemize}
    \item If $H$ has polynomial growth of degree $d\geq 1$, then
    \begin{equation}\label{eq:ClonerPolynomial}
        \left (\ln^{\circ n}(x)\right)^{\frac{1}{2d}}\preccurlyeq\prof{\clonern{n}{H}}(x)\preccurlyeq\left (\ln^{\circ n}(x)\right )^{\frac{1}{d}}.
    \end{equation}
    \item If the isoperimetric profile $\prof{H}$ of $H$ satisfies Assumption~$(\star)$ and $\prof{H}\left(\sqrt{x}\right)\simeq \prof{H}(x)$, then
    \begin{equation*}
        \prof{\clonern{n}{H}}(x)\simeq\prof{H}(\ln^{\circ n}(x)).
    \end{equation*}
\end{itemize}
\end{theorem}

\bigskip
\section{Quantitative orbit equivalence: stability results under iterations of halo products}\label{sec:stability}

\bigskip

\subsection{Orbit equivalence coupling between wreath products, by~\cite{DKLMT22}}\label{sec:generalmethodstability}

Let us start by explaining a construction of orbit equivalence between the first examples of halo products: lamplighter groups. Delabie, Koivisto, Le Maître and Tessera proved the following:

\begin{theorem}[{\cite[Corollary~7.3]{DKLMT22}}]\label{thm:dklmtWreath}
Let $\G$, $\La$, $H$ and $K$ be finitely generated groups and $\varphi,\psi\colon\R_{+}\to\R_{+}$ be increasing maps. Assume that there exist $(\varphi,\psi)$-integrable orbit equivalence couplings from $H$ to $K$, and from $\G$ to $\La$. Then there exists a $(\varphi,\psi)$-integrable orbit equivalence coupling from $\G\wr H$ to $\La\wr K$.
\end{theorem}

Here, we directly explain the construction they use within the appropriate formalism for our purposes, and we refer the reader to~\cite[Section~7.1]{DKLMT22} where the more general notion of wreath product of equivalence relations is introduced to prove Theorem~\ref{thm:dklmtWreath}. 

\smallskip

Given free pmp actions $H\curvearrowright (X,\mu)$ and $\G\curvearrowright (Y,\nu)$, we can define a pmp action of the group $\G\wr H$ on $(Y^{H}\times X,{\nu}^{\otimes H}\otimes \mu)$ in the following way. Let us denote by $\mathds{1}\colon H\to\G$ the constant map equal to $1_{\G}$. Then, for every $h\in H$, every finitely supported $f\colon H\to\G$ and every $((y_{g})_{g\in H},x)\in Y^{H}\times X$, we set
\begin{equation*}
    (\mathds{1},h)\cdot ((y_{g})_{g\in H},x)=((y_{gh})_{g\in H},h\cdot x)
\end{equation*}
and
\begin{equation*}
    (f,1_{H})\cdot ((y_{g})_{g\in H},x)=((f(g^{-1})\cdot y_{g})_{g\in H},x).
\end{equation*}
A direct computation shows that this provides a pmp action $\Lambda\wr H\curvearrowright (\Lambda^{H}\times X,{m_{\Lambda}}^{\otimes H}\otimes \mu)$.

\smallskip

Now, if $(X,\mu)$ and $(Y,\nu)$ are $(\varphi,\psi)$-integrable orbit equivalence couplings from $H$ to $K$ and from $\G$ to $\La$ respectively, and if $c_{H,K}\colon H\times X\to  K$, $c_{K,H}\colon K\times X\to  H$, $c_{\G,\La}\colon \G\times X\to  \La$, $c_{\La,\G}\colon \La\times X\to  \G$ are the associated cocycles, then the standard probability spaces $(Y^H\times X,\nu^{\otimes H}\otimes \mu)$ and $(Y^K\times X,\nu^{\otimes K}\otimes \mu)$ provide an orbit equivalence coupling between $\G\wr H$ and $\La\wr K$, via the identification
\begin{equation*}
    \begin{array}{ccl}
    Y^{H}\times X &\longrightarrow &Y^K\times X\\
    ((y_{h})_{h\in H},x)&\longmapsto & ((y_{c_{K,H}(k,x)})_{k\in K},x).
\end{array}
\end{equation*}
One can check directly that the latter is a measured isomorphism.

\smallskip

Additionally, we can describe the associated cocycle $c_{\G\wr H,\La\wr K}\colon (\G\wr H)\times (Y^H\times X)\to \La\wr K$ (and similarly $c_{\La\wr K,\G\wr H}\colon(\La\wr K)\times(Y^K\times X)\to \G\wr H$) on the generating set 
\begin{equation*}
    S_{\G\wr H}\defeq\lbrace(\delta^{H}_{\g},1_{H}) : \g\in\G\rbrace\cup\lbrace(\mathds{1},h) : h\in H\rbrace
\end{equation*}
of $\G\wr H$, where $\delta^{H}_{\g}\colon H\to\G$ is the map which sends $1_{H}$ to $\g$ and any $h\neq 1_{H}$ to $1_{\G}$. For every $\g\in\G$, we have
\begin{equation*}
    c_{\G\wr H,\La\wr K}\left((\delta^{H}_{\g},1_{H}),((y_{h})_{h\in H},x)\right)=\big(\delta^{K}_{c_{\G,\La}(\g,y_{1_H})},1_{K}\big)
\end{equation*}
and, for every $h\in H$,
\begin{equation*}
    c_{\G\wr H,\Lambda\wr K}\big((\mathds{1},h),((y_h)_{h\in H},x)\big)=\big(\mathds{1},c_{H,K}(h,x)\big).
\end{equation*}
This enables us to prove that this orbit equivalence coupling from $\G\wr H$ to $\La\wr K$ is $(\varphi,\psi)$-integrable.

\subsection{A general method}

Inspired by this construction of orbit equivalence couplings between lamplighters, we can build such couplings between some examples of halo products. Before presenting the applications for permutational lampjugglers and permutational lampcloners, let us explain informally the two main steps of the strategy. We fix normal subgroups $M\lhd H$ and $N\lhd K$.

\begin{enumerate}[label=(\arabic*)]
    \item\label{step1} First, given a free pmp action $H\curvearrowright (X,\mu)$, we want to build a standard probability space $(X_{H/M},\mu_{H/M})$ and a free pmp action $\halo_{H/M} H\curvearrowright (X_{H/M},\mu_{H/M})$;
    \item\label{step2} If now $(X,\mu)$ is an orbit equivalence coupling between the pairs $(H,M)$ and $(K,N)$, we want a natural identification between $(X_{H/M},\mu_{H/M})$ and $(X_{K/N},\mu_{K/N})$, using the pseudo-cocycles, in order to get an orbit equivalence coupling between $\halo_{H/M} H$ and $\halo_{K/N} K$.
\end{enumerate}

For step~\ref{step1}, to build a free pmp action $\halo_{H/M} H\curvearrowright (X_{H/M},\mu_{H/M})$ from a free pmp action $H\curvearrowright (X,\mu)$, we first need to find a free p.m.p action of $L(H/M)$ on some probability space $Z_{H/M}$ and then to define $X_{H/M}\defeq Z_{H/M}\times X$. Informally, $H/M$ must \enquote{appear} in this set $Z_{H/M}$. Let us be more explicit on this point:
\begin{itemize}
    \item First, we want the action of $H$ on $H/M$ to induce a free pmp $\halo_{H/M} H$-action on $Z_{H/M}$. For instance, we will choose $Z_{H/M}=Y^{H/M}$ for $\halo_{H/M} H=\G\wr_{H/M} H$ with the same notations as above. We will choose $Z_{H/M}=[0,1]^{H/M}$ for $\halo_{H/M} H=\shuf{H,M}$; $Z_{H/M}=[0,1]^{V_{H/M}}$ for $\halo_{H/M} H=\cloner{H,M}$ where the action of $H$ on $[0,1]^{V_{H/M}}$ is induced by the action of $H$ on $V_{H/M}$, which is itself induced by the actions of $H$ on $H/M$ by left-multiplication. We thus have an $H$- and a $L(H/M)$-action on $X_{H/M}\defeq Z_{H/M}\times X$, which motivate the following formula to deduce a $\halo_{H/M} H$-action:
    \begin{equation*}
        (\sigma,h)\cdot (z,x)=(\sigma\cdot (h\cdot z),h\cdot x)
    \end{equation*}
    for every $(\sigma,h)\in\halo_{H/M} H$ and every $(z,x)\in X_{H/M}$. For this formula to provide a well-defined $\halo_{H/M} H$-action, the $H$- and $L(H/M)$-actions on $Z_{H/M}$ must be compatible in a certain way, namely they must satisfy
    \begin{equation}\label{eq:compatible}
        h\cdot (\sigma\cdot z)=(h\cdot\sigma)\cdot (h\cdot z)
    \end{equation}
    for every $h\in H$, $\sigma\in L(H/M)$ and $z\in Z_{H/M}$. 
    \item The second reason for $H/M$ to \enquote{appear} in $Z_{H/M}$ is for step~\ref{step2}. To define a measured isomorphism between $Z_{H/M}\times X$ and $Z_{K/N}\times X$, we use the fact that for every $x\in X$, the pseudo-cocycle $\kappa_{K/N,H/M}(\cdot,x)\colon K/N\to H/M$ is a bijection and, in some sense, it allows us to go from $Z_{H/M}$ to $Z_{K/N}$, with $x$ being the second coordinate of an element of $Z_{H/M}\times X$. For instance, for the wreath products $\G\wr_{H/M} H$ and $\Lambda\wr_{K/N} K$, we define
    \begin{equation*}
        \begin{array}{ccl}
    Y^{H/M}\times X &\longrightarrow & Y^{K/N}\times X\\
    ((y_{\overline{h}})_{\overline{h}\in H/M},x)&\longmapsto & \big((y_{\kappa_{K/N,H/M}(\overline{k},x)})_{\overline{k}\in K/N},x\big)
    \end{array}
    \end{equation*}
    where again $(Y,\nu)$ denotes a coupling between $\G$ and $\La$ as described earlier. For lampshufflers, where $Z_{H/M}=[0,1]^{H/M}$, we will rather use the bijection $\kappa_{K/N,H/M}((\cdot)^{-1},x)^{-1}\colon K/N\to H/M$ and define
    \begin{equation*}
    \begin{array}{ccl}
    [0,1]^{H/M}\times X &\longrightarrow &[0,1]^{K/N}\times X\\
    ((\varepsilon_{\overline{h}})_{\overline{h}\in H/M},x)&\longmapsto & \big((\varepsilon_{\kappa_{K/N,H/M}(\overline{k}^{-1},x)^{-1}})_{\overline{k}\in K/N},x\big).
    \end{array}
    \end{equation*}
    More subtle will be the case of lampcloners, where $Z_{H/M}=[0,1]^{V_{H/M}}$. In this case, we will set
    \begin{equation*}
    \begin{array}{ccl}
    [0,1]^{V_{H/M}}\times X &\longrightarrow &[0,1]^{V_{K/N}}\times X\\
    ((\varepsilon_v)_{v\in V_{H/M}},x)&\longmapsto & \big((\varepsilon_{\varphi_{K/N,H/M}(w,x)})_{w\in V_{K/N}},x\big)
    \end{array}
    \end{equation*}
where $\varphi_{K/N,H/M}\colon V_{K/N}\to V_{H/M}$ is a bijection mapping $\sum_{\overline{k}\in K/N}{\mu_{\overline{k}}e_{\overline{k}}}\in V_{K/N}$ to $\sum_{\overline{h}\in H/M}{\mu_{\kappa_{H/M,K/N}(\overline{h}^{-1},x)^{-1}}e_{\overline{h}}}$ (again we use the bijection $H/M\to K/N$ provided by the pseudo-cocycle $\kappa_{H/M,K/N}$).
\end{itemize}

Finally, steps~\ref{step1} and~\ref{step2} being achieved, it will remain to check, case by case, that we indeed get an orbit equivalence coupling between $\halo_{H/M} H$ and $\halo_{K/N} K$, and to quantify the cocycles. This strategy will then provide similar statements as the one of Theorem~\ref{thm:dklmtWreath} for permutational lamplighters, lampjugglers and lampcloners.

\subsection{Permutational wreath products}

Similarly to Theorem~\ref{thm:dklmtWreath}, we prove the following.

\begin{theorem}\label{thm:stabilitypropertyforOEbetweenPWP}
Let $\G, \La,H,K$ be groups, with normal subgroups $M\lhd H$, $N\lhd K$. if there exist an orbit equivalence (resp.~$\ld^{\infty}$ orbit equivalence) coupling between the normal pairs $(H,M)$ and $(K,N)$, and an orbit equivalence (resp.~$\ld^{\infty}$ orbit equivalence) coupling between $\G$ and $\La$, then there exists an orbit equivalence (resp.~$\ld^{\infty}$ orbit equivalence) coupling between $\G\wr_{H/M}H$ and $\La\wr_{K/N}K$.

\smallskip

Moreover, given non-decreasing maps $\varphi,\psi\colon\R_+\to\R_+$, and assuming that $\G, \La, H$ and $K$ are finitely generated, if there exist a $(\varphi,\psi)$-integrable orbit equivalence coupling between the normal pairs $(H,M)$ and $(K,N)$, and a $(\varphi,\psi)$-integrable orbit equivalence coupling between $\G$ and $\La$, then there exists a $(\varphi,\psi)$-integrable orbit equivalence coupling between $\G\wr_{H/M}H$ and $\La\wr_{K/N}K$.
\end{theorem}

Recall from Section~\ref{sec:defHaloExample} that, if $S_H$ and $S_{\G}$ denote generating sets of $H$ and $\G$, then
\begin{equation*}
    S_{\G\wr_{H/M}H} \defeq \left\lbrace(\iota_{H/M}(t),1_{H}) : t\in S_{\G}\right\rbrace\cup\left\lbrace(\mathds{1},s) : s\in S_{H}\right\rbrace
\end{equation*}
is a generating set for $\G\wr_{H/M}H$, where for every $\g\in\G$, the map $\iota_{H/M}(\g)$ is defined by
\begin{equation*}
\begin{array}{llcl}
\iota_{H/M}(\g)\colon& H/M&\longrightarrow&\G\\
&\overline{h} &\longmapsto &\left\{
        \begin{array}{ll}
            \g&\text{ if }\overline{h}=1_{H/M}\\
            e_{\G}&\text{ otherwise}
        \end{array}
\right.
\end{array}.
\end{equation*}
Note that $S_{\G\wr_{H/M}H}$ is finite if and only if $S_{H}$ and $S_{\G}$ are finite.

\begin{proof}[Proof of Theorem~\ref{thm:stabilitypropertyforOEbetweenPWP}]
We must first define free pmp $\G\wr_{H/M}H$- and $\La\wr_{K/N}K$-actions, provided the $H$- and $K$-actions coming with a coupling $(X,\mu)$ between the pairs $(H,M)$ and $(K,N)$, and the $\G$- and $\La$-actions coming with a coupling $(Y,\nu)$. Denote $G_1=\G\wr_{H/M}H$ and $G_2=\La\wr_{K/N}K$. Following the techniques we explained in Section~\ref{sec:generalmethodstability}, we first define pmp $H$- and $\bigoplus_{H/M}{\G}$-actions on $\left(Y^{H/M},\nu^{\otimes H/M}\right)$ by
\begin{equation*}
    h\cdot (y_{\overline{g}})_{\overline{g}\in H/M}\defeq (y_{\overline{gh}})_{g\in H/M}
\end{equation*}
and
\begin{equation*}
    f\cdot(y_{\overline{g}})_{\overline{g}\in H/M}\defeq (f(\overline{g}^{-1})\cdot y_{\overline{g}})_{\overline{g}\in H/M}
\end{equation*}
for every $h\in H$, $f\in\bigoplus_{H/M}{\G}$ and $(\varepsilon_{\overline{g}})_{\overline{g}\in H/M}\in Y^{H/M}$. The compatibility condition~\eqref{eq:compatible} is satisfied, so the formula 
\begin{equation*}
(f,h)\cdot \left((y_{\overline{g}})_{\overline{g}\in H/M},x\right)\defeq \left ((f(g^{-1})\cdot y_{\overline{gh}})_{\overline{g}\in H/M},h\cdot x\right )
\end{equation*}
defines a free pmp $G_{1}$-action on $X_{H/M}\defeq Y^{H/M}\times X$, endowed with the probability measure $\mu_{H/M}\defeq\nu^{\otimes H/M}\otimes\mu$. We similarly define a free pmp $G_{2}$-action on 
\begin{equation*}
  (X_{K/N},\mu_{K/N})\defeq (Y^{K/N}\times X,\nu^{\otimes K/N)}\otimes\mu).
\end{equation*}
Now assume that $(X,\mu)$ is an orbit equivalence coupling between the normal pairs $(H,M)$ and $(K,N)$. We want to prove that $(X_{H/M},\mu_{H/M})$ and $(X_{K/N},\mu_{K/N})$ are orbit equivalent couplings between $G_{1}$ and $G_{2}$, via the identification
\begin{equation*}
    \theta\colon\begin{array}{ccl}
    Y^{H/M}\times X &\longrightarrow &Y^{K/N}\times X\\
    ((y_{\overline{h}})_{\overline{h}\in H/M},x)&\longmapsto & \big((y_{\kappa_{K/N,H/M}(\overline{k},x)})_{\overline{k}\in K/N},x\big)
\end{array}.
\end{equation*}
Given generating sets $S_{H}$ and $S_{K}$ of $H$ and $K$ (not necessarily finite), we get generating sets $S_{G_{1}}$ and $S_{G_{2}}$ of $G_{1}$ and $G_{2}$ respectively, as described after Theorem~\ref{thm:stabilitypropertyforOEbetweenPWP} (note that these sets are not necessarily finite for now). Let $((y_{\overline{h}})_{\overline{h}\in H/M},x)\in X_{H/M}$. Given $s\in S_{H}$, we have
\begin{align*}
    \theta\left((\mathds{1},s)\cdot ((y_{\overline{h}})_{\overline{h}\in H/M},x)\right)&=\theta\left((y_{\overline{hs}})_{\overline{h}\in H/M},s\cdot x\right)\\
    &=\big((y_{\kappa_{K/N,H/M}(\overline{k},s\cdot x)\overline{s}})_{\overline{k}\in K/N},s\cdot x\big).
\end{align*}
Notice that, by definition,
\begin{equation*}
    c_{K,H}(k,s\cdot x)s\cdot x=ks\cdot x=k\ c_{H,K}(s,x)\cdot x=c_{K,H}(k\ c_{H,K}(s,x),x)\cdot x
\end{equation*}
so $c_{K,H}(k,s\cdot x)s=c_{K,H}(k\ c_{H,K}(s,x),x)$ and in $H/M$:
\begin{equation*}
    \kappa_{K/N,H/M}(\overline{k},s\cdot x)\overline{s}=\kappa_{K/N,H/M}(\overline{k}\ \kappa_{H/M,K/N}(\overline{s},x),x)
\end{equation*}
where we use the fact that $\kappa_{H/M,K/N}(\overline{s},x)$ is the class of $c_{H,K}(s,x)$ modulo $N$, and the same fact for $\kappa_{K/N,H/M}$ (modulo $M$). This gives
\begin{align*}
    &\theta\left((\mathds{1},s)\cdot ((y_{\overline{h}})_{\overline{h}\in H/M},x)\right)\\
    &=\big((y_{\kappa_{K/N,H/M}(\overline{k}\ \kappa_{H/M,K/N}(\overline{s},x),x)})_{\overline{k}\in K/N},s\cdot x\big)\\
    &=\big((y_{\kappa_{K/N,H/M}(\overline{k}\ \overline{c_{H,K}(s,x)},x)})_{\overline{k}\in K/N},c_{H,K}(s,x)\cdot x\big)\\
    &=(\mathds{1},c_{H,K}(s,x))\cdot \big((y_{\kappa_{K/N,H/M}(\overline{k},x)})_{\overline{k}\in K/N},x\big)\\
    &=(\mathds{1},c_{H,K}(s,x))\cdot\theta\left((y_{\overline{h}})_{\overline{h}\in H/M},x\right).
\end{align*}

Moreover, given $t\in S_{\G_1}$, we have
\begin{align*}
    &\theta\left((\iota_{H/M}(t),1_{H})\cdot ((y_{\overline{h}})_{\overline{h}\in H/M},x)\right)\\
    &=\theta\left(((\iota_{H/M}(t)(\overline{h}^{-1})\cdot y_{\overline{h}})_{\overline{h}\in H/M},x)\right)\\
    &=\left((\iota_{H/M}(t)(\kappa_{K/N,H/M}(\overline{k},x)^{-1})\cdot y_{\kappa_{K/N,H/M}(\overline{k},x)})_{\overline{k}\in K/N},x\right).
\end{align*}
Depending on whether $\overline{k}=1_{K/N}$ or not, and using the fact that $\kappa_{K/N,H/M}(\overline{k},x)$ is equal to $1_{H/M}$ if and only if $\overline{k}=1_{K/N}$, we can prove that
\begin{equation*}
    \iota_{H/M}(t)(\kappa_{K/N,H/M}(\overline{k},x)^{-1})\cdot y_{\kappa_{K/N,H/M}(\overline{k},x)}=\iota_{K/N}(c_{\G,\La}(t,y_{1_{H/M}}))(\overline{k}^{-1})\cdot y_{\kappa_{K/N,H/M}(\overline{k},x)}.
\end{equation*}
This finally gives
\begin{align*}
    &\theta\left((\iota_{H/M}(t),1_{H})\cdot ((y_{\overline{h}})_{\overline{h}\in H/M},x)\right)\\
    &=\left((\iota_{K/N}(c_{\G,\La}(t,y_{1_{H/M}}))(\overline{k}^{-1})\cdot y_{\kappa_{K/N,H/M}(\overline{k},x)})_{\overline{k}\in K/N},x\right)\\
    &=\big(\iota_{K/N}(c_{\G,\La}(t,y_{1_{H/M}})),1_{K}\big)\cdot \left((y_{\kappa_{K/N,H/M}(\overline{k},x)})_{\overline{k}\in K/N},x\right)\\
    &=\big(\iota_{K/N}(c_{\G,\La}(t,y_{1_{H/M}})),1_{K}\big)\cdot\theta\left(((y_{\overline{h}})_{\overline{h}\in H/M},x)\right)
\end{align*}

\smallskip

Thus, we have proved $\theta(S_{G_1}\cdot z)\subset G_2\cdot\theta(z)$ for every $z\in X_{H/M}$ and we similarly get $S_{G_2}\cdot \theta(z)\subset \theta(G_1\cdot z)$, which is enough to prove the orbit equality for the $G_1$- and $G_2$-actions. The last computations in fact provide the cocycles on the generators: for every $s\in S_{H}$, one has
\begin{equation*}
    c_{G_1,G_2}\big((\mathrm{id},s),(y_{\overline{h}})_{\overline{h}\in H/M},x)\big)=(\mathrm{id},c_{H,K}(s,x))
\end{equation*}
and for every $t\in S_{\G}$,
\begin{equation*}
    c_{G_1,G_2}\big((\iota_{H/M}(t),1_H),(y_{\overline{h}})_{\overline{h}\in H/M},x)\big)=\big(\iota_{K/N}(c_{\G,\La}(t,y_{1_{H/M}})),1_{K}\big)
\end{equation*}
and similarly for the cocycle $c_{G_2,G_1}$. From this, we deduce that the coupling is $\ld^{\infty}$ if $(X,\mu)$ is an $\ld^{\infty}$ orbit equivalence coupling.

\smallskip

Let us finally assume that $H$ and $K$ are finitely generated (i.e. $S_{H}$ and $S_{K}$ are finite) and that $(X,\mu)$ is a $(\varphi,\psi)$-integrable orbit equivalence coupling from $(H,M)$ to $(K,N)$, and $(Y,\nu)$ a $(\varphi,\psi)$-integrable orbit equivalence coupling from $\G$ to $\La$. First, we easily get
\begin{equation*}
    \big|c_{G_1,G_2}\big((\mathrm{id},s),(y_{\overline{h}})_{\overline{h}\in H/M},x)\big)\big|_{G_2}=|c_{H,K}(s,x)|_{S_{K}}
\end{equation*}
and
\begin{equation*}
    \big|c_{G_1,G_2}\big((\iota_{H/M}(t),1_H),(y_{\overline{h}})_{\overline{h}},x)\big)\big|_{G_2}\le |c_{\G,\La}(t,y_{1_{H/M}})|_{S_{K}}
\end{equation*}
and similarly for $c_{G_2,G_1}$. Notice that the pushforward of $\mu_{H/M}$ by $((y_{\overline{h}})_{\overline{h}\in H/M},x)\mapsto x$ (resp.~$((y_{\overline{h}})_{\overline{h}\in H/M},x)\mapsto y_{1_{H/M}}$) is $\mu$ (resp.~$\nu$). Thus the $\varphi$- and $\psi$-integrabilities of $c_{H,K}$ and $c_{K,H}$, and for $c_{\G,\La}$ and $c_{\La,\G}$, directly imply the same properties for $c_{G_1,G_2}$ and $c_{G_2,G_1}$ on generators, and thus on all elements of the groups by Remark~\ref{rm:checkongenerators}. This proves that we get a $(\varphi,\psi)$-integrable orbit equivalence coupling from $\G\wr_{H/M}H$ to $\La\wr_{K/N}K$.
\end{proof}

\subsection{Permutational lampshufflers and lampjugglers}

Analogously to permutational wreath products, we prove the following.

\begin{theorem}\label{thm:stabilityofcouplings+quantificationLampjugglers}
Let $H$ and $K$ be groups, with normal subgroups $M\lhd H$ and $N\lhd K$, and let $r\ge 1$ be an integer. If the pairs $(H,M)$ and $(K,N)$ are orbit equivalent (resp.~$\ld^{\infty}$ orbit equivalent), then $\juggler{r}{H,M}$ and $\juggler{r}{K,N}$ are orbit equivalent (resp.~$\ld^{\infty}$ orbit equivalent).

\smallskip

Furthermore, given non-decreasing maps $\varphi,\psi\colon\R_{+}\to\R_{+}$, and assuming $H$ and $K$ are finitely generated, if there exists a $(\varphi,\psi)$-integrable orbit equivalence coupling from $(H,M)$ to $(K,N)$, then the same holds from $\juggler{r}{H,M}$ to $\juggler{r}{K,N}$.
\end{theorem}

For the proof of the second part of this theorem, we will need the following intermediate observation on word lengths of elements of $\juggler{r}{H}$. Recall from Section~\ref{sec:defHaloExample} that, if $S_{H}$ denotes a generating set, then
\begin{equation*}
    S_{\juggler{r}{H,M}} \defeq \left\lbrace(\tau_{(1_{H/M},i),(\overline{s},j)},1_{H}) : s\in S_{H}\cup S_{H}^{-1}, 1\le i,j\le r\right \rbrace\cup\left\lbrace(\text{id}_{H/M\times\lbrace1,\dots,r\rbrace},s) : s\in S_{H}\right\rbrace
\end{equation*}
is a generating set for $\juggler{r}{H}$, where $\overline{s}$ is the class of $s\in H$ modulo $M$. Note that $S_{H}$ is finite if and only if $S_{\juggler{r}{H}}$ is finite.

\begin{lemma}\label{lem:wordlengthinlampshuffler}
Let $H$ be a finitely generated group, with a finite generating set $S_{H}$. Let $M\lhd H$ be a normal subgroup. For every $g\in H$ and every integers $1\le i,j\le r$, we have
\begin{equation*}
    \big|(\tau_{(1_{H/M},i),(\overline{g},j)},1_{H})\big|_{S_{\juggler{r}{H,M}}}\le 4|g|_{S_{H}}.
\end{equation*}
\end{lemma}

\begin{proof}
We write $g=h_{1}\dots h_{n}$ with $n=|g|_{S_H}$ and $h_{1},\dots,h_{n}\in S_{H}\cup (S_{H})^{-1}$. We also set $\tau^{(\ell)}\defeq\tau_{(1_{H/M},i),(\overline{h_{\ell}\dots h_{n}},j)}$ and $\sigma^{(\ell)}\defeq\tau_{(1_{H/M},i),(\overline{h_{\ell}},i)}$, for any $1\le \ell\le n$. Then we have
\begin{equation*}
h_{\ell}\cdot\tau^{(\ell+1)}=\tau_{(\overline{h_{\ell}},i),(\overline{h_{\ell}\dots h_{n}},j)}
\end{equation*}
and
\begin{equation*}
    \sigma^{(\ell)}(h^{-1}_{\ell}\cdot\tau^{(\ell+1)})\sigma^{(\ell)}=\tau^{(\ell)},
\end{equation*}
which implies 
\begin{equation*}
    (\sigma^{(\ell)},1_{H})(\mathrm{id}_{H/M\times\{1,\dots, r\}},h_{\ell})(\tau^{(\ell+1)},1_H)(\mathrm{id}_{H/M\times\{1,\dots, r\}},h_{\ell}^{-1})(\sigma^{(\ell)},1_{H})=(\tau^{(\ell)},1_{H}).
\end{equation*}
Thus it follows that $|(\tau^{(\ell)},1_{H})|_{S_{\juggler{r}{H,M}}} \le 4+|(\tau^{(\ell+1)}, 1_{H})|_{S_{\juggler{r}{H,M}}}$, so by induction we get 
\begin{equation*}
\big|(\tau_{(1_{H/M},i),(\overline{g},j)},1_{H})\big|_{S_{\juggler{r}{H,M}}}=|(\tau^{(1)},1_{H})|_{S_{\juggler{r}{H,M}}}\le 4n=4|g|_{S_{H}}
\end{equation*}
and we are done.
\end{proof}

\begin{proof}[Proof of Theorem~\ref{thm:stabilityofcouplings+quantificationLampjugglers}]
We first define free pmp $\juggler{r}{H,M}$- and $\juggler{r}{K,N}$-actions, provided the $H$- and $K$-actions coming from a coupling $(X,\mu)$ between the pairs $(H,M)$ and $(K,N)$. Denote $G_{1}=\juggler{r}{H,M}$ and $G_{2}=\juggler{r}{K,N}$. Following the techniques we explained in Section~\ref{sec:generalmethodstability}, we first define pmp $H$- and $\fsym{H/M\times\lbrace1,\dots,r\rbrace}$-actions on $\left([0,1]^{H/M\times\lbrace 1,\ldots,r\rbrace},\mathrm{Leb}^{\otimes (H/M\times\lbrace 1,\ldots,r\rbrace)}\right)$ by
\begin{equation*}
    h\cdot (\varepsilon_{\overline{g},i})_{(\overline{g},i)\in H/M\times\lbrace 1,\dots,r\rbrace}\defeq (\varepsilon_{\overline{h^{-1}g},i})_{(g,i)\in H/M\times\lbrace1,\dots,r\rbrace}
\end{equation*}
and
\begin{equation*}
    \sigma\cdot(\varepsilon_{\overline{g},i})_{(\overline{g},i)\in H/M\times\lbrace 1,\dots,r\rbrace}\defeq (\varepsilon_{\sigma^{-1}(\overline{g},i)})_{(\overline{g},i)\in H/M\times\lbrace1,\dots,r\rbrace}
\end{equation*}
for every $h\in H$, $\sigma\in\fsym{H/M\times\lbrace 1,\dots,r\rbrace}$ and $(\varepsilon_{\overline{g},i})_{(\overline{g},i)\in H/M\times\lbrace1,\dots,r\rbrace}\in [0,1]^{H/M\times\lbrace1,\dots,r\rbrace}$. The compatibility condition~\eqref{eq:compatible} in Section~\ref{sec:generalmethodstability} is satisfied, so the formula 
\begin{equation*}
(\sigma,h)\cdot\left((\varepsilon_{\overline{g},i})_{(\overline{g},i)\in H/M\times\lbrace 1,\ldots,r\rbrace},x\right)\defeq \left((\varepsilon_{\overline{h}^{-1}\sigma^{-1}(\overline{g},i)})_{(\overline{g},i)\in H/M\times\lbrace1,\dots,r\rbrace},h\cdot x\right),
\end{equation*}
where $\overline{h}^{-1}\sigma^{-1}(\overline{g},i)=(\overline{h^{-1}g'},i')$ if $\sigma^{-1}(\overline{g},i)=(\overline{g'},i')$ (action of $H$ on $H/M\times\lbrace1,\dots,r\rbrace$), defines a free pmp $G_1$-action on $X_{H/M}\defeq [0,1]^{H/M}\times X$, endowed with the probability measure $\mu_{H/M}\defeq\mathrm{Leb}^{\otimes (H/M\times\{1,\ldots,r\})}\otimes\mu$. We similarly define a free pmp $G_{2}$-action on 
\begin{equation*}
  (X_{K/N},\mu_{K/N})\defeq \big([0,1]^{K/N\times\lbrace 1,\dots,r\rbrace}\times X,\mathrm{Leb}^{\otimes (K/N\times\lbrace 1,\ldots,r\rbrace)}\otimes\mu\big).
\end{equation*}
Now assume that $(X,\mu)$ is an orbit equivalence coupling between the normal pairs $(H,M)$ and $(K,N)$. We want to prove that $(X_{H/M},\mu_{H/M})$ and $(X_{K/N},\mu_{K/N})$ are orbit equivalent couplings between $G_{1}$ and $G_{2}$, via the identification
\begin{equation*}
    \theta\colon\begin{array}{ccl}
    [0,1]^{H/M\times\{1,\ldots,r\}}\times X &\longrightarrow &[0,1]^{K/N\times\lbrace1,\dots,r\rbrace}\times X\\
    \big((\varepsilon_{\overline{h},i})_{(\overline{h},i)\in H/M\times\lbrace1,\dots,r\rbrace},x\big)&\longmapsto & \big((\varepsilon_{\kappa_{K/N,H/M}(\overline{k}^{-1},x)^{-1},i})_{(\overline{k},i)\in K/N\times\lbrace1,\dots,r\rbrace},x\big)
\end{array}.
\end{equation*}
Given generating sets $S_{H}$ and $S_{K}$ of $H$ and $K$ (not necessarily finite), we get generating sets $S_{G_{1}}$ and $S_{G_{2}}$ of $G_{1}$ and $G_{2}$ respectively, as described before Lemma~\ref{lem:wordlengthinlampshuffler} (note that these sets are not necessarily finite for now). Let $((\varepsilon_{\overline{h},i})_{(\overline{h},i)\in H/M\times\lbrace1,\dots,r\rbrace},x)\in X_{H/M}$. Given $s\in S_{H}$, we have
\begin{align*}
    \theta\left((\mathrm{id},s)\cdot ((\varepsilon_{\overline{h},i})_{(\overline{h},i)\in H/M\times\lbrace1,\dots,r\rbrace},x)\right)&=\theta\left((\varepsilon_{\overline{s}^{-1}\overline{h},i})_{(\overline{h},i)\in H/M\times\lbrace1,\dots,r\rbrace},s\cdot x\right)\\
    &=\big((\varepsilon_{\overline{s}^{-1}\kappa_{K/N,H/M}(\overline{k}^{-1},s\cdot x)^{-1},i})_{(\overline{k},i)\in K/N\times\lbrace1,\dots,r\rbrace},s\cdot x\big).
\end{align*}
The indices are equal to
\begin{align*}
    \overline{s}^{-1}\kappa_{K/N,H/M}(\overline{k}^{-1},s\cdot x)^{-1}&=\left (\kappa_{K/N,H/M}(\overline{k}^{-1},s\cdot x)\overline{s}\right )^{-1}\\
    &=\kappa_{K/N,H/M}\big(\overline{k}^{-1}\kappa_{H/M,K/N}(\overline{s},x),x\big)^{-1}\\
    &=\kappa_{K/N,H/M}\big((\kappa_{H/M,K/N}(\overline{s},x)^{-1}\overline{k})^{-1},x\big)^{-1}
\end{align*}
where we use the fact that $\kappa_{H/M,K/N}(\overline{s},x)$ is the class of $c_{H,K}(s,x)$ modulo $N$. This gives
\begin{align*}
    &\theta\left((\mathrm{id},s)\cdot ((\varepsilon_{\overline{h},i})_{(\overline{h},i)\in H/M\times\lbrace1,\dots,r\rbrace},x)\right)\\
    &=\big((\varepsilon_{\kappa_{K/N,H/M}(\overline{c_{H,K}(s,x)}^{-1}\overline{k})^{-1},x)^{-1},i}\big)_{(\overline{k},i)\in K/N\times\lbrace1,\dots,r\rbrace},c_{H,K}(s,x)\cdot x)\\
    &=(\mathrm{id},c_{H,K}(s,x))\cdot \big((\varepsilon_{\kappa_{H/M,K/N}(\overline{k}^{-1},x)^{-1},i})_{(\overline{k},i)\in K/N\times\lbrace1,\dots,r\rbrace},x\big)\\
    &=(\mathrm{id},c_{H,K}(s,x))\cdot\theta \left((\varepsilon_{\overline{h},i})_{(\overline{h},i)\in H/M\times\lbrace1,\dots,r\rbrace},x\right).
\end{align*}
For the sequel, we set the following notations: $\kappa_{K/N,H/M}((\overline{k},i),x)\defeq (\kappa_{K/N,H/M}(\overline{k},x),i)$ and $(\overline{k},i)^{-1}\defeq (\overline{k}^{-1},i)$. Then, for every $s\in S_{H}\cup S_{H}^{-1}$ and $j,\ell\in\lbrace1,\ldots,r\rbrace$, denoting $\tau=\tau_{(1_{H/M},j),(\overline{s},\ell)}$ (satisfying $\tau^{-1}=\tau$), we have
\begin{align*}
    &\theta\left((\tau,1_{H})\cdot ((\varepsilon_{\overline{h},i})_{(\overline{h},i)\in H/M\times\lbrace1,\dots,r\rbrace},x)\right)\\
    &=\theta\left( (\varepsilon_{\tau(\overline{h},i)})_{(\overline{h},i)\in H/M\times\lbrace1,\dots,r\rbrace},x\right)\\
    &=\big((\varepsilon_{\tau(\kappa_{K/N,H/M}(\overline{k}^{-1},x)^{-1},i)})_{(\overline{k},i)\in K/N\times\lbrace1,\dots,r\rbrace},x\big).
\end{align*}
Depending on whether or not $(\overline{k},i)$ lies in $\lbrace(1_{K/N},j),(\kappa_{H/M,K/N}(\overline{s}^{-1},x)^{-1},\ell)\rbrace$, we can prove that
\begin{equation*}
    \tau\big(\kappa_{K/N,H/M}(\overline{k}^{-1},x)^{-1},i\big)=\kappa_{K/N,H/M}\big([\tau_{(1_{K},j),(\kappa_{H/M,K/N}(\overline{s}^{-1},x)^{-1},\ell)}(\overline{k},i)]^{-1},x\big)^{-1}.
\end{equation*}
This finally gives
\begin{align*}
    &\theta\left((\tau,1_{H})\cdot ((\varepsilon_{\overline{h},i})_{(\overline{h},i)\in H/M\times\lbrace1,\dots,r\rbrace},x)\right)\\
    &=(\tau_{(1_{K/N},j),(\kappa_{H/M,K/N}(\overline{s}^{-1},x)^{-1},\ell)},1_K)\cdot \big((\varepsilon_{\kappa_{K/N,H/M}(\overline{k}^{-1},x)^{-1},i})_{(\overline{k},i)\in K/N\times\lbrace1,\dots,r\rbrace},x\big)\\
    &=(\tau_{(1_{K/N},j),(\kappa_{H/M,K/N}(\overline{s}^{-1},x)^{-1},\ell)},1_K)\cdot \theta\left ((\varepsilon_{\overline{h},i})_{(\overline{h},i)\in H/M\times\lbrace1,\dots,r\rbrace},x\right).
\end{align*}

\smallskip

Thus, we have proved $\theta(S_{G_{1}}\cdot z)\subset G_{2}\cdot\theta(z)$ for every $z\in X_{H/M}$ and we similarly get $S_{G_{2}}\cdot \theta(z)\subset \theta(G_1\cdot z)$, which is enough to prove the orbit equality for the $G_{1}$- and $G_{2}$-actions. The last computations in fact provide the cocycles on the generators: for every $s\in S_{H}\cup S_{H}^{-1}$, one has
\begin{equation*}
    c_{G_{1},G_{2}}\big((\mathrm{id},s),((\varepsilon_{\overline{h},i})_{(\overline{h},i)\in H/M\times\{1,\ldots,r\}},x)\big)=(\mathrm{id},c_{H,K}(s,x))
\end{equation*}
and
\begin{equation*}
    c_{G_{1},G_{2}}\big((\tau_{(1_{H/M},j),(\overline{s},\ell)},1_H),((\varepsilon_{\overline{h},i})_{(\overline{h},i)\in H/M\times\{1,\ldots,r\}},x)\big)=\big(\tau_{(1_{K/N},j),(\kappa_{H/M,K/N}(\overline{s}^{-1},x)^{-1},\ell)},1_{K}\big)
\end{equation*}
and similarly for the cocycle $c_{G_{2},G_{1}}$. From this, we deduce that the coupling is $\ld^{\infty}$ if $(X,\mu)$ is an $\ld^{\infty}$ orbit equivalence coupling.

\smallskip

Let us finally assume that $H$ and $K$ are finitely generated (i.e. $S_{H}$ and $S_{K}$ are finite) and that $(X,\mu)$ is a $(\varphi,\psi)$-integrable orbit equivalence coupling from $(H,M)$ to $(K,N)$. First, we easily get
\begin{equation*}
    \big|c_{G_{1},G_{2}}\big((\mathrm{id},s),((\varepsilon_{\overline{h},i})_{(\overline{h},i)\in H/M\times\{1,\ldots,r\}},x)\big)\big|_{G_{2}}=|c_{H,K}(s,x)|_{S_{K}}
\end{equation*}
and
\begin{equation*}
    \big|c_{G_{1},G_{2}}\big((\tau_{(1_{H/M},j),(\overline{s},\ell)},1_H),((\varepsilon_{\overline{h},i})_{(\overline{h},i)\in H/M\times\{1,\ldots,r\}},x)\big)\big|_{G_{2}}\le  4|c_{H,K}(s^{-1},x)|_{S_{K}},
\end{equation*}
using Lemma~\ref{lem:wordlengthinlampshuffler}, and similarly for $c_{G_{2},G_{1}}$. Thus the $\varphi$- and $\psi$-integrabilities of $c_{H,K}$ and $c_{K,H}$ directly imply the same properties for $c_{G_{1},G_{2}}$ and $c_{G_{2},G_{1}}$ on generators, and thus on all elements of the groups by Remark~\ref{rm:checkongenerators}, so we get a $(\varphi,\psi)$-integrable orbit equivalence coupling from $\juggler{r}{H,M}$ to $\juggler{r}{K,N}$.
\end{proof}

\subsection{Permutational lampcloners}

Analogously to permutational wreath products and lampjugglers, we prove the following.

\begin{theorem}\label{thm:stabilityofcouplings+quantificationLampcloners}
Let $H$ and $K$ be groups, with normal subgroups $M\lhd H$ and $N\lhd K$, and let $\field$ be a finite field. If the normal pairs $(H,M)$ and $(K,N)$ are orbit equivalent (resp.~$\ld^{\infty}$ orbit equivalent), then $\cloner{H,M}$ and $\cloner{K,N}$ are orbit equivalent (resp.~$\ld^{\infty}$ orbit equivalent).

\smallskip

Furthermore, given non-decreasing maps $\varphi,\psi\colon\R_{+}\to\R_{+}$, and assuming $H$ and $K$ are finitely generated, if there exists a $(\varphi,\psi)$-integrable orbit equivalence coupling from $(H,M)$ to $(K,N)$, then the same holds from $\cloner{H,M}$ to $\cloner{K,N}$.
\end{theorem}

For the proof, we will need the following intermediate observation on word lengths of elements of $\cloner{H,M}$. Recall from Section~\ref{sec:defHaloExample} that, if $S_{H}$ denotes a generating set, then
\begin{equation*}
    S_{\cloner{H,M}}\defeq\lbrace (\delta_{1_{H/M}}(\lambda), 1_{H}) : \lambda \in \field\setminus\lbrace 0\rbrace\rbrace\cup\lbrace (\tau_{1_{H/M},\overline{s}}(\lambda),1_{H}) : s\in S_{H}\cup S_{H}^{-1}\rbrace \cup\lbrace (\text{id}, s) : s\in S_{H}\rbrace
\end{equation*}
is a generating set for $\cloner{H,M}$, constructed from a finite generating set $S_{H}$ of $H$. Note that $S_{H}$ is finite if and only if $S_{\cloner{H,M}}$ is finite.

\begin{lemma}\label{lem:wordlengthinlampcloner}
Let $H$ be a finitely generated group. For every $g\in H$ and every $\lambda\in\field$, we have
\begin{equation*}
    \big|(\tau_{1_{H/M},\overline{g}}(\lambda),1_{H})\big|_{S_{\cloner{H}}}\le 14|g|_{S_{H}}.
\end{equation*}
\end{lemma}

\begin{proof}
We write $g=h_{1}\dots h_{n}$ with $n=|g|_{S_H}$ and $h_{1},\dots,h_{n}\in S_{H}\cup S_{H}^{-1}$. We also set, for any $1\le \ell\le n$, $\tau^{(\ell)}\defeq\tau_{1_{H/M},\overline{h_{\ell}\dots h_{n}}}(\lambda)$ and $\sigma^{(\ell)}$ the linear automorphism which acts on the canonical basis of $V_{H}$ in the following way: it swaps $e_{1_{H/M}}$ and $e_{\overline{h_{\ell}}}$, and it fixes $e_{\overline{h}}$ for any other $\overline{h}\in (H/M)\setminus\lbrace 1_{H/M},\overline{h_{\ell}}\rbrace$. Then we have
\begin{equation*}
h_{\ell}\cdot\tau^{(\ell+1)}=\tau_{\overline{h_{\ell}},\overline{h_{\ell}\dots h_{n}}}(\lambda)
\end{equation*}
and
\begin{equation*}
    \sigma^{(\ell)}(h_{\ell}\cdot\tau^{(\ell+1)})\sigma^{(\ell)}=\tau^{(\ell+1)}
\end{equation*}
which implies
\begin{equation*}
    (\sigma^{(\ell)},1_{H})(\mathrm{id}_{H/M},h_{\ell})(\tau^{(\ell+1)},1_{H})(\mathrm{id}_{H/M},h_{\ell}^{-1})(\sigma^{(\ell)},1_{H})=(\tau^{(\ell+1)},1_{H}).
\end{equation*}
Moreover, we have $|(\sigma^{(\ell)},1_{H})|_{\cloner{H,M}}\le 6$. Indeed, $\sigma^{(\ell)}$ can be written as $\delta_{1_{H/M}}(-1)\tau_{1_{H/M},\overline{h_{\ell}}}(-1)\tau_{\overline{h_{\ell}},1_{H/M}}(1)\tau_{1_{H/M},\overline{h_{\ell}}}(-1)$, with $\tau_{\overline{h_{\ell}},1_{H/M}}(1)=h_{\ell}\cdot\tau_{1_{H/M},\overline{h_{\ell}}^{-1}}(1)$, so we get
\begin{equation*}
    (\sigma^{(\ell)},1_{H})=(\delta_{1_{H/M}}(-1),1_{H})(\tau_{1_{H/M},\overline{h_{\ell}}}(-1),1_{H})(\mathrm{id},h_{\ell})(\tau_{1_{H/M},\overline{h_{\ell}}^{-1}}(1)(\mathrm{id},h_{\ell}^{-1})(\tau_{1_{H/M},\overline{h_{\ell}}}(-1),1_{H})
\end{equation*}
and it has length $\le 6$.

\smallskip

Thus it follows that $|(\tau^{(\ell)},1_{H})|_{S_{\cloner{H,M}}} \le 14+|(\tau^{(\ell+1)}, 1_{H})|_{S_{\cloner{H,M}}}$, and by induction we get \begin{equation*}
\big|(\tau_{1_{H/M},\overline{g}}(\lambda), 1_{H})\big|_{S_{\cloner{H,M}}}=|(\tau^{(1)},1_{H})|_{S_{\cloner{H,M}}}\le 14n=14|g|_{S_{H}}
\end{equation*}
as claimed.
\end{proof}

\begin{proof}[Proof of Theorem~\ref{thm:stabilityofcouplings+quantificationLampcloners}]
Given generating sets $S_{H}$ and $S_{K}$ (not necessarily finite) of $H$ and $K$, we get generating subsets $S_{\cloner{H,M}}$ and $S_{\cloner{K,N}}$ of $\cloner{H,M}$ and $\cloner{K,N}$, as described before Lemma~\ref{lem:wordlengthinlampcloner}. We denote $G_{1}=\cloner{H,M}$ and $G_{2}=\cloner{H,N}$.

\smallskip
    
We first define  pmp $G_{1}$- and $G_{2}$-actions provided the $H$- and $K$-actions on $(X,\mu)$ coming with a coupling $(X,\mu)$ between the pairs $(H,M)$ and $(K,N)$. As in the proof of Therem~\ref{thm:stabilityofcouplings+quantificationLampjugglers}, we follow the techniques we explained in Section~\ref{sec:generalmethodstability}. Let us set
\begin{equation*}
    (X_{H/M},\mu_{H/M})\defeq ([0,1]^{V_{H/M}}\times X,\mathrm{Leb}^{\otimes V_{H/M}}\otimes\mu).
\end{equation*}
Then $H$ and $\mathrm{FGL}(H/M)$ act on $[0,1]^{V_{H/M}}$ in the following way:
\begin{equation*}
    \sigma\cdot(\varepsilon_v)_{v\in V_{H/M}}\defeq (\varepsilon_{\sigma^{-1}(v)})_{v\in V_{H/M}}
\end{equation*}
and
\begin{equation*}
        h\cdot (\varepsilon_v)_{v\in V_{H/M}}\defeq (\varepsilon_{h\cdot v})_{v\in V_{H/M}}
\end{equation*}
for every $\sigma\in\mathrm{FGL}(H/M)$, $h\in H$ and $(\varepsilon_v)_{v\in V_{H/M}}\in [0,1]^{V_{H/M}}$, using the action of $H$ on $V_{H/M}$ given by
\begin{equation*}
    h\cdot\sum_{\overline{g}\in H/M}{\mu_{\overline{g}} e_{\overline{g}}}\defeq\sum_{\overline{g}\in H/M}{\mu_{\overline{g}}e_{\overline{h^{-1}g}}}.
\end{equation*}
These $H$- and $\mathrm{FGL}(H/M)$-actions are compatible in the sense of~\eqref{eq:compatible} (see Section~\ref{sec:generalmethodstability}), and we get a free pmp $G_{1}$-action on $(X_{H/M},\mu_{H/M})$ defined by
\begin{equation*}
    (\sigma,h)\cdot ((\varepsilon_v)_{v\in V_{H/M}},x)\defeq \big((\varepsilon_{h\cdot\sigma^{-1}(v)})_{v\in V_{H/M}},h\cdot x\big).
\end{equation*}
We similarly define a free pmp $G_{2}$-action on $(X_{K/N},\mu_{K/N})\defeq ([0,1]^{V_{K/N}}\times X,\mathrm{Leb}^{\otimes V_{K/N}}\otimes \mu)$. Now assume that $(X,\mu)$ is an orbit equivalence coupling between the normal pairs $(H,M)$ and $(K,N)$. We want to prove that $(X_{H/M},\mu_{H/M})$ and $(X_{K/N},\mu_{K/N})$ are orbit equivalent couplings between $G_{1}$ and $G_{2}$, via the identification 
\begin{equation*}
    \theta\colon\begin{array}{ccl}
    [0,1]^{V_{H/M}}\times X &\longrightarrow &[0,1]^{V_{K/N}}\times X\\
    ((\varepsilon_{v})_{v\in V_{H/M}},x)&\longmapsto & \big((\varepsilon_{\varphi_{K/M,H/N}(w,x)})_{w\in V_{K/N}},x\big)
\end{array},
\end{equation*}
where $\varphi_{K/M,H/N}\colon V_{K/N}\to V_{H/M}$ is a bijection mapping $\sum_{\overline{k}\in K/N}{\mu_{\overline{k}}e_{\overline{k}}}\in V_{K/N}$ to $\sum_{\overline{h}\in H/M}{\mu_{\kappa_{H/M,K/N}(\overline{h}^{-1},x)^{-1}}e_{\overline{h}}}$. For every $s\in S_{H}\cup S_{H}^{-1}$, one has
\begin{equation*}
    \theta\left((\mathrm{id},s)\cdot ((\varepsilon_{v})_{v\in V_{H/M}},x)\right)=(\mathrm{id},c_{H,K}(s,x))\cdot\theta ((\varepsilon_{v})_{v\in V_{H/M}},x)
\end{equation*}
as well as
\begin{equation*}
    \theta\left((\tau_{1_{H/M},\overline{s}}(\lambda),1_{H})\cdot ((\varepsilon_{v})_{v\in V_{H/M}},x)\right)=(\tau_{1_{K/N},\kappa_{H/M,K/N}(\overline{s}^{-1},x)^{-1}},1_{K})\cdot \theta\left(((\varepsilon_{v})_{v\in V_{H/M}},x)\right)
\end{equation*}
and
\begin{equation*}
    \theta\left((\delta_{1_{H/M}}(\lambda),1_{H})\cdot ((\varepsilon_{v})_{v\in V_{H/M}},x)\right)=(\delta_{1_{K/N}}(\lambda),1_{K})\cdot \theta\left(((\varepsilon_{v})_{v\in V_{H/M}},x)\right).
\end{equation*}

\smallskip

We thus have proved $\theta(S_{G_{1}}\cdot z)\subset G_{2}\cdot\theta(z)$ for every $z\in X_{H/M}$ and we similarly get $S_{G_{2}}\cdot \theta(z)\subset \theta(G_{1}\cdot z)$, so we have built an orbit equivalence coupling. We have also identified the cocycles $c_{G_{1},G_{2}}$, $c_{G_{2},G_{1}}$ on the generators, and the quantitative refinments of the statement are proved similarly to proof of Theorem~\ref{thm:stabilityofcouplings+quantificationLampjugglers}, using this time Lemma~\ref{lem:wordlengthinlampcloner}.
\end{proof}

\subsection{Applications to quantitatively optimal orbit equivalence couplings}

We now move on to applications of Theorems~\ref{thm:stabilitypropertyforOEbetweenPWP},~\ref{thm:stabilityofcouplings+quantificationLampjugglers} and~\ref{thm:stabilityofcouplings+quantificationLampcloners}.

\smallskip

We first focus on the $\ld^{\infty}$ condition. Recall that, by Shalom~\cite{Sha04}, two amenable groups are $\ld^{\infty}$ orbit equivalent if and only if they are bi-Lipschitz equivalent. For pairs, $\ld^{\infty}$ orbit equivalence boils down to the existence of a bi-Lipschitz equivalence between the ambiant groups, mapping cosets to cosets~\cite{CDT26}. We thus deduce:

\begin{corollary}\label{cor:PreservationBiLipHalo}
Let $K$ and $H$ be amenable groups, with normal subgroups $M\lhd H$ and $N\lhd K$. Assume that there exists a bi-Lipschitz equivalence between $K$ and $H$, mapping $M$-cosets to $N$-cosets. Then:
    \begin{itemize}
        \item if $\G$ and $\La$ are amenable groups that are bi-Lipschitz equivalent, then $\G\wr_{H/M}H$ and $\La\wr_{K/N}K$ are bi-Lipschitz equivalent;
        \item for every $r\ge 1$, $\juggler{r}{H,M}$ and $\juggler{r}{K,N}$ are bi-Lipschitz equivalent;
        \item for every finite field $\field$, $\cloner{H,M}$ and $\cloner{K,N}$ are bi-Lipschitz equivalent.
    \end{itemize}
\end{corollary}

\begin{remark}\label{rem:PreservationBiLipHalo}
Applying Corollary~\ref{cor:PreservationBiLipHalo} to $M=\lbrace 1_{H}\rbrace$ and $N=\lbrace 1_{K}\rbrace$ (namely the classical notion of orbit equivalence), we know that if $H$ and $K$ are bi-Lipschitz equivalent, then so are $\G\wr H$ and $\La\wr K$ (if $\G$ and $\La$ are also bi-Lipschitz equivalent), $\juggler{r}{H}$ and $\juggler{r}{K}$, and $\cloner{H}$ and $\cloner{K}$. Then we can partly recover the results of Genevois and Tessera (see e.g.~Theorem~\ref{th:QIrigidityShufGT24} at the beginning of the introduction).\par
For pairs with non-trivial normal subgroups, a bi-Lipschitz equivalence between $H$ and $K$, mapping $M$-cosets to $N$-cosets is a particular instance of~\textit{quasi-isometry of pairs} studied by the second-named author to get a sufficient condition for permutational wreath products to be bi-Lipschitz equivalent, see~\cite[Proposition~6.1]{Dum24}. Corollary~\ref{cor:PreservationBiLipHalo} thus enables us to partly recover this classification result.
\end{remark}

For the part of our statements about $(\varphi,\psi)$-integrability, we give applications when $\G_{1}$ and $\G_{2}$ are free abelian groups.

\begin{corollary}\label{cor:stabilityQOEHaloPermZd}
Let $d_{1},d_{2},k_{1},k_{2}$ be positive integers such that $0\le k_{1}\le d_{1}$ and $0\le k_{2}\le d_{2}$. Let $M$ and $N$ be normal subgroups of $\Z^{d_{1}}$ and $\Z^{d_{2}}$, with ranks $k_{1}$ and $k_{2}$ respectively. We assume that $k_{1}=0$ if and only if $k_{2}=0$, and that $k_{1}=d_{1}$ if and only if $k_{2}=d_{2}$, and if the latter holds, then $|\G_{1}/N_{1}|=|\G_{2}/N_{2}|$.
    
\smallskip

Let $P\defeq\min{\left (\frac{k_{2}}{k_{1}},\frac{d_{2}-k_{2}}{d_{1}-k_{1}}\right)}$ and $Q\defeq\min{\left (\frac{k_{1}}{k_{2}},\frac{d_{1}-k_{1}}{d_{2}-k_{2}}\right)}$ if $0<k_{1}<d_{1}$ and $0<k_{2}<d_{2}$, and $P\defeq\frac{d_{2}}{d_{1}}$ and $Q\defeq\frac{d_{1}}{d_{2}}$ if $k_{1}=k_{2}=0$ or $d_{1}-k_{1}=d_{2}-k_{2}=0$. Then:
\begin{itemize}
    \item for every finitely generated groups $\G$ and $\La$, if there exists an $(\ld^{<P},\ld^{<Q})$-integrable orbit equivalence from $\G$ to $\La$, then there exists an $(\ld^{<P},\ld^{<Q})$-integrable orbit equivalence from $\G\wr_{\Z^{d_{1}}/M}\Z^{d_{1}}$ to $\La\wr_{\Z^{d_{2}}/N}\Z^{d_{2}}$;
    \item for every $r\ge 1$, there exists an $(\ld^{<P},\ld^{<Q})$-integrable orbit equivalence from $\juggler{r}{\Z^{d_{1}},M}$ to $\juggler{r}{\Z^{d_{2}},N}$;
    \item for every finite field $\field$, there exists an $(\ld^{<P},\ld^{<Q})$-integrable orbit equivalence from $\cloner{\Z^{d_{1}},M}$ to $\cloner{\Z^{d_{2}},N}$.
\end{itemize}
\end{corollary}

\begin{proof}
We combine Theorems~\ref{thm:stabilitypropertyforOEbetweenPWP},~\ref{thm:stabilityofcouplings+quantificationLampjugglers} and~\ref{thm:stabilityofcouplings+quantificationLampcloners} with either~\cite[Theorem~6.12]{DKLMT22} (for trivial subgroups $M$ and $N$), or the classification in Proposition~\ref{prop:OEofpairsfreeabeliangroups}, carried out in~\cite{CDT26}, for pairs when $M$ and $N$ are not trivial.
\end{proof}

We can moreover prove that the integrability conditions of the above corollary are sharp in many cases. As examples:

\begin{corollary}\label{cor:ClassificationQOEPermutationalZd}
Let $d_{1},d_{2},k_{1},k_{2}$ be positive integers such that $0\le k_{1}< d_{1}$, $0\le k_{2}<d_{2}$ and $\frac{d_{2}-k_{2}}{d_{1}-k_{1}}<1$, and such that $k_{1}=0$ if and only if $k_{2}=0$. Let $N_{1}$ and $N_{2}$ be normal subgroups of $\Z^{d_{1}}$ and $\Z^{d_{2}}$, with rank $k_{1}$ and $k_{2}$ respectively. In the case where $k_{1}>0$ and $k_{2}>0$, assume that $\frac{d_{2}-k_{2}}{d_{1}-k_{1}}<\frac{k_{2}}{k_{1}}$. Then
    \begin{itemize}
        \item if $\G$ is either a finite group or a polynomial growth group, then there exists an $(\ld^{p},\ld^{0})$-integrable orbit equivalence from $\G\wr_{\Z^{d_{1}}/M}\Z^{d_{1}}$ to $\G\wr_{\Z^{d_{2}}/N}\Z^{d_{2}}$ if and only if $p<\frac{d_{2}-k_{2}}{d_{1}-k_{1}}$;
        \item for every $r\ge 1$, there exists an $(\ld^{p},\ld^{0})$-integrable orbit equivalence from $\juggler{r}{\Z^{d_{1}},M}$ to $\juggler{r}{\Z^{d_{2}},N}$ if and only if $p<\frac{d_{2}-k_{2}}{d_{1}-k_{1}}$.
    \end{itemize}
\end{corollary}

\begin{proof}
The \enquote{if} parts follow from Corollary~\ref{cor:stabilityQOEHaloPermZd}.
    
For \enquote{only if} part, we first need to know the isoperimetric profiles of the groups appearing in the statement we want to prove. In the extreme case where $k_{1}=k_{2}=0$, we directly use~\cite{Ers03} for wreath products and~\cite{cordum25} for lampjugglers. In the case $0<k_{1}<d_{1}$, Corollary~\ref{cor:PreservationBiLipHalo} implies that $\G\wr_{\Z^{d_{1}}/M}\Z^{d_{1}}$ (resp.~$\juggler{r}{\Z^{d_{1}},M}$) has the same isoperimetric profile as $\G\wr_{\Z^{d_{1}-k_{1}}}\Z^{d_{1}}\simeq (\G\wr\Z^{d_{1}-k_{1}})\times\Z^{k_{1}}$ (resp.~$\juggler{r}{\Z^{d_{1}},\Z^{k_{1}}}\simeq\juggler{r}{\Z^{d_{1}-k_{1}}}\times \Z^{k_{1}}$), so it has the same isoperimetric profile as $\G\wr\Z^{d_{1}-k_{1}}$ (resp.~$\juggler{r}{\Z^{d_{1}-k_{1}}}$), see e.g.~\cite[Lemma~6.8]{cordum25}.
    
Once we know the isoperimetric profiles, the obstructions provided by Delabie, Koivisto, Le Maître and Tessera, and by the first-named author, see Theorems~\ref{thm:dklmt} and~\ref{thm:threshold} in Section~\ref{sec:orbit equivalence} immediately give the desired bound on $p$.
\end{proof}

There are also consequences for iterated halo products with the same number of iterations. We simply focus on standard halo products, rather than permutational ones. Here is an example for standard lampjugglers (for iterated standard wreath products, an analogous statement is more or less what has already been done in~\cite[Corollary~7.5]{DKLMT22}).

\begin{theorem}\label{thm:optimalitypolynomialgrowthJugglers}
Let $d_{1},d_{2},r\ge 1$ be integers such that $d_{1}>d_{2}$. Let $n\ge 0$ be an integer. If $H$ and $K$ are polynomial growth groups of degrees $d_{1}$ and $d_{2}$ respectively, then there exists an $(\ld^p,\ld^{<\frac{d_{1}}{d_{2}}})$ orbit equivalence from $\jugglern{n}{r}{H}$ to $\jugglern{n}{r}{K}$ if and only if $p<\frac{d_{2}}{d_{1}}$.
\end{theorem}

\begin{proof} From~\cite[Theorem~1.6]{DLIT25}, we know that there exists an $(\ld^{<\frac{d_{2}}{d_{1}}},\ld^{<\frac{d_{1}}{d_{2}}})$ orbit equivalence from $H$ to $K$, so the \enquote{if} part follows from Theorem~\ref{thm:stabilityofcouplings+quantificationLampjugglers}.

\smallskip

Conversely, assume that there exists an $(\ld^p,\ld^{0})$ orbit equivalence from $\jugglern{n}{r}{H}$ to $\jugglern{n}{r}{K}$ for $p\in \left]0,1\right]$. 
Then Theorem~\ref{thm:dklmt} implies that 
\begin{equation*}
    \left (\prof{\jugglern{n}{r}{K}}(x)\right )^{p} \preccurlyeq \prof{\jugglern{n}{r}{H}}(x)
\end{equation*}
and thus, from~\cite[Theorem~C]{cordum25}, we get 
\begin{equation*}
    \left(\frac{\ln^{\circ n}(x)}{\ln^{\circ (n+1)}(x)}\right)^{\frac{p}{d_{2}}} \preccurlyeq \left(\frac{\ln^{\circ n}(x)}{\ln^{\circ (n+1)}(x)}\right)^{\frac{1}{d_{1}}}.
\end{equation*}
This domination forces $p\le \frac{d_{2}}{d_{1}}$. Lastly, from Theorem~\ref{thm:threshold} we deduce that $p\neq\frac{d_{2}}{d_{1}}$.
\end{proof}

However, since we do not have precise estimates on isoperimetric profiles of lampcloners over free abelian groups, the "if and only if" of the last statement fails for this class of groups.

\begin{theorem}\label{thm:optimalitypolynomialgrowthCloners}
Let $\field$ be a finite field and let $d_{1},d_{2}\ge 1$ be integers such that $d_{1}>d_{2}$. If $H$ and $K$ are polynomial growth groups of degrees $d_{1}$ and $d_{2}$ respectively, then there exists an $(\ld^{<\frac{d_{2}}{d_{1}}},\ld^{<\frac{d_{1}}{d_{2}}})$ orbit equivalence from $\clonern{n}{H}$ to $\clonern{n}{K}$. Conversely, given $p>0$, if there exists an $(\ld^{p},\ld^{<\frac{d_{1}}{d_{2}}})$ orbit equivalence from $\clonern{n}{H}$ to $\clonern{n}{K}$, then $p<\frac{2d_{2}}{d_{1}}$.
\end{theorem}

\begin{proof}
From~\cite[Theorem~1.6]{DLIT25}, we know that there exists an $(\ld^{<\frac{d_{2}}{d_{1}}},\ld^{<\frac{d_{1}}{d_{2}}})$ orbit equivalence from $H$ to $K$, so the same holds from $\clonern{n}{H}$ to $\clonern{n}{K}$ by Theorem~\ref{thm:stabilityofcouplings+quantificationLampcloners}.
For the converse, we know from~\cite[Theorem~G]{cordum25} that
\begin{equation*}
    \prof{\cloner{\Z^{d_{1}}}}(x)\preccurlyeq \ln(x)^{\frac{1}{d_{1}}}
\end{equation*}
and
\begin{equation*}
    \prof{\cloner{\Z^{d_{2}}}}(x)\succcurlyeq \ln(x)^{\frac{1}{2d_{2}}}
\end{equation*}
so Theorems~\ref{thm:dklmt} and~\ref{thm:threshold} force $p<\frac{2d_{2}}{d_{1}}$.
\end{proof}

Other examples of quantitatively optimal orbit equivalence couplings have been found in~\cite{DKLMT22}, for instance between $\Z$ and $F\wr\Z$, where $F$ is a non-trivial finite group. Combined with our stability result, we get the following equivalences. This time, we can get an \enquote{if and only if} for lampcloners.

\begin{theorem}\label{th:ShufflerZdLamplighter}
Let $F$ be a non-trivial finite group, let $n$ be an integer and let $p>0$. Then:
\begin{itemize}
    \item for every integer $r>1$, there is an $(\ln^p,\exp)$-integrable orbit equivalence coupling from $\jugglern{n}{r}{F\wr\Z}$ to $\jugglern{n}{r}{\Z}$ if and only if $p<1$;
    \item for every finite field $\field$, there is an $(\ln^p,\exp)$-integrable orbit equivalence coupling from $\clonern{n}{F\wr\Z}$ to $\clonern{n}{\Z}$ if and only if $p<1$.
\end{itemize}
\end{theorem}

\begin{proof}
From~\cite[Proposition~6.20]{DKLMT22}, we know that there exists an $(\ln^p,\exp)$-integrable orbit equivalence coupling from $F\wr\Z$ to $\Z$. Hence Theorems~\ref{thm:stabilityofcouplings+quantificationLampjugglers} and~\ref{thm:stabilityofcouplings+quantificationLampcloners} provide, for every $p<1$, an $(\ln^p,\exp)$-integrable orbit equivalence coupling from $\jugglern{n}{r}{F\wr\Z}$ to $\jugglern{n}{r}{\Z}$, and another one from $\clonern{n}{F\wr\Z}$ to $\clonern{n}{\Z}$.

\smallskip

Now let us assume the existence of an $(\ln^p,\ld^0)$-integrable orbit equivalence coupling from $\jugglern{n}{r}{F\wr\Z}$ to $\jugglern{n}{r}{\Z}$ and let us prove that $p<1$. By Theorem~\ref{thm:dklmt}, we get
\begin{equation}\label{eq:1}
        \left(\ln(\prof{\jugglern{n}{r}{\Z}}(x))\right)^{p}\preccurlyeq\prof{\jugglern{n}{r}{F\wr\Z}}(x)
\end{equation}
Using~\cite[Theorem~C]{cordum25}, we have $\prof{\jugglern{n}{r}{\Z}}(x)\simeq\ln^{\circ n}(x)$ and $\prof{\jugglern{n}{r}{F\wr\Z}}(x)\simeq\ln^{\circ (n+1)}(x)$. The inequality~\eqref{eq:1} thus becomes 
\begin{equation*}
    \left (\ln^{\circ (n+1)}(x)\right )^p\preccurlyeq\ln^{\circ (n+1)}(x)
\end{equation*}
which forces $p\le 1$. Lastly, an orbit equivalence in the case $p=1$ is excluded by Theorem~\ref{thm:threshold}.

\smallskip

For the analogous result about lampcloners, we know from~\cite[Theorem~G]{cordum25} that $\prof{\clonern{n}{F\wr\Z}}(x)\simeq\ln^{\circ (n+1)}(x)$. The only difference with lampjugglers is that we do not have a precise estimate on $\prof{\clonern{n}{\Z}}(x)$, but we can have an exact one for $\ln{(\prof{\clonern{n}{\Z}}(x))}$. Indeed, by~\cite[Theorem~G]{cordum25}, we know that
\begin{equation*}
    \sqrt{\ln^{\circ n}{(x)}}\preccurlyeq\prof{\clonern{n}{\Z}}(x)\preccurlyeq\ln^{\circ n}{(x)}
\end{equation*}
and applying the logarithm, the upper and lower bounds asymptotically become the same and we get $\ln{(\prof{\clonern{n}{\Z}}(x))}\simeq\ln^{\circ (n+1)}(x)$. We finally conclude as for lampjugglers.
\end{proof}

\bigskip
\section{Construction of orbit equivalence couplings using F\o lner tiling sequences}\label{sec:folnerTiling}

In this section, given a halo product $\halo$ (for instance a lampjuggler or a lampcloner), we construct orbit equivalence couplings between $\Z^{d_1}$ and $\halo \Z^{d_2}$. A powerful tool to construct such equivalences is provided by~\textit{F\o lner tiling sequences}, introduced first in~\cite{DKLMT22} and that we present now. Combined with our stability results, we also construct couplings between iterated halo products. These applications are presented in the subsequent subsections.

\subsection{Preliminaries}\label{sec:PreliminariesFolnerTiling}

Given an amenable group $ G$, a (right) F\o lner tiling sequence $(F_{n})_{n\ge 0}$ is a (right) F\o lner sequence of $G$ satisfying the~\textit{tiling condition}: for every $n\ge 0$, there exists a finite subset $\Sigma_{n}$ of $G$ such that
\begin{itemize}
    \item the translates $g F_{n}$, for $g\in\Sigma_{n}$, are pairwise disjoint;
    \item $F_{n+1}=\Sigma_{n} F_{n}$.
\end{itemize}
The set $F_{n}$ is a~\textit{tile}, $\Sigma_{n}$ is a set of~\textit{shifts}, and the sequence $(\Sigma_{n})_{n\ge 0}$ is the~\textit{F\o lner tiling shift} associated to the F\o lner tiling sequence $(F_{n})_{n\ge 0}$. We can define analogously left F\o lner tiling sequences, but we will only work with right F\o lner tiling sequences in this paper. For convenience, we also assume $F_{0}=\lbrace 1_{G} \rbrace$, so that any element of $F_{n+1}$ can be uniquely written as $x_{n}x_{n-1}\dots x_{0}$ with $x_{i}\in\Sigma_{i}$.

\begin{example}
Given integer $m\ge 2$ and $d\ge 1$, the sequence $(F_{n})_{n\ge 0}$ defined by
\begin{equation*}
        F_{n}\defeq \lbrace 0,1,\ldots,m^{n}-1\rbrace^{d}
\end{equation*}
is a F\o lner tiling sequence of $\Z^{d}$, with F\o lner tiling shifts $(\Sigma_{n})_{n\ge 0}$ given by
\begin{equation*}
        \Sigma_{n}\defeq \lbrace 0,m^n,2m^n,\dots,(m-1)m^n\rbrace^{d}.
\end{equation*}
\end{example}

A F\o lner tiling sequence gives rise to a probability measure-preserving $G$-action on the Cantor set
\begin{equation*}
    X_{F}\defeq\prod_{n\ge 0}{\Sigma_{n}}
\end{equation*}
endowed with the product measure of uniform distributions on each $\Sigma_{n}$. It is defined in the following way: the F\o lner condition implies that, for almost every $(x_{n})_{n\ge 0}$, for every $g\in G$, we have
\begin{equation}\label{eq:defactionfolnertiling}
    x_{N}x_{N-1}\dots x_{0} g\in F_{N+1},
\end{equation}
for a large enough integer $N$, so it can uniquely be written as $x'_{N}\dots x'_{0}$ with $x'_{i}\in\Sigma_{i}$, and we define
\begin{equation*}
g (x_{n})_{n\ge 0}\defeq (x'_{n})_{n\ge 0},
\end{equation*}
with $x'_{i}\defeq x_{i}$ for every $i\ge N+1$. This definition does not depend on the integer $N$ for which Equation~\eqref{eq:defactionfolnertiling} occurs.

\smallskip

The main interest of this action is that the equivalence relation it generates is (up to a null set) the cofinite equivalence relation: for almost every $\bm{x}=(x_{n})_{n\ge 0}\in X_F$, for every $g\in G$, we have $x_{n}=(g\bm{x})_{n}$ for large enough integers $n$. Moreover, if $(F'_{n})_{n\ge 0}$ is a F\o lner tiling sequence of another finitely generated group $ H$ satisfying $|F_{n}|=|F'_{n}|$ for every $n\ge 0$, then using a bijection $\Sigma_{n}\to\Sigma_{n}'$ for each $n\ge 0$ , we get a measure isomorphism between $X_{F}$ and $X_{F'}$. This is actually an orbit equivalence between the underlying actions since they generate the cofinite equivalence relation. We now present a criteria, proved in~\cite{DKLMT22}, to quantify this orbit equivalence.

\begin{definition}
Let $ G$ be a finitely generated group with a finite generating set $S_{G}$. Given sequences $R=(R_{n})_{n\ge 0}$ and $\varepsilon=(\varepsilon_{n})_{n\ge 0}$ of positive real numbers, we say that the sequence $(F_{n})_{n\ge 0}$ is an $(R,\varepsilon)$-F\o lner tiling sequence of $G$ if
\begin{itemize}
    \item $(F_{n})_{n\ge 0}$ is F\o lner tiling sequence of $ G$;
    \item for every $n\ge 0$, $\diam{F_{n}}\le R_{n}$;
    \item for every $n\ge 0$ and every $s\in S_{G}$, $\frac{|F_{n}s\setminus F_{n}|}{|F_{n}|}\le\varepsilon_{n}$,
\end{itemize}
where $\diam{F}$ denotes the diameter of a finite subset $F$ of $G$, defined as 
\begin{equation*}
    \diam{F}\defeq\sup_{f,g\in F}{|f^{-1}g|_{S_{G}}}.
\end{equation*}
\end{definition}

Here is then a sufficient condition to check on the tilings to ensure that the couplings constructed above have the desired level of integrability.

\begin{theorem}[{\cite[Proposition~6.9]{DKLMT22}}]\label{thm:TilingSufficientConditionQuantitative}
Let $ G$ and $H$ be two finitely generated groups, and $\varphi,\psi\colon\R_{+}\rightarrow\R_{+}$ be non-decreasing maps. Assume that there exist sequences $R=(R_{n})_{n\ge 0}$, $R'=(R'_{n})_{n\ge 0}$, $\varepsilon=(\varepsilon_{n})_{n\ge 0}$ and $\varepsilon'=(\varepsilon'_{n})_{n\ge 0}$ of positive real numbers, an $(R,\varepsilon)$-F\o lner tiling sequence $(F_{n})_{n\in\N}$ of $G$ and an $(R',\varepsilon')$-F\o lner tiling sequence $(F_{n}')_{n\in\N}$ of $H$ such that $|F_{n}|=|F_{n}'|$ for any $n\in\N$, and assume that the series
\begin{equation*}
    \sum_{n\ge 0} \varphi(R_{n+1}')\varepsilon_{n}\; \text{and}\; \sum_{n\ge 0} \psi(R_{n+1})\varepsilon_{n}'
\end{equation*}
converge. Then the orbit equivalence coupling built above from $G$ to $H$ is $(\varphi,\psi)$-integrable.
\end{theorem}

Note that, when one wants to get quantitative orbit equivalences between two groups $G$ and $H$, F\o lner tiling sequences are particularly useful if one of the two groups is $\Z$. Indeed, if we find a F\o lner tiling sequence of $G$, then a natural F\o lner tiling sequence $(F'_{n})_{n\ge 0}$ for $ H=\Z$, satisfying $|F_{n}|=|F_{n}'|$, is
\begin{equation*}
    F_{n}'\defeq \lbrace 0,1,\ldots, |F_{n}|-1\rbrace.
\end{equation*}
However the tiling condition and the condition on the cardinalities make the problem harder when $H$ is bigger than $\Z$, even for $ H=\Z^2$. A clever choice of F\o lner tiling sequence of free abelian groups yields the following.

\begin{theorem}[{\cite[Theorem~6.12]{DKLMT22}}]\label{thm:ExplicitCouplingZd}
Let $d_{1}>d_{2}$ be two positive integers. Then there exists an orbit equivalence coupling from $\Z^{d_{1}}$ to $\Z^{d_{2}}$ which is $(\varphi_{\varepsilon},\psi_{\varepsilon})$-integrable for every $\varepsilon>0$, where
\begin{equation*}
        \varphi_{\varepsilon}(x)=\frac{x^{\frac{d_2}{d_1}}}{\ln{(x)}^{1+\varepsilon}}\; \text{and}\; \psi_{\varepsilon}(x)=\frac{x^{\frac{d_{1}}{d_{2}}}}{\ln(x)^{1+\varepsilon}}.
\end{equation*}
\end{theorem}

Note that, if $d_{1}>d_{2}$ and $p\ge\frac{d_{2}}{d_{1}}$, Theorems~\ref{thm:dklmt} and~\ref{thm:threshold} imply that we cannot have an $(\ld^p,\ld^0)$ orbit equivalence coupling from 
$\Z^{d_{1}}$ to $\Z^{d_{2}}$. Therefore the orbit equivalence coupling built in the last statement is optimal.

\smallskip

It is also possible to find a nice F\o lner tiling sequence for the lamplighter group $\Z/2\Z\wr\Z$, producing a quantitatively optimal coupling with $\Z$ (see~\cite[Proposition~6.20]{DKLMT22}). Using these techniques, Escalier also built in~\cite{Esc24} optimal couplings between $\Z$ and Brieussel-Zheng's groups, which have prescribed isoperimetric profiles~\cite{BZ21} as mentioned at the end of Section~\ref{sec:isoprof}.

\subsection{F\o lner tiling sequences of halo products}

\subsubsection{A general observation}\label{sec:generalobservationtilinghalo}

By~\cite[Proposition~6.19]{DKLMT22}, we know a F\o lner tiling $(L_n)_{n\ge 1}$ of the lamplighter $F\wr\Z$, where $F$ is a finite group. It is defined by
\begin{equation}\label{eq:FolnerTilingLamplighter}
    L_n\defeq \lbrace (f,h)\in F\wr\Z\mid\supp{f}\subset\lbrace 0,1,\dots,2^n-1\rbrace, h\in\lbrace 0,1,\dots,2^n-1\rbrace\rbrace.
\end{equation}

In this section, we describe a general phenomenon that will be useful to construct F\o lner tiling sequences for our favorite halo products.

\begin{proposition}\label{prop:HowToBuildFolnerTiling}
Let $H$ be an amenable group and $\mathcal{L}H$ be a halo product over $H$. Assume that $(F_{n})_{n\in\N}$ is a F\o lner tiling sequence of $H$, with an associated shift sequence $(\Sigma_{n})_{n\in\N}$, and that $L(F_{n})_{n\in\N}$ is finite. Assume furthermore that for every $h\in\Sigma_{n}$, there exists a finite subset $\Sigma_{h,n}$ of $L(F_{n+1})$ such that we have
\begin{equation*}
        L(F_{n+1})=\bigsqcup_{\sigma\in\Sigma_{h,n}}{\sigma L(hF_{n})}.
\end{equation*}
Then $(L(F_{n})\times F_{n})_{n\in\N}$ satisfies the tiling condition (as defined in the beginning of Section~\ref{sec:PreliminariesFolnerTiling}), with shift sequence $\left(\bigcup_{h\in \Sigma_{n}}{(\Sigma_{h,n}\times\lbrace h\rbrace)}\right)_{n\in\N}$.
\end{proposition}

For instance, for the F\o lner tiling sequence described in~\eqref{eq:FolnerTilingLamplighter}, we start from the F\o lner tiling sequence $(\lbrace 0,1,\dots,2^n-1\rbrace)_{n\ge 0}$ of $\Z$, with shift sequence $(\lbrace 0,2^n\rbrace)_{n\ge 1}$, and we get $L_{n}=L(\lbrace 0,1,\dots,2^n-1\rbrace)$. Given $h\in\lbrace 0,2^n\rbrace$, $L(\lbrace h,h+1,\dots,h+2^n-1\rbrace)$ is the set of functions $\Z\to F$ of support in $\lbrace h,h+1,\dots,h+2^n-1\rbrace$. Thus, to get $L(\lbrace 0,1,\dots,2^{n+1}-1\rbrace)$ from $L(\lbrace h,h+1,\dots,h+2^n-1\rbrace)$, we have to translate it by all functions supported in $\lbrace 0,1,\dots,2^{n+1}\rbrace\setminus\lbrace h,h+1,\dots,h+2^{n}-1\rbrace$, this is exactly $\Sigma_{h,n}$. We will show that this proposition also applies to lampjugglers and lampcloners (see Sections~\ref{sec:tilingShuffler} and~\ref{sec:tilingCloner}).

\begin{proof}[Proof of Proposition~\ref{prop:HowToBuildFolnerTiling}]
Using the composition law, we see that
\begin{align*}
    \bigcup_{h\in\Sigma_{n}}{\bigcup_{\sigma\in\Sigma_{h,n}}{(\sigma,h)(L(F_{n})\times F_{n})}}&=\bigcup_{h\in\Sigma_{n}}{\bigcup_{\sigma\in\Sigma_{h,n}}{(\sigma L(hF_{n})\times hF_{n})}}\\
    &=\bigcup_{h\in\Sigma_n}{\left (\left (\bigcup_{\sigma\in\Sigma_{h,n}}{\sigma L(hF_{n})}\right)\times hF_{n}\right)}\\
    &=\bigcup_{h\in\Sigma_{n}}{(L(F_{n+1})\times hF_{n})}\\
    &=L(F_{n+1})\times F_{n+1},
\end{align*}
and it is straightforward to prove that the unions are disjoint.
\end{proof}

For the sequel, we will need to check that the sequence $(L(F_{n})\times F_{n})_{n\in\N}$ provided by the previous proposition is in fact a F\o lner sequence, and, in view of Theorem~\ref{thm:TilingSufficientConditionQuantitative}, we will actually need to quantify this F\o lner sequence, namely finding sequences $R=(R_{n})_{n\ge 0}$ and $\varepsilon=(\varepsilon_{n})_{n\ge 0}$ such that $(L(F_{n})\times F_{n})_{n\ge 0}$ is an $(R,\varepsilon)$-F\o lner tiling sequence. For this we have the following general result.

\begin{proposition}\label{prop:quantitativeFolner}
Let $H$ be a finitely generated amenable group, with a finite generating set $S_H$, and let $\halo H$ be a halo product over $H$. Let $(F_{n})_{n\in \N}$ be a sequence of non empty subsets of $H$ such that $L(F_{n})$ is finite. For any $n\ge 0$, let $L_n=L(F_{n})\times F_{n}$. Then, for any $s\in S_{H}$, one has
\begin{equation*}
    \frac{|L_{n}\cdot (1_{L(H)},s)\setminus L_{n}|}{|L_{n}|}=\frac{|F_{n}s\setminus F_{n}|}{|F_{n}|},
\end{equation*}
as well as 
\begin{equation*}
        \frac{|L_{n}\cdot (\sigma_{s},1_{H})\setminus L_{n}|}{|L_{n}|}\leq\frac{|F_{n}s\setminus F_{n}|}{|F_{n}|}
\end{equation*}
for any $\sigma_{s}\in L(\lbrace 1_{H},s\rbrace)$. 
\end{proposition}

\begin{proof}
Let $(\sigma,h)\in L_{n}$ and let $s\in S_{H}$. The composition law of $\halo H$ directly implies that $(\sigma,h)(1_{L(H)},s)=(\sigma,hs)$ and $(\sigma,h)(\sigma_s,1_H)=(\sigma(h\cdot\sigma_s),h)$. Therefore $(\sigma,h)(1_{L(H)},s)$ lies in $L_n$ if and only if $hs$ lies in $F_{n}$, and since $h\cdot\sigma_s$ lies in $L(\lbrace h,hs\rbrace)$, we know that $(\sigma,h)(\sigma_s,1_H)$ lies in $L_{n}$ if $hs$ lies in $F_{n}$. We thus get
\begin{equation*}
    \frac{|L_{n}\cdot(1_{L(H)},s)\setminus L_{n}|}{|L_{n}|}=\frac{|L_{n}|\cdot|F_{n}s\setminus F_{n}|}{|L_{n}|\cdot |F_n|}=\frac{|F_{n}s\setminus F_{n}|}{|F_{n}|}
\end{equation*}
and similarly
\begin{equation*}
    \frac{|L_{n}\cdot(\sigma_{s},1_{H})\setminus L_{n}|}{|L_{n}|}\leq\frac{|L_{n}|\cdot|F_{n}s\setminus F_{n}|}{|L_{n}|\cdot |F_{n}|}=\frac{|F_{n}s\setminus F_{n}|}{|F_{n}|}.
\end{equation*}
This completes the proof.
\end{proof}

In the case of a halo products $\mathcal{L}H$ generated by
\begin{equation*}
    S_{\halo H}\defeq\left\lbrace (1_{L(H)},s) : s\in S_H\right\rbrace\cup\bigcup_{s\in S_H}{\left\lbrace(\sigma,1_{H}) : \sigma\in L(\lbrace 1_{H},s\rbrace)\right\rbrace}
\end{equation*}
what we call a~\textit{naturally generated} halo product in~\cite{cordum25}, Proposition~\ref{prop:quantitativeFolner} asserts the following: if $(F_{n})_{n\in\N}$ is a F\o lner sequence of $H$ and if $L(F_{n})$ is finite for every $n\in\N$, then $(L(F_{n})\times F_{n})_{n\in\N}$ is an $(R,\epsilon)$-F\o lner sequence with the same $\varepsilon=(\varepsilon_n)_{n\in\N}$ as $(F_{n})_{n\in\N}$. This applies for instance to lamplighters, as was done in~\cite[Proposition~6.19]{DKLMT22}, and also to lampjugglers and lampcloners. For the latters, it remains to compute the diameters $R_{n}$ with the generating set $S_{\halo H}$ provided by this last result. Furthermore, our main examples of halo products satisfy the following property: the cardinality of $L(F)$ only depends on the cardinality of $F$, and is denoted by $\La_{\halo H}(|F|)$. This is what we called halo products with~\textit{consistent blocks} in~\cite{cordum25}, and $\La_{\halo H}\colon\N\rightarrow\N$ is the~\textit{lamp growth sequence} of $\halo H$. For these halo products, the sets of the F\o lner tiling sequences that we build have cardinality $\Lambda_{\halo H}(|F_{n}|)\cdot|F_{n}|$.

\smallskip

In the sequel, we illustrate these results for lampjugglers and lampcloners. Note that, for simplicity, we will only consider F\o lner tiling sequences $(F_{n})_{n\in\N}$ for the base group $H$ satisfying $1_{H}\in F_{n}$ for every $n\in\N$. This will be the case in our applications, when picking $H=\Z^d$.

\subsubsection{F\o lner tiling sequences of lampshufflers and lampjugglers}\label{sec:tilingShuffler}

Let us recall that if $S_{H}$ is a finite generating subset of $H$, then a finite generating set of $\juggler{r}{H}$ is 
\begin{equation*}
    S_{\halo H}=\left\lbrace(\tau_{(1_{H},i),(s,j)},1_{H}) : s\in S_{H}, 1\le i,j\le r\right\rbrace\cup\left\lbrace(\mathrm{id},s) : s\in S_{H}\right\rbrace
\end{equation*}
where $\tau_{(g,i),(h,j)}\colon H\times\lbrace 1,\dots,r\rbrace\to H\times\lbrace 1,\ldots,r\rbrace$ denotes the transposition that swaps $(g,i)$ and $(h,j)$. In particular, this halo product is naturally generated.

\smallskip

Let us now apply Propositions~\ref{prop:HowToBuildFolnerTiling} and~\ref{prop:quantitativeFolner} to get a F\o lner tiling sequence of $\juggler{r}{H}$ from a tiling sequence of $H$, with the associated quantitative informations.

\begin{theorem}\label{thm:folnertilingshuffler}
Let $H$ be a finitely generated amenable group. If $(F_{n})_{n\in\N}$ is a right F\o lner tiling sequence of $H$ satisfying
\begin{equation*}
        \frac{|F_{n}s\setminus F_{n}|}{|F_{n}|}\le\varepsilon_n
\end{equation*}
for all $s\in S_{H}$ and some sequence $(\varepsilon_{n})_{n\in\N}$ tending to $0$,
then there exists a right F\o lner tiling sequence $(L_{n})_{n\in\N}$ of $\juggler{r}{H}$ with the following properties:
\begin{enumerate}[label=(\roman*)]
        \item\label{item:1Folnertilingjuggler} for any $n\in\N$, $|L_{n}|=(r|F_{n}|)!\cdot|F_{n}|$;
        \item\label{item:2Folnertilingjuggler} for any $n\in\N$, $\diam{L_{n}}\le 5\cdot|F_{n}|\cdot\diam{F_{n}}$;
        \item\label{item:3Folnertilingjuggler} for all $\varphi\in S_{\juggler{r}{H}}$, $\frac{|L_{n}\varphi\setminus L_{n}|}{|L_{n}|}\le\varepsilon_{n}$.
\end{enumerate}
\end{theorem}

To prove the upper bound on the diameter, we will need the following claim.

\begin{lemma}\label{lem:diamsymjuggler}
Let $F$ be a finite subset of $H$. Then, for every $\varphi\in\sym{F\times\lbrace 1,\dots,r\rbrace}$, we have
\begin{equation*}
     \left|(\varphi,1_{H})\right|_{\juggler{r}{H}}\le 2K+(5|F|-6)\diam{F},
\end{equation*}
where $K\defeq\max_{h\in F}{|h|_{S_H}}$.
\end{lemma}

\begin{proof}
Let us decompose $\varphi$ as a product of cycles $\varphi=c_{1}c_{2}\dots c_{k}$ on $F\times\lbrace 1,\dots,r\rbrace$ with disjoint supports. For $1\le i \le k$, $c_{i}$ is an $\ell_{i}$-cycle $(x_{i,1},\ldots,x_{i,\ell_{i}})=\left((h_{i,1},j_{i,1}),\dots,(h_{i,\ell_{i}},j_{i,\ell_{i}})\right)$, which can itself be written as a product of transpositions
\begin{equation*}
c_{i}=\tau_{x_{i,1},x_{i,2}}\tau_{x_{i,2},x_{i,3}}\dots\tau_{x_{i,\ell_i-1},x_{i,\ell_i}}.
\end{equation*}
For any $1\le m\le \ell_{i}$, the transposition $\tau_{x_{i,m},x_{i,m+1}}$ can be written as
\begin{equation*}
\tau_{x_{i,m},x_{i,m+1}}=h_{i,m}\cdot\tau_{(1_H,j_{i,m}),(h_{i,m}^{-1}h_{i,m+1},j_{i,m+1})}
\end{equation*}
with the convention that $h_{i,\ell_{i}+1}\defeq h_{i,1}$ and $j_{i,\ell_{i}+1}\defeq j_{i,1}$, so that we have 
\begin{equation*}
(\tau_{x_{i,m},x_{i,m+1}},1_H)=(\mathrm{id},h_{i,m})(\tau_{(1_H,j_{i,m}),(h_{i,m}^{-1}h_{i,m+1},j_{i,m+1})},1_H)(\mathrm{id},h_{i,m}^{-1}).
\end{equation*}
Setting $y_{i,m}\defeq (1_H,j_{i,m})$ and $z_{i,m}\defeq (h_{i,m}^{-1}h_{i,m+1},j_{i,m+1})$, the element $(c_{i},1_{H})$ is thus equal to
\begin{equation*}
    (\mathrm{id},h_{i,1})(\tau_{y_{i,1},z_{i,1}},1_{H})(\mathrm{id},h_{i,1}^{-1}h_{i,2})(\tau_{y_{i,2},z_{i,2}},1_{H})\dots (\mathrm{id},h_{i,\ell_i-2}^{-1}h_{i,\ell_i-1}) (\tau_{y_{i,\ell_i-1},z_{i,\ell_i-1}},1_H)(\mathrm{id},h_{i,\ell_i-1}^{-1}),
\end{equation*}
namely $(c_{i},1_{H})=(\mathrm{id},h_{i,1}) \kappa_{i} (\mathrm{id},h_{i,\ell_i-1}^{-1})$, where 
\begin{equation*}
    \left|\kappa_{i}\right|_{\juggler{r}{H}}\le \left(4(\ell_{i}-1)+(\ell_{i}-2)\right)\diam{F}=(5\ell_{i}-6)\diam{F}
\end{equation*}
using Lemma~\ref{lem:wordlengthinlampshuffler}. Since we have
\begin{equation*}
        (\varphi,1_{H})=(\mathrm{id},h_{1,1})\kappa_1(\mathrm{id},h_{1,\ell_1-1}^{-1}h_{2,1})\kappa_2(\mathrm{id},h_{2,\ell_2-1}^{-1}h_{3,1})\ldots (\mathrm{id},h_{k-1,\ell_{k-1}-1}^{-1}h_{k,1})\kappa_k(\mathrm{id},h_{k,\ell_k-1}^{-1}),
\end{equation*}
we finally get
\begin{align*}
    \left|(\varphi,1_{H})\right|_{\juggler{r}{H}}&\le 2K+(k-1)\diam{F}+\sum_{i=1}^{k}{\left(5\ell_i-6\right)\diam{F}}\\
    &\le 2K+\left(k-1+(5|F|-6k)\right)\diam{F}\\
    &\le 2K+(5|F|-6)\diam{F}
\end{align*}
which concludes the proof.
\end{proof}

\begin{proof}[Proof of Theorem~\ref{thm:folnertilingshuffler}]
For $n\in\N$, let us set
\begin{equation*}
L_{n}\defeq\left\lbrace(\sigma,h)\in \juggler{r}{H} : \sigma\in \sym{F_{n}\times\lbrace 1,\dots ,r\rbrace}, h\in F_{n}\right\rbrace
\end{equation*}
where $\sym{F\times\lbrace 1,\dots ,r\rbrace}$ denotes the set of permutations supported in a finite subset $F\times\lbrace 1,\dots ,r\rbrace$ of $H$. Clearly,~\textit{\ref{item:1Folnertilingjuggler}} holds, as well as~\textit{\ref{item:3Folnertilingjuggler}} by Proposition~\ref{prop:quantitativeFolner}.

\smallskip

Given $(\sigma,h)$ and $(\sigma',h')$ in $L_{n}$, let us find an upper bound for $\left|(\sigma,h)^{-1}(\sigma',h')\right|_{S_{\juggler{r}{H}}}$. We first have
\begin{align*}
    (\sigma,h)^{-1}(\sigma',h')&=(h^{-1}\cdot\sigma^{-1},h^{-1})(\sigma',h')  \\
    &=\left([h^{-1}\cdot\sigma^{-1}]\circ [h^{-1}\cdot\sigma'],h^{-1}h'\right)  \\
    &=\left([h^{-1}\cdot\sigma^{-1}]\circ [h^{-1}\cdot\sigma'],1_H\right)\left(\mathrm{id},h^{-1}h'\right)
\end{align*}
so that
\begin{align*}
    \left|(\sigma,h)^{-1}(\sigma',h')\right|_{S_{\juggler{r}{H}}} &\le \left|([h^{-1}\cdot\sigma^{-1}]\circ [h^{-1}\cdot\sigma'],1_H)\right|_{S_{\juggler{r}{H}}}+\left|(\mathrm{id},h^{-1}h')\right|_{S_{\juggler{r}{H}}}  \\
    &\le \left|([h^{-1}\cdot\sigma^{-1}]\circ [h^{-1}\cdot\sigma'],1_{H})\right|_{S_{\juggler{r}{H}}}+\diam{F_{n}}
\end{align*}
and it remains to find an upper bound for $\left|([h^{-1}\cdot\sigma^{-1}]\circ [h^{-1}\cdot\sigma'],1_H)\right|_{S_{\juggler{r}{H}}}$. But the permutation $[h^{-1}\cdot\sigma^{-1}]\circ [h^{-1}\cdot\sigma']$ lies in $\sym{(h^{-1}F_{n})\times\lbrace 1,\dots,r\rbrace}$, so using Lemma~\ref{lem:diamsymjuggler}, it follows that
\begin{align*}
    \left|([h^{-1}\cdot\sigma^{-1}]\circ [h^{-1}\cdot\sigma'],1_H)\right|_{S_{\juggler{r}{H}}}&\le 2K+(5|h^{-1}F_n|-6)\diam{h^{-1}F_{n}},
\end{align*}
with $K\defeq\max_{k\in F_{n}}{|h^{-1}k|_{S_{H}}}\leq\diam{F_{n}}$, so that we have
\begin{align*}
    \left|([h^{-1}\cdot\sigma^{-1}]\circ [h^{-1}\cdot\sigma'],1_H)\right|_{S_{\juggler{r}{H}}}&\le (5|F_{n}|-4)\diam{F_{n}}.
\end{align*}
Hence, we have 
\begin{equation*}
    \left|(\sigma,h)^{-1}(\sigma',h')\right|_{S_{\juggler{r}{H}}}\le 5|F_{n}|\diam{F_{n}}
\end{equation*}
and we are done for the proof of~\textit{\ref{item:2Folnertilingjuggler}}.

\smallskip
    
Lastly, we show the tiling condition, using Proposition~\ref{prop:HowToBuildFolnerTiling}. Let $(\Sigma_{n})_{n\in\N}$ be a sequence of F\o lner shifts for $(F_{n})_{n\in\N}$, satisfying $F_{n+1}=\Sigma_{n} F_{n}$. Let us fix $h\in\Sigma_n$ and let us find a finite subset $\Sigma_{h,n}$ of $L(F_{n+1})$ such that we have the following disjoint union:
\begin{equation}\label{eq:DesiredDisjointUnionjuggler}
    \sym{F_{n+1}}=\bigsqcup_{\sigma\in\Sigma_{h,n}}{\sigma\  \sym{hF_n}}.
\end{equation}
For every $A\subset F_{n+1}\times\lbrace 1,\dots,r\rbrace$ of cardinality $r|F_{n}|$ and every family $\textbf{$x$}=(x_{g,i})_{(g,i)\in (F_{n+1}\setminus hF_n)\times\{1,\ldots,r\}}$ satisfying 
\begin{equation}\label{eq:CompatibleJuggler}
    A\sqcup\lbrace x_{g,i} : (g,i)\in (F_{n+1}\setminus hF_{n})\times\lbrace 1,\dots,r\rbrace\rbrace=F_{n+1}\times\lbrace 1,\dots,r\rbrace,
\end{equation}
let us choose one permutation $\tau^{A,\textbf{$x$}}\in\sym{F_{n+1}}$ satisfying
\begin{equation*}
        \tau^{A,\textbf{$x$}}((hF_{n})\times\lbrace 1,\dots,r\rbrace)=A \;\text{and} \; \forall (g,i)\in (F_{n+1}\setminus hF_{n})\times\lbrace 1,\dots,r\rbrace, \tau^{A,\textbf{$x$}}(g,i)=x_{g,i},
\end{equation*}
then the set
\begin{equation*}
    \left\lbrace\tau^{A,\textbf{$x$}}\circ\sigma : \sigma\in\sym{hF_{n}\times\lbrace 1,\dots,r\rbrace}\right\rbrace
\end{equation*}
describes all permutations $\rho\in\sym{F_{n+1}\times\lbrace 1,\dots,r\rbrace}$ satisfying
\begin{equation*}
\rho(hF_{n}\times\lbrace 1,\dots,r\rbrace)=A\;\text{and}\; \forall (g,i)\in (F_{n+1}\setminus hF_{n})\times\lbrace 1,\dots,r\rbrace, \rho(g,i)=x_{g,i}.
\end{equation*}
It remains to define $\Sigma_{h,n}$ as the set of all chosen permutations $\tau^{A,\textbf{$x$}}$ for every subset $A\subset F_{n+1}\times\lbrace 1,\dots,r\rbrace$ of cardinality $r|F_{n}|$ and every family $\textbf{$x$}=(x_g)_{g\in (F_{n+1}\setminus hF_{n})\times\lbrace 1,\dots r\rbrace}$ such that~\eqref{eq:CompatibleJuggler} holds, and we get~\eqref{eq:DesiredDisjointUnionjuggler}. So the F\o lner sequence $(L_{n})$ satisfies the tiling condition by Proposition~\ref{prop:HowToBuildFolnerTiling}.
\end{proof}

\subsubsection{F\o lner tiling sequences of lampcloners}\label{sec:tilingCloner}

Let us recall that if $S_{H}$ is a finite generating subset of $H$, then a finite generating set of $\cloner{H}$ is 
\begin{equation*}
    S_{\cloner{H}}\defeq\lbrace (\delta_{1_{H}}(\lambda), 1_{H}) : \lambda \in \field\setminus\lbrace 0\rbrace\rbrace\cup\lbrace (\tau_{1_{H},s}(\lambda),1_{H}) : s\in S_{H}\rbrace \cup\lbrace (\text{id}, h) : s\in S_{H}\rbrace
\end{equation*}
where the elements $\delta_{p}(\lambda)$ are diagonal matrices and elements $\tau_{pq}(\lambda)$ are transvections. In particular, this halo product is naturally generated.

\smallskip

Let us now apply Propositions~\ref{prop:HowToBuildFolnerTiling} and~\ref{prop:quantitativeFolner} to get a F\o lner tiling sequence of $\cloner{H}$ from a tiling sequence of $H$, with the associated quantitative information on it.

\begin{theorem}\label{thm:folnertilingcloner}
Let $H$ be a finitely generated amenable group. If $(F_{n})_{n\in\N}$ is a right F\o lner tiling sequence of $H$ satisfying
\begin{equation*}
        \frac{|F_{n}s\setminus F_{n}|}{|F_{n}|}\le\varepsilon_{n}
\end{equation*}
for all $s\in S_{H}$ and some sequence $(\varepsilon_{n})_{n\in\N}$ tending to $0$,
then there exists a right F\o lner tiling sequence $(L_{n})_{n\in\N}$ of $\cloner{H}$ with the following properties:
\begin{enumerate}[label=(\roman*)]
        \item\label{item:1Folnertilingcloner} for any $n\in\N$, $|L_{n}|=\Lambda_{\cloner{H}}(|F_n|)\cdot|F_{n}|$;
        \item\label{item:2Folnertilingcloner} there exists a constant $C>0$ such that for any $n\in\N$, $\diam{L_{n}}\le C|F_{n}|^3\diam{F_{n}}$;
        \item\label{item:3Folnertilingcloner} for all $\varphi\in S_{\cloner{H}}$, $\frac{|L_{n}\varphi\setminus L_{n}|}{|L_{n}|}\le\varepsilon_{n}$.
\end{enumerate}
\end{theorem}

To prove the upper bound on the diameter, we will need the following claim.

\begin{lemma}\label{lem:diamsymcloner}
Let $F$ be a finite subset of $H$. Then there exist a constant $C>0$ such that, for every $\varphi\in\mathrm{FGL(F)}$, we have
\begin{equation*}
     \left|(\varphi,1_{H})\right|_{S_{\cloner{H}}}\le 2K+C|F|^3\diam{F}
\end{equation*}
where $K\defeq\max_{h\in F}{|h|_{S_{H}}}$.
\end{lemma}

\begin{proof}
By Gaussian elimination, we need $N$ basic operations to get $\varphi$ from the identity in $\mathrm{FGL(F)}$, with $N=O(|F|^3)$. By operations, we mean transvections and dilatations of support in $F$. Let us notice that we have 
\begin{equation*}
    \tau_{g,h}(\lambda)=g\cdot\tau_{1_{H},g^{-1}h}(\lambda)
\end{equation*}
for transvections, so that we get 
\begin{equation*}
    (\tau_{g,h}(\lambda),1_{H})=(\mathrm{id},g)(\tau_{1_H,g^{-1}h}(\lambda),1_{H})(\mathrm{id},g^{-1})
\end{equation*}
in $\cloner{H}$. Also, for dilatations, we have
\begin{equation*}
    \delta_g(\lambda)=g\cdot\delta_{1_{H}}(\lambda)
\end{equation*}
and this can be written in $\cloner{H}$ as 
\begin{equation*}
    (\delta_g(\lambda),1_{H})=(\mathrm{id},g)(\delta_{1_{H}}(\lambda),1_{H})(\mathrm{id},g^{-1}).
\end{equation*}
We thus get $(\varphi,1_{H})$ as a product of the following form
\begin{equation*}
    (\varphi,1_{H})=(\mathrm{id},h_{1})(\varphi_{1},1_{H})(\mathrm{id},h_{1}^{-1}h_{2})(\varphi_{2},1_{H})(\mathrm{id},h_{2}^{-1}h_{3})\dots (\mathrm{id},h_{N-1}^{-1}h_{N})(\varphi_{N},1_{H})(\mathrm{id},h_{N}^{-1})
\end{equation*}
with basic operations $\varphi_{i}$ supported in $\lbrace 1_{H},F^{-1}F\rbrace$ and elements $h_{i}\in F$. Now using Lemma~\ref{lem:wordlengthinlampcloner}, we get
\begin{equation*}
    |(\varphi,1_{H})|_{S_{\cloner{H}}}\le 2K+N\diam{F}+14N\diam{F}
\end{equation*}
and we are done.
\end{proof}

\begin{proof}[Proof of Theorem~\ref{thm:folnertilingcloner}]
For $n\in\N$, let us set
\begin{equation*}
L_{n}\defeq\left\lbrace(\sigma,h)\in \cloner{H} : \sigma\in \mathrm{FGL}(F_{n}), h\in F_{n}\right\rbrace.
\end{equation*}
Clearly,~\textit{\ref{item:1Folnertilingcloner}} holds by definition of the lamp growth sequence $\Lambda_{\cloner{H}}$, as well as~\textit{\ref{item:3Folnertilingcloner}} by Proposition~\ref{prop:quantitativeFolner}.

\smallskip

Given $(\sigma,h)$ and $(\sigma',h')$ in $L_n$, let us find an upper bound for $\left|(\sigma,h)^{-1}(\sigma',h')\right|_{S_{\cloner{H}}}$. We first have
\begin{align*}
    (\sigma,h)^{-1}(\sigma',h')&=(h^{-1}\cdot\sigma^{-1},h^{-1})(\sigma',h')  \\
    &=\left([h^{-1}\cdot\sigma^{-1}]\circ [h^{-1}\cdot\sigma'],h^{-1}h'\right)  \\
    &=\left([h^{-1}\cdot\sigma^{-1}]\circ [h^{-1}\cdot\sigma'],1_{H}\right)\left(\mathrm{id},h^{-1}h'\right)
\end{align*}
so that
\begin{align*}
    \left|(\sigma,h)^{-1}(\sigma',h')\right|_{S_{\cloner{H}}} &\le \left|([h^{-1}\cdot\sigma^{-1}]\circ [h^{-1}\cdot\sigma'],1_{H})\right|_{S_{\cloner{H}}}+\left|(\mathrm{id},h^{-1}h')\right|_{S_{\cloner{H}}}  \\
    &\le \left|([h^{-1}\cdot\sigma^{-1}]\circ [h^{-1}\cdot\sigma'],1_{H})\right|_{S_{\cloner{H}}}+\diam{F_{n}}
\end{align*}
and it remains to find an upper bound for $\left|([h^{-1}\cdot\sigma^{-1}]\circ [h^{-1}\cdot\sigma'],1_{H})\right|_{S_{\cloner{H}}}$. But the linear automorphism $[h^{-1}\cdot\sigma^{-1}]\circ [h^{-1}\cdot\sigma']$ lies in $\mathrm{FGL}(h^{-1}F_{n})$, so using Lemma~\ref{lem:diamsymcloner}, it follows that
\begin{align*}
    \left|([h^{-1}\cdot\sigma^{-1}]\circ [h^{-1}\cdot\sigma'],1_{H})\right|_{S_{\cloner{H}}}&\le 2K+C|h^{-1}F_{n}|^3\diam{h^{-1}F_{n}},
\end{align*}
with $K\defeq\max_{k\in F_{n}}{|h^{-1}k|_{S_H}}\le\diam{F_{n}}$ and for some constant $C>0$, so that we have
\begin{align*}
    \left|([h^{-1}\cdot\sigma^{-1}]\circ [h^{-1}\cdot\sigma'],1_{H})\right|_{S_{\cloner{H}}}&\le (C+2)|F_{n}|^3\diam{F_{n}}.
\end{align*}
Hence, we have 
\begin{equation*}
    \left|(\sigma,h)^{-1}(\sigma',h')\right|_{S_{\cloner{H}}}\le (C+3)|F_{n}|^3\diam{F_{n}}
\end{equation*}
and we are done for the proof of~\textit{\ref{item:2Folnertilingcloner}}.

\smallskip
    
Lastly, we show the tiling condition, using Proposition~\ref{prop:HowToBuildFolnerTiling}. Let $(\Sigma_{n})_{n\in\N}$ be a sequence of F\o lner shifts for $(F_{n})_{n\in\N}$, satisfying $F_{n+1}=\Sigma_{n} F_{n}$. Let us fix $h\in\Sigma_n$ and let us find a finite subset $\Sigma_{h,n}$ of $L(F_{n+1})$ such that we have the following disjoint union:
\begin{equation}\label{eq:DesiredDisjointUnioncloner}
    \mathrm{FGL}(F_{n+1})=\bigsqcup_{\sigma\in\Sigma_{h,n}}{\sigma\  \mathrm{FGL}(hF_n)}.
\end{equation}
For every subspace $A$ of $V_{F_{n+1}}$, of dimension $|F_{n}|$ and every family $\textbf{$x$}=(x_{v})_{v\in V_{F_{n+1}\setminus hF_{n}}}$ satisfying 
\begin{equation}\label{eq:CompatibleCloner}
    A\sqcup\lbrace x_{v} : v\in V_{F_{n+1}\setminus hF_n}\rbrace=V_{F_{n+1}},
\end{equation}
let us choose, if it exists, one linear automorphism $\tau^{A,\textbf{$x$}}\in\mathrm{FGL}(F_{n+1})$ satisfying
\begin{equation*}
        \tau^{A,\textbf{$x$}}(V_{hF_{n}})=A \;\text{and} \; \forall v\in V_{F_{n+1}\setminus hF_{n}}, \tau^{A,\textbf{$x$}}(v)=x_{v},
\end{equation*}
then the set
\begin{equation*}
    \left\lbrace\tau^{A,\textbf{$x$}}\circ\sigma : \sigma\in\mathrm{FGL}(hF_{n})\right\rbrace
\end{equation*}
describes all the linear automorphisms $\rho\in\mathrm{FGL}(F_{n+1})$ satisfying
\begin{equation*}
\rho(V_{hF_{n}})=A\;\text{and}\; \forall v\in V_{F_{n+1}\setminus hF_{n}}, \rho(v)=x_{v}.
\end{equation*}
It remains to define $\Sigma_{h,n}$ as the set of all the chosen linear automorphisms $\tau^{A,\textbf{$x$}}$ for every subspace $A$ of $V_{F_{n+1}}$, of dimension $|F_{n}|$ and every family $\textbf{$x$}=(x_{v})_{v\in V_{F_{n+1}\setminus hF_n}}$ such that~\eqref{eq:CompatibleCloner} holds and such that such a $\tau^{A,\textbf{$x$}}$ exists, and we get~\eqref{eq:DesiredDisjointUnioncloner}. Hence the F\o lner sequence $(L_{n})_{n\in\N}$ satisfies the tiling condition by Proposition~\ref{prop:HowToBuildFolnerTiling}.
\end{proof}

\subsection{Applications to quantitatively optimal orbit equivalence couplings}

\subsubsection{Lampshufflers and lampjugglers}\label{sec:folnertilingjugglers}

As a first application of our construction of F\o lner tiling sequences, we give an explicit orbit equivalence coupling between $\juggler{r}{\Z^{d_{1}}}$ and $\Z^{d_{2}}$.

\begin{theorem}\label{thm:couplingsZdandshufZk}
Let $d_{1},d_{2},r\ge 1$ be three integers. There exists an orbit equivalence coupling from $\juggler{r}{\Z^{d_{1}}}$ to $\Z^{d_{2}}$, which is $(\varphi_{d_{1},\varepsilon},\psi)$-integrable for every $\varepsilon>0$, where 
\begin{equation*}
    \varphi_{d_{1},\varepsilon}(x)\defeq\frac{\ln(x)^{\frac{1}{d_{1}}}}{\ln(\ln(x))^{1+\frac{1}{d_{1}}+\varepsilon}}\; \text{and} \; \psi(x)=x^{x^{\frac{d_{1}}{d_{1}+1}}}.
\end{equation*}
\end{theorem}

\begin{remark}
Note that $\left(\frac{\ln(x)}{\ln(\ln(x))}\right)^{\frac{1}{d_1}}$ is (asymptotically) the isoperimetric profile of $\juggler{r}{\Z^{d_{1}}}$. Theorem~\ref{thm:couplingsZdandshufZk} implies in particular that, for every $\varepsilon>0$, there is a $\left(\frac{\prof{\juggler{r}{\Z^{d_1}}}(x)}{\ln(\ln(x))^{1+\varepsilon}},\ld^{<\infty}\right)$-integrable orbit equivalence coupling from $\juggler{r}{\Z^{d_1}}$ to $\Z^{d_{2}}$.
\end{remark}

\begin{proof}[Proof of Theorem~\ref{thm:couplingsZdandshufZk}]
Endow $\Z^{d_{1}}$ (resp. $\Z^{d_{2}}$) with its standard generating set $S_{\Z^{d_{1}}}$ (resp. $S_{\Z^{d_{2}}}$). Using Theorem~\ref{thm:folnertilingshuffler}, the F\o lner tiling sequence 
\begin{equation*}
    (F_{n})_{n\in\N}\defeq \big(\lbrace 0,\dots,d_{2}^{d_{2} n}-1\rbrace^{d_{1}}\big)_{n\in\N}
\end{equation*}
of $\Z^{d_{1}}$ provides a F\o lner tiling sequence $(L_{n})_{n\in\N}$ of $\juggler{r}{\Z^{d_{1}}}$ satisfying
\begin{itemize}
        \item $|L_{n}|=(rd_2^{n d_{1} d_{2}})!\cdot d_{2}^{n d_{1} d_{2}}$, for all $n\in\N$;
        \item $\diam{L_{n}}\le 5 d_{2}^{n(d_{1}+1)d_{2}+1}$ for all $n\in\N$;
        \item $\frac{|L_{n}s\setminus L_{n}|}{|L_{n}|}\le\frac{1}{d_{2}^{n d_{2}}}$, for all $s\in S_{\juggler{r}{\Z^{d_{1}}}}$.
\end{itemize}
Given a positive integer $n$, let us define, for every $i\in\lbrace 1,2,\dots, d\rbrace$, the integer
\begin{equation*}
    \ell_{i}(n)\defeq\prod_{\substack{j\in\N \\ d_{2}j+i\le n}}{(d_{2}j+i)}.
\end{equation*}
These quantities satisfy
\begin{equation*}
    \ell_{1}(n)\ell_{2}(n)\dots \ell_{d}(n)=n!,
\end{equation*}
\begin{equation*}
    \ell_{1}(n d_{2})\le \ell_{2}(n d_{2})\le\dots\le \ell_{d_{2}}(n d_{2})\le \ell_{1}((n+1) d_{2})
\end{equation*}
and
\begin{equation*}
    \ell_{d_{2}}(n d_{2})=d_{2}^n\cdot n!\sim \left(\frac{d_{2} n}{e}\right)^{n}\sqrt{2\pi n}
\end{equation*}
using Stirling's formula. Now, we find a F\o lner tiling sequence $(G_{n})_{n\in\N}$ of $\Z^{d_2}$ satisfying $|L_{n}|=|G_{n}|$, setting
\begin{equation*}
    G_{n}\defeq\prod_{1\le i\le d_{2}}{\lbrace 0,1,\dots,d_{2}^{d_{1} n}\ell_i(rd_{2}^{n d_{1} d_{2}})\rbrace}.
\end{equation*}
It satisfies
\begin{itemize}
    \item for all $n\in\N$, $\diam{G_{n}}\le d_{2}^{d_{1} n}(\ell_1(rd_{2}^{n d_{1} d_{2}}) +\ldots +\ell_{d_{2}}(rd_{2}^{n d_{1} d_{2}}))\le d_{2}^{d_{1} n+1}\ell_{d_{2}}(rd_{2}^{n d_{1} d_{2}})$;
    \item for all $s\in S_{\Z^{d_{2}}}$, $\frac{|(s+G_{n})\setminus G_{n}|}{|G_{n}|}\le\frac{1}{d_{2}^{d_{1} n}\min{(\ell_1(rd_{2}^{n d_{1} d_{2}}),\ldots ,\ell_d(rd_{2}^{n d_{1} d_{2}}))}}\le\frac{1}{d_{2}^{d_{1} n}\ell_{d_{2}}(rd_{2}^{n d_{1} d_{2}-1})}$.
\end{itemize}

\noindent Observe that we have
\begin{equation*}
    \ln(d_{2}^{d_{1} n+1}\ell_{d_{2}}(rd_2^{n d_1 d_2}))\sim rd_2^{n d_{1} d_{2}-1}\ln(rd_{2}^{n d_1 d_2})=rd_{2}^{n d_1 d_2-1}n d_1 d_2\ln(d_{2})
\end{equation*}
and
\begin{equation*}
    \ln\big(\ln(d_{2}^{d_{1} n+1}\ell_{d_2}(rd_2^{n d_1 d_2}))\big)\sim n d_{1} d_{2}\ln(d_2),
\end{equation*}
which in turn implies
\begin{equation*}
    \frac{\varphi_{d_{1},\varepsilon}(d_{2}^{d_1 n+1}\ell_{d_2}(d_{2}^{n d_1 d_2}))}{d_2^{d_2 n}}\sim\frac{K}{n^{1+\varepsilon}}
\end{equation*}
for some constant $K>0$. Hence the series $\displaystyle\sum_{n\ge 0}{\frac{\varphi_{d_1,\varepsilon}(d_2^{d_1 n+1}\ell_{d_2}(d_2^{n d_1 d_2}))}{d_2^{d_2 n}}}$ converges. Secondly, we get 
\begin{equation*}
    \frac{\psi(5Cd_2^{n(d_1+1)d_2+1})}{d_2^{d_1 n}\ell_{d_2}(rd_2^{n d_1 d_2-1})}=O\left(\left ((5Cd_2)^{(5Cd_2)^{\frac{d_1}{d_1+1}}}\left (\frac{e}{r}\right )^{\frac{r}{d_2^2}}d_2^{(5Cd_2)^{\frac{d_1}{d_1+1}}n(d_1+1)d_2-\frac{rd_1 n}{d_2}+\frac{r}{d_2^2}}\right )^{d_2^{n d_1 d_2}}\right)
\end{equation*}
for every $C>0$. Taking $C$ small enough, the quantity $(5Cd_{2})^{\frac{d_{1}}{d_1+1}}d_2(d_1+1)-\frac{rd_1}{d_2}$ is negative and the series $\sum{\frac{\psi(5C d_2^{n(d_1+1)d_2+1})}{d_2^{d_1 n}\ell_{d_2}(rd_2^{n d_1 d_2-1})}}$ converges. By Theorem~\ref{thm:TilingSufficientConditionQuantitative}, we get a $(\varphi_{d_1,\varepsilon},\psi)$-integrable orbit equivalence coupling from $\juggler{r}{\Z^{d_1}}$ to $\Z^{d_2}$.
\end{proof}

Next, we consider the case of iterated lampjugglers over free abelian groups. Using the notion of composition of couplings (see Theorem~\ref{thm:CompositionCouplings}), we deduce the following.

\begin{corollary}\label{cor:couplingsbetweeniteratedshufflers}
Let $m,n\ge 0$ be natural integers such that $m>n$. Let $d_1,d_2,r\ge 1$. Then there exists an orbit equivalence coupling from $\jugglern{m}{r}{\Z^{d_{1}}}$ to $\jugglern{n}{r}{\Z^{d_{2}}}$, which is $(\varphi_{m-n,d_1,\varepsilon}(x),\mathrm{L}^{<\infty})$-integrable for every $\varepsilon>0$, where 
\begin{equation*}
    \varphi_{i,d_{1},\varepsilon}(x)\defeq \frac{ \ln^{\circ i}(x)^{\frac{1}{d_1}}} {\left(\ln^{\circ (i+1)}(x)\right)^{1+\frac{1}{d_{1}}+\varepsilon}}.
\end{equation*}
\end{corollary}

One remark is in order before the proof.

\begin{remark}\label{rem:couplingsbetweeniteratedshufflers}
Note that we have
\begin{equation}\label{eq:profileComposedInverseProfile}
    \prof{\jugglern{m}{r}{\Z^{d_{1}}}}\circ\prof{\jugglern{n}{r}{\Z^{d_{2}}}}^{-1}(x)\simeq\left(\frac{\ln^{\circ (m-n)}(x)}{\ln^{\circ (m-n+1)}(x)}\right)^{\frac{1}{d_{1}}}
\end{equation}
for every natural integers $m,n,d_{1},d_{2}$ satisfying $m>n$ and $d_{1},d_{2}\ge 1$. Indeed, if we denote by $f$ the inverse of the isoperimetric profile of $\jugglern{n}{r}{\Z^{d_{2}}}$, we get
\begin{equation*}
    \ln^{\circ (n+1)}(f(x))\sim d_{2}\ln{x}
\end{equation*}
using the equality $\prof{\jugglern{n}{r}{\Z^{d_{2}}}}(f(x))=x$. This implies
\begin{align*}
    \left(\prof{\jugglern{m}{r}{\Z^{d_{1}}}}\circ\prof{\jugglern{n}{r}{\Z^{d_{2}}}}^{-1}(x)\right)^{d_{1}}&=\frac{\ln^{\circ m}{(f(x))}}{\ln^{\circ (m+1)}{(f(x))}} \\
    &=\frac{\ln^{\circ (m-n-1)}({\ln^{\circ (n+1)}(f(x)))}}{\ln^{\circ (m-n)}(\ln^{\circ (n+1)}(f(x)))}\\
    &\simeq\frac{\ln^{\circ (m-n)}(x)}{\ln^{\circ (m-n+1)}(x)}
\end{align*}
and~\eqref{eq:profileComposedInverseProfile} follows.

\smallskip

In particular, Corollary~\ref{cor:couplingsbetweeniteratedshufflers} implies that, for every $\varepsilon>0$, there is a $\left(\varphi,\ld^{<\infty}\right)$-integrable orbit equivalence coupling from $\jugglern{m}{r}{\Z^{d_{1}}}$ to $\jugglern{n}{r}{\Z^{d_{2}}}$, with
\begin{equation*}
    \varphi(x)=\frac{\prof{\jugglern{m}{r}{\Z^{d_{1}}}}\circ\prof{\jugglern{n}{r}{\Z^{d_{2}}}}^{-1}(x)}{\left (\ln^{\circ (m-n+1)}{(x)}\right )^{1+\varepsilon}}.
\end{equation*}
\end{remark}

\begin{proof}[Proof of Corollary~\ref{cor:couplingsbetweeniteratedshufflers}]
\sloppy From Theorem~\ref{thm:couplingsZdandshufZk}, there is an orbit equivalence coupling from $\juggler{r}{\Z^{d_{2}}}$ to $\Z^{d_{2}}$ which is $(\varphi_{1,d_{2},\varepsilon},\ld^{<\infty})$-integrable. Applying $i\ge 1$ times Theorem~\ref{thm:stabilityofcouplings+quantificationLampjugglers}, it follows that there is an orbit equivalence coupling from $\jugglern{(i+1)}{r}{\Z^{d_{2}}}$ to $\jugglern{i}{r}{\Z^{d_{2}}}$, which is $(\varphi_{1,d_{2},\varepsilon},\ld^{<\infty})$-integrable. In particular, we have a coupling from $\jugglern{(m-1)}{r}{\Z^{d_{2}}}$ to $\jugglern{(m-2)}{r}{\Z^{d_{2}}}$, a coupling from $\jugglern{(m-2)}{r}{\Z^{d_{2}}}$ to $\jugglern{(m-3)}{r}{\Z^{d_{2}}}$, $\ldots$, and a coupling from $\jugglern{(n+1)}{r}{\Z^{d_{2}}}$ to $\jugglern{(n)}{r}{\Z^{d_{2}}}$, which are all $(\varphi_{1,d_{2},\varepsilon},\ld^{<\infty})$-integrable. By Theorem~\ref{thm:CompositionCouplings}, composing these successive couplings yields a coupling from $\jugglern{(m-1)}{r}{\Z^{d_{2}}}$ to $\jugglern{n}{r}{\Z^{d_{2}}}$, which is $(\varphi_{1,d_{2},\varepsilon}^{\circ (m-n-1)},\ld^{<\infty})$-integrable. Using the asymptotic equivalence 
\begin{equation*}
    \varphi_{1,d_{2},\varepsilon}^{\circ (m-n-1)}(x)\simeq\varphi_{m-n-1,d_{2},\varepsilon}(x)
\end{equation*}
and the second item of Remark~\ref{rm:checkongenerators}, this coupling is $(\varphi_{m-n-1,d_{2},\varepsilon},\ld^{<\infty})$-integrable.

\smallskip

By the same techniques, there is an orbit equivalence coupling from $\jugglern{m}{r}{\Z^{d_{1}}}$ to $\jugglern{(m-1)}{r}{\Z^{d_{2}}}$, which is $(\varphi_{1,d_{1},\varepsilon},\ld^{\infty})$-integrable, and composing this coupling with the one of the previous paragraph, we get a $(\varphi_{m-n,d_{1},\varepsilon},\ld^{<\infty})$-integrable orbit equivalence from $\jugglern{m}{r}{\Z^{d_{1}}}$ to $\jugglern{n}{r}{\Z^{d_{2}}}$. This proves the corollary.
\end{proof}

Hence, we can also prove that:

\begin{corollary}\label{cor:optimalityIteratedLampshuffler}
\sloppy Let $n,m,d_{1},d_{2}\ge 1$ be two integers such that $m>n$. There exists a $\left(\left(\frac{\ln^{\circ (m-n)}(x)}{\ln^{\circ (m-n+1)}(x)}\right)^{p},\mathrm{L}^{<\infty}\right)$-integrable orbit equivalence coupling from $\jugglern{m}{r}{\Z^{d_{1}}}$ to $\jugglern{n}{r}{\Z^{d_{2}}}$ if and only if $p<\frac{1}{d_{1}}$. 
\end{corollary}

Once again, before the proof, let us recall that $\prof{\jugglern{m}{r}{\Z^{d_{1}}}}\circ\prof{\jugglern{n}{r}{\Z^{d_{2}}}}^{-1}(x)$ is asymptotically equivalent to $\left (\frac{\ln^{\circ (m-n)}(x)}{\ln^{\circ (m-n+1)}(x)}\right )^{\frac{1}{d_1}}$ (see Remark~\ref{rem:couplingsbetweeniteratedshufflers}).

\begin{proof}
Assume first that such an orbit equivalence exists. Then Theorem~\ref{thm:dklmt} implies that 
\begin{equation}\label{eq:ComparisonProfileIterated}
    \left (\frac{\ln^{\circ (m-n)}(\prof{\jugglern{n}{r}{\Z^{d_2}}}(x))}{\ln^{\circ (m-n+1)}(\prof{\jugglern{n}{r}{\Z^{d_2}}}(x))}\right )^p\preccurlyeq \prof{\jugglern{m}{r}{\Z^{d_1}}}(x).
\end{equation}
From \cite[Theorem~C]{cordum25}, we know that, for any integers $i,j\ge 1$, the isoperimetric profile of $\jugglern{j}{r}{\Z^{i}}$ is $\simeq \left(\frac{\ln^{\circ j}(x)}{\ln^{\circ (j+1)}(x)}\right)^{\frac{1}{i}}$, so that~\eqref{eq:ComparisonProfileIterated} implies 
\begin{equation*}
    \left(\frac{\ln^{\circ m}(x)}{\ln^{\circ (m+1)}(x)}\right)^{p} \preccurlyeq \left (\frac{\ln^{\circ m}(x)}{\ln^{\circ (m+1)}(x)}\right)^{\frac{1}{d_1}}.
\end{equation*}
This inequality now forces $p\le \frac{1}{d_1}$, and once again the fact that there is no $\left(\left(\frac{\ln^{\circ (m-n)}(x)}{\ln^{\circ (m-n+1)}(x)}\right)^{\frac{1}{d_1}}, \mathrm{L}^{0}\right)$-orbit equivalence coupling follows from Remark~\ref{rem:couplingsbetweeniteratedshufflers} and Theorem~\ref{thm:threshold}.

\smallskip

Conversely, assume that $p<\frac{1}{d_1}$. From Corollary~\ref{cor:couplingsbetweeniteratedshufflers}, we have an orbit equivalence coupling from $\jugglern{m}{r}{\Z^{d_{1}}}$ to $\jugglern{n}{r}{\Z^{d_{2}}}$ which is $(\varphi_{m-n,d_1,\varepsilon},\mathrm{L}^{<\infty})$-integrable for all $\varepsilon>0$. Since $p<\frac{1}{d_1}$, we have
\begin{equation*}
    \left(\frac{\ln^{\circ (m-n)}(x)}{\ln^{\circ (m-n+1)}(x)}\right )^{p}=O(\varphi_{m-n,d_{1},\varepsilon}(x))
\end{equation*}
and it follows that our coupling is also $\left (\left(\frac{\ln^{\circ (m-n)}(x)}{\ln^{\circ (m-n+1)}(x)}\right )^{p},\mathrm{L}^{<\infty}\right )$-integrable, as claimed. 
\end{proof}

About iterated lampjugglers with different number of iterations, let us point out the following result when the base groups are finitely generated and have slow profiles, which does not use F\o lner tiling sequences but our work on isoperimetric profiles of such groups~\cite{cordum25}.

\begin{corollary}\label{cor:boundOfIntegrability}
\sloppy Let $n,m\ge 1$ be two integers such that $m>n$. Let $H$ be a finitely generated amenable group with isoperimetric profile $\prof{H}(x)\simeq \left (\ln^{\circ k}(x)\right )^{\alpha}$, for some integer $k>0$ and $\alpha>0$. If there exists a $\left (\left(\ln^{\circ (m-n)}(x)\right )^{p},\mathrm{L}^0\right )$-integrable orbit equivalence coupling from $\jugglern{m}{r}{H}$ to $\jugglern{n}{r}{H}$, then $p<\alpha$. 
\end{corollary}

\begin{proof}
We immediately get $p\le\alpha$ by Theorem~\ref{thm:dklmt}. With similar techniques as in Remark~\ref{rem:couplingsbetweeniteratedshufflers}, we show that $\prof{\jugglern{m}{r}{H}}\circ\prof{\jugglern{n}{r}{H}}^{-1}$ is asymptotically equivalent to $\left (\ln^{\circ (m-n)}(x)\right)^{\alpha}$, so that Theorem~\ref{thm:threshold} implies $p\neq\alpha$.  
\end{proof}

\begin{remark}\label{rem:FolnerWreathProduct}
Following our methods, we can deduce much more examples of quantitatively optimal couplings. For instance, we get an optimal coupling between $\Z$ and $\juggler{r}{F\wr\Z}$ (with a non-trivial finite group $F$) by composing two couplings:
\begin{itemize}
    \item an optimal coupling between $\Z$ and $\juggler{r}{\Z}$;
    \item an optimal coupling between $\juggler{r}{\Z}$ and $\juggler{r}{F\wr\Z}$, coming from our stability result (Theorem~\ref{thm:stabilityofcouplings+quantificationLampjugglers}) and a coupling between $\Z$ and $F\wr\Z$, coming from~\cite[Proposition~6.20]{DKLMT22}.
\end{itemize}
\end{remark}

\subsubsection{Lampcloners}

Let us finally apply our construction of F\o lner tiling sequences for couplings between $\cloner{\Z^{d_1}}$ and $\Z^{d_2}$.

\begin{theorem}\label{thm:couplingsZdandclonerZk}
Let $d_1,d_2\ge 1$ be positive integers and let $\field$ be a finite field. There exists an orbit equivalence coupling from $\cloner{\Z^{d_1}}$ to $\Z^{d_2}$, which is $(\varphi_{d_1,\varepsilon},\psi)$-integrable for every $\varepsilon>0$, where
\begin{equation*}
    \varphi_{d_1,\varepsilon}(x)\defeq\frac{\ln(x)^{\frac{1}{2d_1}}}{\ln(\ln(x))^{1+\varepsilon}}\; \text{and} \; \psi(x)=x^{x^{\frac{d_1}{3d_1+1}}}.
\end{equation*}
\end{theorem}

\begin{remark}
Theorem~\ref{thm:couplingsZdandshufZk} tells us that, for every $\varepsilon>0$, there is a $\left(\frac{j(x)}{\ln(\ln(x))^{1+\varepsilon}},\ld^{<\infty}\right)$-integrable orbit equivalence coupling from $\cloner{\Z^{d_1}}$ to $\Z^{d_2}$, where $j(x)\defeq \ln(x)^{\frac{1}{2d_1}}$ is a lower bound for $\prof{\cloner{\Z^{d_1}}}(x)$ that we find in~\cite[Theorem~G]{cordum25} (recall that we do not have precise estimates of the isoperimetric profile in this case).
\end{remark}

\begin{proof}[Proof of Theorem~\ref{thm:couplingsZdandclonerZk}]
Endow $\Z^{d_1}$ (resp. $\Z^{d_2}$) with its standard generating set $S_{\Z^{d_1}}$ (resp. $S_{\Z^{d_2}}$). Using Theorem~\ref{thm:folnertilingcloner}, the F\o lner tiling sequence 
\begin{equation*}
    (F_{n})_{n\in\N}\defeq \big(\lbrace 0,\dots,d_{2}^{n d_{2}}-1\rbrace^{d_1}\big)_{n\in\N}
\end{equation*}
of $\Z^{d_1}$ provides a F\o lner tiling sequence $(L_{n})_{n\in\N}$ of $\cloner{\Z^{d_1}}$ satisfying
\begin{itemize}
        \item $|L_{n}|=\Lambda_{\cloner{\Z^{d_1}}}(d_2^{n d_1 d_2})\cdot d_2^{n d_1 d_2}$, for all $n\in\N$;
        \item there exists a constant $C>0$ such that $\diam{L_{n}}\le C d_2^{n (3d_1+1) d_2+1}$ for all $n\in\N$;
        \item $\frac{|L_{n}s\setminus L_{n}|}{|L_{n}|}\le\frac{1}{d_2^{n d_2}}$, for all $s\in S_{\cloner{\Z^{d_1}}}$.
\end{itemize}
Let us recall that the lamp growth sequence of $\cloner{\Z^{d_1}}$ is given by
\begin{equation*}
    \forall n\ge 0,\; \Lambda_{\cloner{\Z^{d_1}}}(n)=\prod_{i=1}^{n}{(q^n-q^{n-i})}
\end{equation*}
where $q\defeq |\field|$. Given a positive integer $n$, let us define, for every $i\in\lbrace 1,2,\ldots d\rbrace$, the integer
\begin{equation*}
    \ell_{i}(n)\defeq\prod_{\substack{j\in\N \\ j d_2+i\le n}}{(q^n-q^{n-(j d_2+1)})}.
\end{equation*}
These quantities satisfy
\begin{equation*}
    \ell_{1}(n)\ell_{2}(n)\ldots \ell_{d_2}(n)=\Lambda_{\cloner{\Z^{d_1}}}(n),
\end{equation*}
\begin{equation*}
    \ell_{1}(n d_2)\le \ell_{2}(n d_2)\le\ldots\le \ell_{d_2}(n d_2)\le \ell_{1}((n+1) d_2)
\end{equation*}
and
\begin{equation*}
    \ell_{d_2}(n d_2)=\prod_{j=0}^{n-1}{(q^{n d_2}-q^{n d_2-(j+1) d_2})}=q^{n^2 d_2}\prod_{j=0}^{n-1}{(1-q^{-(j+1) d_2})}\sim C'q^{n^2 d_2}
\end{equation*}
where $C'$ is the limit of the limit of the convergent product $\prod_{j=0}^{\infty}{(1-q^{-(j+1) d_2})}$. Now, we find a F\o lner tiling sequence $(G_{n})_{n\in\N}$ of $\Z^{d_2}$ satisfying $|L_{n}|=|G_{n}|$, setting
\begin{equation*}
    G_{n}\defeq\prod_{1\le i\le d_2}{\lbrace 0,1,\dots,d_2^{n d_1}\ell_{i}(d_2^{n d_1 d_2})\rbrace}.
\end{equation*}
It satisfies
\begin{itemize}
    \item for all $n\in\N$, $\diam{G_{n}}\le d_2^{n d_1}(\ell_{1}(d_2^{n d_1 d_2}) +\dots +\ell_{d_2}(d_2^{n d_1 d_2}))\le d_2^{n d_1+1}\ell_{d_2}(d_2^{n d_1 d_2})$;
    \item for all $s\in S_{\Z^{d_2}}$, $\frac{|(s+G_{n})\setminus G_{n}|}{|G_{n}|}\le\frac{1}{d_2^{n d_1}\min{\left (\ell_{1}(d_2^{n d_1 d_2}),\ldots ,\ell_{d_2}(d_2^{n d_1 d_2})\right )}}\le\frac{1}{d_2^{n d_1}\ell_{d_2}(d_2^{n d_1 d_2-1})}$.
\end{itemize}

\noindent Observe that we have
\begin{equation*}
    \ln{(d_2^{n d_1+1}\ell_{d_2}(d_2^{n d_1 d_2}))}\sim d_2^{2n d_1 d_2 -1}\ln{q}
\end{equation*}
and
\begin{equation*}
    \ln(\ln(d_2^{n d_1+1}\ell_{d_2}(rd_2^{n d_1 d_2})))\sim 2n d_1 d_2\ln(d_2),
\end{equation*}
which in turn implies
\begin{equation*}
    \frac{\varphi_{k,\varepsilon}(d_2^{n d_1+1}\ell_{d_2}(d_2^{n d_1 d_2}))}{d_2^{n d_2}}\sim\frac{K}{n^{1+\varepsilon}}
\end{equation*}
for some constant $K>0$. Hence the series $\displaystyle\sum_{n\ge 0}{\frac{\varphi_{k,\varepsilon}(d_2^{n d_1+1}\ell_{d_2}(d_2^{n d_1 d_2}))}{d_2^{n d_2}}}$ converges. Secondly, we get 
\begin{equation*}
    \frac{\psi(C d_2^{n (3 d_1+1) d_2+1})}{d_2^{n d_1}\ell_{d_2}(rd_2^{n d_1 d_2-1})}=O\left(\left((Cd_2)^{(Cd_2)^{\frac{d_1}{3d_1+1}}}d_2^{(Cd_2)^{\frac{d_1}{3d_1+1}}n (3 d_1+1) d_2}q^{-d_2^{n d_1 d_2}}\right)^{d_2^{n d_1 d_2}}\right)
\end{equation*}
so the series $\sum{\frac{\psi(C d_2^{n (3 d_1+1) d_2+1})}{d_2^{n d_1}\ell_{d_2}(d_2^{n d_1 d_2-1})}}$ converges. By Theorem~\ref{thm:TilingSufficientConditionQuantitative}, we get a $(\varphi_{d_{1},\varepsilon},\psi)$-integrable orbit equivalence coupling from $\cloner{\Z^{d_{1}}}$ to $\Z^{d_{2}}$.
\end{proof}

Hence, analogously to Corollary~\ref{cor:optimalityIteratedLampshuffler}, we can deduce the following.

\begin{corollary}\label{cor:couplingsbetweeniteratedjugglers}
Let $\field$ be a finite field. Let $m,n\ge 0$ be natural integers such that $m>n$. Let $d_1,d_2\ge 1$ be positive integers and let $\field$ be a finite field.
\begin{itemize}
    \item For every $p<\frac{1}{2d_1}$, there exists an orbit equivalence coupling from $\clonern{m}{\Z^{d_1}}$ to $\clonern{n}{\Z^{d_2}}$, which is $((\ln^{\circ (m-n)})^{p},\mathrm{L}^{<\infty})$-integrable.
    \item If there exists an orbit equivalence coupling from $\clonern{m}{\Z^{d_1}}$ to $\clonern{n}{\Z^{d_2}}$, which is $((\ln^{\circ (m-n)})^{p},\mathrm{L}^{<\infty})$-integrable, then $p<\frac{1}{d_1}$.
\end{itemize}
\end{corollary}

\begin{proof}
    We follow the same ideas as in the proofs of Corollaries~\ref{cor:couplingsbetweeniteratedshufflers} and~\ref{cor:optimalityIteratedLampshuffler}. The only difference is the fact that we do not have a precise estimate on $\prof{\clonern{m}{\Z^{d_{1}}}}\circ\prof{\clonern{n}{\Z^{d_{2}}}}^{-1}(x)$, but we know that it is dominated by $(\ln^{\circ (m-n)}x)^{\frac{1}{d_1}}$. This shows that, in the second point of the statement, $p$ must be less than or equal to $\frac{1}{d_1}$. To prove that $p\neq\frac{1}{d_1}$, we apply a slight improvement of Theorem~\ref{thm:threshold}, using exactly the strategy of the first-named author in his proof~\cite{Corr24}. Indeed, if $p$ were equal to $\frac{1}{d_1}$, then there would exist a non-decreasing map $\psi\colon\R_+\to\R_+$ such that
    \begin{itemize}
        \item $t\mapsto \frac{t}{\psi(t)}$ is non-decreasing;
        \item there exists an orbit equivalence coupling from $\clonern{m}{\Z^{d_1}}$ to $\clonern{n}{\Z^{d_2}}$, which is $(\psi,\mathrm{L}^{<\infty})$-integrable;
        \item $\psi(x)\not\preccurlyeq (\ln^{\circ (m-n)}x)^{\frac{1}{d_1}}$.
    \end{itemize}
    The last point implies that $\psi(x)$ is not dominated by $\prof{\clonern{m}{\Z^{d_{1}}}}\circ\prof{\clonern{n}{\Z^{d_{2}}}}^{-1}(x)$, this is in contradiction with Theorem~\ref{thm:dklmt}.
\end{proof}

\bigskip
\section{Comments and questions}\label{sec:commentsandquestions}

In this final section, we record some questions that naturally arise from what has been done in the article.

\medskip

One of our main results is the construction of an explicit orbit equivalence coupling between the lampjugglers $\juggler{r}{H}$ and $\juggler{r}{K}$, provided a coupling between $H$ and $K$. Additionally, our construction preserves quantification. Hence:

\begin{question}
Let $r,s\ge 1$, $r\neq s$. How can one construct an orbit equivalence coupling between $\juggler{s}{H}$ and $\juggler{r}{K}$ from a coupling between $H$ and $K$?
\end{question}

As pointed out in Remark~\ref{rem:LimitationsTechniques}, there are reasons to think that a $(\varphi,\psi)$-integrable orbit equivalence between $H$ and $K$ does not provide the same for $\juggler{s}{H}$ and $\juggler{r}{K}$, if $r\neq s$. This kind of stability result could also gives stability of bi-Lipschitz equivalence, which would contradicts the main results in~\cite{GT24a}. Replacing orbit equivalence by measure equivalence would maybe provide a better framework when considering \enquote{similar}, but different, halo products, for instance $\juggler{s}{H}$ and $\juggler{r}{K}$ when $r\neq s$, or $\cloner{H}$ and $\mathsf{Cloner}_{\mathbb{l}}(K)$ when $|\field|\neq |\mathbb{l}|$.

\smallskip

Beyond the quasi-isometric classification of lampshufflers established in~\cite{GT24a}, another important direction of the latter is to compare geometrically different classes of halo products to rule out, or to establish, the existence of a quasi-isometry (resp. quasi-isometric/coarse/regular embedding) between, for instance, a lampshuffler and a lamplighter, or a lampdesigner and a lamplighter. Analoguous questions are also relevant from the measured point of view. For example, note that the isoperimetric profiles of $\shuf{\Z}$ and $\Z/2\Z \wr\Z$ are asymptotically equivalent up to a logarithmic factor, so we may wonder:

\begin{question}
Are $\shuf{\Z}$ and $\Z/2\Z \wr\Z$ $\ld^{p}$-orbit equivalent for $p\in \left]0,1\right[$?
\end{question}

Moreover, there are pairs of amenable groups with asymptotically equivalent isoperimetric profiles, such as
\begin{itemize}
    \item $\shuf{\Z}$ and $\Z\wr\Z$;
    \item $\shuf{H}$ and $\Z/2\Z\wr H$, with $H$ such that Assumption~$(\star)$ holds and $\prof{H}\left(\frac{\ln(x)}{\ln(\ln(x))}\right) \simeq \prof{H}(\ln(x))$,
\end{itemize}
and thus, we may wonder whether these pairs of groups are $\ld^{p}$ orbit equivalent for every $p>0$.

\smallskip

More generally, even without quantifications, it would be interesting to have an explicit description of actions from a wreath product and a lampshuffler on a common probability space sharing the same orbits.

\begin{question}
What could be an explicit orbit equivalence coupling between a lampshuffler and a wreath product?
\end{question}

Alternatively, another useful technique to produce orbit equivalences is the one of F\o lner tiling sequences, recalled and implemented in Section~\ref{sec:folnerTiling} in the case of lampshufflers. More precisely, Theorem~\ref{thm:folnertilingshuffler} provides an explicit description of a F\o lner tiling sequence for $\juggler{r}{H}$ from such a sequence for $H$. However, this technique remains hard to apply for the comparison between lamplighters and lampjugglers, and we did not manage to find F\o lner tiling sequences of both groups whose tiles have same cardinalities. Also for lampcloners, we have a explicit description of F\o lner tiling sequences (Theorem~\ref{thm:folnertilingcloner}) which does not enable us to get a coupling between a lampcloner and a lamplighter, or between a lampcloner and a lampshuffler.

\smallskip

We end this section by dealing with families of groups similar to halo products: Houghton groups $H_n$, $n\geq 2$. Note that $H_2$ is exactly the lampshuffler $\shuf{\Z}$.

\begin{question}
    For any $n\neq m$, what could be an explicit orbit equivalence coupling between $H_n$ and $H_m$? Which quantitative forms of orbit equivalence can we obtain between such groups?
\end{question}

\printbibliography

{\bigskip
		\footnotesize
		
		\noindent C.~Correia, \textsc{Université Paris Cité, Institut de Mathématiques de Jussieu-Paris Rive Gauche, 75013 Paris, France}\par\nopagebreak\noindent
		\textit{E-mail address: }\texttt{corentin.correia@imj-prg.fr}}
{\bigskip
		\footnotesize
		
		\noindent V.~Dumoncel, \textsc{Université Paris Cité, Institut de Mathématiques de Jussieu-Paris Rive Gauche, 75013 Paris, France}\par\nopagebreak\noindent
		\textit{E-mail address: }\texttt{vincent.dumoncel@imj-prg.fr}}

\end{titlepage}
\end{document}